\documentclass{amsart}
\usepackage{amsopn}
\usepackage{amssymb, amscd, amsmath}
\usepackage{array}
\usepackage{xcolor} 
\usepackage{hyperref}
\usepackage[shortlabels]{enumitem}
\usepackage{tikz-cd}

\newcommand{\nc}{\newcommand}

\DeclareRobustCommand{\SkipTocEntry}[4]{}

\newcounter{enum_counter}

\nc{\NS}{N_\beta}
\nc{\NbS}{N_{\beta_S}}
\nc{\bS}{{\beta^+_S}}
\nc{\gS}{g^S}
\nc{\gB}{g^B}
\nc{\gP}{g^P}
\nc{\gF}{g^\ve}
\nc{\divgB}{\divg_{\!B}}
\nc{\divgP}{\divg_{\!P}}
\nc{\vbMv}{\log (v_\beta v)}

\nc{\R}{\mathsf{R}}
\nc{\nablasym}{\nabla^{\rm sym}}
\nc{\gsym}{g_{\sf sym}}
\nc{\ev}{{i}}
 \nc{\mcv}{{\operatorname{H}}} 
  \nc{\mcvl}{{\operatorname{H}_L}} 
 \nc{\EL}{{\operatorname{E}}}
 \nc{\UL}{{\operatorname{U}}}
 \nc{\VL}{{\operatorname{V}}}
\nc{\HK}{H}
 \nc{\UK}{U}
\nc{\AK}{A}
\nc{\KKi}{K}
\nc{\VK}{V}
\nc{\EK}{E}
\nc{\Deltab}{\bar \Delta}
\nc{\Nw}{{N_0}}
\nc{\fx}{{n_0}}
\nc{\ggouni}{\ggo_{\rm uni}}

\nc{\Rca}{\mathcal{R}}

\nc{\Guni}{\G_{\rm uni}}
\nc{\Ubg}{\bar U^*} 
\nc{\Einstein}{}

\nc{\nsc}{{negative sca\-lar curvature}}
\nc{\psc}{{positive scalar curvature}}

\renewcommand{\gg}{{\boldsymbol \g}}
\nc{\bkgg}{{\boldsymbol \bkg}}
\nc{\betabg}{{\boldsymbol \betab}}
\nc{\qMg}{{\boldsymbol \qM}}
\nc{\bvg}{{ \boldsymbol \bv}}
\nc{\bMg}{{ ^\ggo \boldsymbol \beta  }}
\nc{\tbMg}{{ ^\ggo  \tilde {\boldsymbol\beta}  }}
\nc{\ab}{{\boldsymbol a}}

\nc{\lla}{\la \la}
\nc{\rra}{\ra\ra}
\nc{\ipp}{\lla \cdot, \cdot \rra}

\nc{\g}{h}     
\nc{\hM}{{h}}  
\nc{\qM}{{q}}
\nc{\muM}{{\mu}}
\nc{\vbM}{{\log v_\beta}}  
\nc{\vbg}{\log v_\betab}   
\nc{\bM}{{^\ngo\beta}}      
\nc{\bv}{\beta}    
\nc{\bE}{\bv_E}
\nc{\betab}{{\bar {\beta}}}     
\nc{\bkg}{\bar \g}          
\nc{\vbbg}{{v_{\betab}}}       
\nc{\Qbb}{{\mathsf{Q}_\betab}}
\nc{\Bbeb}{{\mathsf{B}_\betab}}   
\nc{\Slbb}{\Sl_{\betab}}
\nc{\ve}{{\mathcal{V}}}      
\nc{\ZK}{Z^*}
\nc{\detbb}{{\det}_{\betab}}
\nc{\Kbb}{\mathsf{K}_\betab}
\nc{\lb}{[\cdot \,,\cdot]}

\nc{\oG}{{\overline \G}} \nc{\oK}{{\overline \K}}
\nc{\Beta}{{\beta}}
\nc{\Q}{\mathsf{Q}}
\nc{\vca}{\mathcal{V}}
\nc{\Hess}{\operatorname{Hess}}
\nc{\nablab}{\overline{\nabla}}
\nc{\nablaP}{\nabla^P}
\nc{\divb}{\overline{\divg}}

\nc{\detb}{{\det}_{\Beta}}
\nc{\Slb}{\Sl_{\Beta}}
\nc{\Slbm}{\Sl_{\Beta_\mg}^H}
\nc{\slbm}{\slg_{\Beta_\mg}^H(n)}
\nc{\ggl}{\mathfrak{g}}
\nc{\ppm}{\mathfrak{p}}
\nc{\GG}{G}
\nc{\Gm}{\Gl(\mg)}
\nc{\Gg}{\Gl(\ggl)}
\nc{\glgg}{{\mathfrak{gl}(\mathfrak{g})}}
\nc{\GH}{\Gl^{\!H}}
\nc{\GHm}{\Gl^{\!H} (\mg)}
\nc{\GHmk}{\Gl^H(\mg_k)}
\nc{\gHm}{{\mathfrak{gl}^H(\mathfrak{m})}}
\nc{\OHm}{\Or^H(\mg)}
\nc{\sogHm}{ {\mathfrak{so}^H(\mathfrak{m})} }
\nc{\Vm}{V(\mg)}
\nc{\Vg}{V(\ggl)}
\nc{\Vmi}{V(\mg_\infty)}
\nc{\Om}{\Or(\mg)}
\nc{\Og}{\Or(\ggl)}
\nc{\Symm}{{\rm Sym}(\mg)}
\nc{\Symg}{{\rm Sym}(\ggl)}
\nc{\gm}{\mathfrak{gl}(\mg)}
\nc{\som}{\sog(\mg)}
\nc{\sogg}{{\mathfrak{so}(\mathfrak{g})}}
\nc{\hml}{{\mu^{\ggl}}}
\nc{\ml}{{\mu^{\ggl}_{\mg}}}
\nc{\Ol}{{\mathcal{O}_{\ggl}}}
\nc{\OlH}{{\mathcal{O}^H_{\ggl}}}
\nc{\Vn}{V(n)}
\nc{\SymV}{{\rm Sym}(V)}
\nc{\mub}{{\bar \mu}}
\nc{\mumg}{{\mu_\mg}}
\nc{\Betam}{{\Beta_\mg}}
\nc{\Betag}{{\Beta_\ggo}}
\nc{\Betagp}{{\Beta_\ggo^+}}
\nc{\ii}{{\mathrm{i}}}
\nc{\Nrm}{{\mathrm{N}}}
\nc{\Srm}{{\mathrm{S}}}
\nc{\spec}{{\operatorname{spec}}}
\nc{\iR}{{\ii\RR}}

\nc{\Vs}{{V(\sg)}}
\nc{\Syms}{{\Sym(\sg)}}
\nc{\glgs}{{\glg(\sg)}}

\nc{\slgb}{{\slg_\Beta}}

\nc{\Vzerog}{V_{\Beta_\ggo^+}^{0}}
\nc{\Vnng}{V_{\Beta_\ggo^+}^{\geq 0}}
\nc{\Vnnssg}{U_{\Beta_\ggo^+}^{\geq 0}}
\nc{\Vzerossg}{U_{\Beta_\ggo^+}^{0}}

\nc{\musol}{{\mu_{\mathsf{sol}}}}

\nc{\Gl}{\mathsf{GL}} \nc{\Or}{\mathsf{O}}  \nc{\SO}{\mathsf{SO}}   \nc{\Sl}{\mathsf{SL}}  
\nc{\G}{\mathsf{G}} \nc{\K}{\mathsf{K}}    \nc{\Ll}{\mathsf{L}}   \nc{\Bb}{\mathsf{B}}
 \nc{\A}{\mathsf{A}}  \nc{\Hg}{\mathsf{H}}  \nc{\Ii}{\mathsf{I}}
 \nc{\N}{\mathsf{N}}   \nc{\SU}{\mathsf{SU}}
 \nc{\T}{\mathsf{T}} \nc{\Lsf}{\mathsf{L}}
\nc{\Qb}{\mathsf{Q}_\Beta} \nc{\Hb}{\mathsf{H}_\Beta} \nc{\Ub}{\mathsf{U}_\Beta} 
\nc{\Gb}{\mathsf{G}_\Beta} \nc{\Kb}{\mathsf{K}_\Beta} \nc{\Hh}{\mathsf{H}}
\nc{\PPP}{\mathsf{P}}  
\nc{\Ss}{\mathsf{S}}
\nc{\U}{\mathsf{U}}
\nc{\B}{\mathsf{B}}
\nc{\M}{\mathsf{M}} \nc{\F}{\mathsf{F}}

\nc{\Gs}{{\Gl(\sg)}}  \nc{\Os}{{\Or(\sg)}}
\nc{\gsol}{{g_{\mathsf{sol}}}}
\nc{\bgsol}{{\bar g_{\mathsf{sol}}}}

\nc{\GGs}{S}
\nc{\ggs}{\mathfrak{s}}

 \nc{\ggo}{\mathfrak{g}}
 \nc{\ggob}{\overline{\mathfrak{g}}}
\nc{\lamg}{\Lambda^2\ggo^*\otimes\ggo}
\nc{\gkp}{(\ggo=\kg\oplus\pg,\ip)} \nc{\ukh}{(\ug=\kg\oplus\hg,\ip)}
\nc{\tgkp}{(\tilde{\ggo}=\kg\oplus\pg,\ip)}

\nc{\fg}{\mathfrak{f}}  \nc{\vg}{\mathfrak{v}} \nc{\wg}{\mathfrak{w}} \nc{\zg}{\mathfrak{z}} \nc{\ngo}{\mathfrak{n}} \nc{\kg}{\mathfrak{k}} \nc{\mg}{\mathfrak{m}} \nc{\bg}{\mathfrak{b}}  \nc{\sog}{\mathfrak{so}} \nc{\sug}{\mathfrak{su}} \nc{\spg}{\mathfrak{sp}} \nc{\slg}{\mathfrak{sl}} \nc{\glg}{\mathfrak{gl}} \nc{\cg}{\mathfrak{c}} \nc{\rg}{\mathfrak{r}}  \nc{\hg}{\mathfrak{h}} \nc{\tgo}{\mathfrak{t}} \nc{\ug}{\mathfrak{u}} \nc{\dg}{\mathfrak{d}} \nc{\ag}{\mathfrak{a}} \nc{\pg}{\mathfrak{p}} \nc{\sg}{\mathfrak{s}} \nc{\affg}{\mathfrak{aff}} \nc{\qg}{\mathfrak{q}}
\nc{\Xg}{\mathfrak{X}} \nc{\lgo}{\mathfrak{l}} \nc{\tg}{\mathfrak{t} }  \nc{\eg}{\mathfrak{e}}

\nc{\pca}{\mathcal{P}} \nc{\nca}{\mathcal{N}} \nc{\lca}{\mathcal{L}} \nc{\oca}{\mathcal{O}} \nc{\mca}{\mathcal{M}} \nc{\tca}{\mathcal{T}} \nc{\aca}{\mathcal{A}} \nc{\cca}{\mathcal{C}} \nc{\gca}{\mathcal{G}} \nc{\sca}{\mathcal{S}} \nc{\hca}{\mathcal{H}} \nc{\bca}{\mathcal{B}} \nc{\dca}{\mathcal{D}} \nc{\fca}{\mathcal{F}} \nc{\Qca}{\mathcal{Q}} \nc{\Eca}{\mathcal{E}} \nc{\Kca}{\mathcal{K}}
\nc{\rca}{\mathcal{R}}  

\nc{\dd}{{\rm d}}  \nc{\ddt}{\tfrac{{\rm d}}{{\rm d}t}}        \nc{\dds}{\tfrac{{\rm d}}{{\rm d}s}} 
\nc{\ddtbig}{\frac{{\rm d}}{{\rm d}t}}      \nc{\dpar}{\tfrac{\partial}{\partial t}}    

\nc{\im}{\mathtt{i}} \renewcommand{\Re}{{\rm Re}}   

\nc{\RR}{{\mathbb R}} \nc{\HH}{{\mathbb H}} \nc{\CC}{{\mathbb C}} \nc{\ZZ}{{\mathbb Z}}
\nc{\FF}{{\mathbb F}} \nc{\NN}{{\mathbb N}} \nc{\QQ}{{\mathbb Q}} \nc{\PP}{{\mathbb P}}
\nc{\KK}{{\mathbb K}}

\nc{\vs}{\vspace{.2cm}} \nc{\vsp}{\vspace{1cm}} 
\nc{\ip}{{\langle \,\cdot \,,\cdot \,\rangle }}
 \nc{\la}{\langle} \nc{\ra}{\rangle} \nc{\unm}{\tfrac{1}{2}}
\nc{\unc}{\tfrac{1}{4}} \nc{\und}{\tfrac{1}{16}} \nc{\no}{\vs\noindent}
\nc{\lam}{\Lambda^2(\RR^n)^*\otimes\RR^n} \nc{\tangz}{{\rm T}^{\rm Zar}}

\nc{\nor}{{\sf n}}  \nc{\mum}{/\!\!/} \nc{\kir}{/\!\!/\!\!/}
\nc{\Ri}{\tfrac{4\Ric_{\mu}}{||\mu||^2}} \nc{\ds}{\displaystyle}
\nc{\ben}{\begin{enumerate}} \nc{\een}{\end{enumerate}} \nc{\f}{\frac}
 \nc{\isn}{\tfrac{1}{||v||^2}}

\nc{\wt}{\widetilde}
\nc{\raw}{\rightarrow} \nc{\lraw}{\longrightarrow} \nc{\hqn}{\mathcal{H}_{q,n}}
\nc{\minimatrix}[4]{\left(\begin{smallmatrix} {#1} & {#2} \\ {#3} & {#4} \end{smallmatrix}\right)}
\nc{\twomatrix}[4]{\left[\begin{array}{cc} {#1} & {#2} \\ {#3} & {#4} \end{array} \right]}
\nc{\threematrix}[9]{\left[\begin{array}{ccc} {#1} & {#2} & {#3} \\ {#4} & {#5} & {#6}\\ {#7} & {#8} & {#9} \end{array} \right]}
\nc{\mut}{\tilde{\mu}} \nc{\mur}{{\mu_r}} \nc{\mutr}{{\tilde{\mu}_r}}

\nc{\alert}{\color{blue}}

\nc{\glgan}{\minimatrix{0}{0}{\star}{0}} \nc{\glgna}{\minimatrix{0}{\star}{0}{0}}  \nc{\glgnn}{\minimatrix{0}{0}{0}{\star}}  \nc{\glgaa}{\minimatrix{\star}{0}{0}{0}}
\nc{\Vaan}{{\left(\ag \wedge \ag\right)^* \otimes \ngo}} \nc{\Vann}{{\left(\ag \otimes \ngo \right)^* \otimes \ngo}} \nc{\Vnnn}{{\left(\ngo \wedge \ngo \right)^* \otimes \ngo}}

\nc{\ad}{\operatorname{ad}}  \nc{\Aut}{\operatorname{Aut}}   \nc{\Inn}{\operatorname{Inn}}   \nc{\Lie}{\operatorname{Lie}} \nc{\Ad}{\operatorname{Ad}} \nc{\Der}{\operatorname{Der}} \nc{\rad}{\operatorname{rad}} \nc{\kf}{\operatorname{B}}

\nc{\End}{\operatorname{End}} \nc{\rank}{\operatorname{rank}} \nc{\Ker}{\operatorname{Ker}} \nc{\tr}{\operatorname{tr}} \nc{\Sym}{\operatorname{Sym}} \nc{\diag}{\operatorname{diag}} \nc{\proy}{\operatorname{pr}} \nc{\Adj}{\operatorname{Adj}} \nc{\proj}{\operatorname{pr}} \nc{\Id}{{\operatorname{Id}}} \nc{\IdV}{{\Id_{\ve}}} \nc{\Span}{\operatorname{span}}
 \nc{\ran}{\operatorname{ran}} 

  \nc{\dif}{\operatorname{d}} \nc{\sen}{\operatorname{sen}} \nc{\grad}{\operatorname{grad}} \nc{\Order}{\operatorname{O}} 
  \nc{\divg}{\operatorname{div}}
	\nc{\divgb}{\overline{\operatorname{div}}}
  \nc{\Ric}{\operatorname{Ric}}

\nc{\Iso}{\operatorname{Iso}} \nc{\Diff}{\operatorname{Diff}} 
\nc{\Ricci}{\operatorname{Ric}}
\nc{\ric}{\operatorname{ric}} 
\nc{\Riem}{\operatorname{Rm}} \nc{\scal}{\operatorname{scal}} \nc{\scalm}{\operatorname{scal}^\star} \nc{\Riccim}{\operatorname{Ric}^{\star}} \nc{\tang}{\operatorname{T}} \nc{\vol}{\operatorname{vol}} \nc{\inj}{\operatorname{inj}}
\nc{\isog}{\mathfrak{iso}}
\nc{\ho}{\mathcal{H}}

\nc{\mm}{\operatorname{M}} \nc{\CH}{\operatorname{CH}} \nc{\Irr}{\operatorname{Irr}} \nc{\mcc}{\operatorname{mcc}} \nc{\Sb}{\mathcal{S}_\Beta} \nc{\mmm}{\operatorname{m}} 

\theoremstyle{plain}
\newtheorem{theorem}{Theorem}[section]
\newtheorem{proposition}[theorem]{Proposition}
\newtheorem{corollary}[theorem]{Corollary}
\newtheorem{lemma}[theorem]{Lemma}
\newtheorem{teointro}{Theorem}
\newtheorem{corintro}[teointro]{Corollary}

\newtheorem{assump}[teointro]{Assumption}
\newtheorem{propintro}[teointro]{Proposition}

\theoremstyle{definition}
\newtheorem{definition}[theorem]{Definition}

\newtheorem{assumption}[theorem]{Assumption}

\theoremstyle{remark}
\newtheorem{remark}[theorem]{Remark}

\newtheorem{example}[theorem]{Example}

\AtBeginDocument{%
   \def\MR#1{}
}

\begin{document}
\begin{titlepage}

\title{Non-compact Einstein manifolds with symmetry}

\author{Christoph B\"ohm}	
\address{University of M\"unster, Einsteinstra{\ss}e 62, 48149 M\"unster, Germany}
\email{cboehm@math.uni-muenster.de}
\author{Ramiro A.~ Lafuente} 
\address{School of Mathematics and Physics, The University of Queensland, St Lucia QLD 4072, Australia}
\email{r.lafuente@uq.edu.au}

\begin{abstract}
For 
Einstein manifolds with negative scalar curvature admitting an isometric action of a 
Lie group $\G$ with  compact, smooth orbit space, we show that the nilradical $\N$ of $\G$ acts polarly and that the $\N$-orbits can be extended to
 minimal Einstein submanifolds. As an application, we prove the Alekseevskii conjecture:
Any  homogeneous Einstein manifold with negative scalar curvature is diffeomorphic to a Euclidean space.
\end{abstract}


\end{titlepage}

\maketitle

\setcounter{tocdepth}{1}

\setcounter{page}{1}

\vspace{-.5cm}

\tableofcontents

\section{Introduction}


A Riemannian manifold $(M^n,g)$ is called Einstein if its Ricci tensor satisfies $\ric(g) = \lambda \, g$, for some $\lambda \in \RR$.
 In this article we study Einstein manifolds with \nsc{}, that is $\lambda <0$,  admitting an isometric action of a connected Lie group. Note that by a classical theorem of Bochner the underlying space 
 $M^n$ must be non-compact.
 
 Our first main result confirms the Alekseevskii Conjecture,
 a long-standing open problem formulated in 1975  by D.~V.~ Alekseevskii
 (see \cite{Alek75}, \cite[7.57]{Bss}):

\begin{teointro}[Alekseevskii Conjecture]\label{thm_alek}
Any homogeneous Einstein space with \nsc{} is diffeomorphic to a Euclidean space.
\end{teointro}

A connected Riemannian manifold is called \emph{homogeneous} if its isometry group acts transitively. 
Theorem \ref{thm_alek} was known in dimensions $n\leq 10$ \cite{Jns,Nkn1,AL16,Beri21}, with a few exceptions, and other partial results were obtained in \cite{Nkn2,alek,JblPet14,Jab15}. 
Since there exist  
cohomogeneity one Einstein manifolds with
\nsc{} with non-vanishing Betti numbers,  see e.g.~\cite{C.Ef,C.E, B.B,W-W, D-W, Bo1999},
 the homogeneity assumption  in Theorem \ref{thm_alek}  is essential. Note also,  that
  relaxing the Einstein assumption  in Theorem \ref{thm_alek} to negative Ricci curvature is not possible due to \cite{DttLt,DttLtMtl,Will17,Will20,LW20}. 

Concerning the algebraic structure of non-compact homogeneous Einstein  spaces, let us mention 
that by combining Theorem \ref{thm_alek} with \cite{BL18}, any such space must be isometric to an \emph{Einstein solvmanifold}, that is, it  admits a transitive  solvable Lie group of isometries. 
As a consequence, the deep structure theory 
initiated in the seminal work by J.~Heber \cite{Heb} in 1998, and further developed in  \cite{standard, Nik11, Jab15, GJ19} (among many others), now applies: see also \cite{cruzchica} and references therein. 
In particular, this reduces the classification of non-compact homogeneous Einstein spaces to that of \emph{nilsolitons} \cite{Heb,soliton}. By contrast,  despite several structure results concerning existence and non-existence, the classification of compact homogeneous Einstein spaces remains wide open. 

We turn to further consequences of Theorem \ref{thm_alek}. Recall that a Riemannian manifold $(M^n,g)$  is  an \emph{expanding Ricci soliton} if  $\ric(g) = \lambda g + \lca_X g$
for some  $\lambda <0$ and a smooth vector field $X$ on $M^n$. 
Ricci solitons give rise to Ricci flow solutions which evolve only by scaling and pull-back by diffeomorphisms. 
If the latter are automorphisms of a solvable Lie group  acting simply-transitively and isometrically,  $(M^n, g)$ is called a \emph{solvsoliton} \cite{solvsolitons}. These are also diffeomorphic to a Euclidean space, and applying  \cite[Thm.~1.1]{Jab15} and \cite{alek}, Theorem \ref{thm_alek} yields:

\begin{corintro}
Any homogeneous expanding Ricci soliton is isometric to a solvsoliton.
\end{corintro}

It follows from \cite{BL17} that any immortal homogeneous Ricci flow  subconverges to such a space after parabolic rescaling.

Another consequence is the   classification  
of homogeneous Riemannian manifolds with special holonomy. Indeed, 
Ricci flat homogeneous manifolds are flat \cite{AlkKml} and
homogeneous K\"ahler manifolds are classified \cite{DN88}, whereas in the quaternionic K\"ahler case Theorem \ref{thm_alek} implies:

\begin{corintro}
Any homogeneous quaternionic K\"ahler manifold is an Alekseevskii space or a Wolf space.
\end{corintro}

An \emph{Alekseevskii space} is a non-compact homogeneous quaternionic K\"ahler manifold admitting a transitive solvable  group of isometries. They were classified  in \cite{Alek75b} (see also \cite{Cor96}).
The \emph{Wolf spaces} \cite{Wolf65} are certain symmetric spaces exhausting all compact homogeneous quaternionic K\"ahler manifolds \cite{Alek68}.
It has been conjectured in \cite{LeBSal94} that the latter should hold even without  homogeneity.

Finally, combining Theorem \ref{thm_alek} with \cite[Thm.~1.13]{Jab15} we deduce that

\begin{corintro}
Any compact, locally homogeneous Einstein manifold with \nsc{} is locally symmetric.
\end{corintro}

We turn now to the proof of Theorem \ref{thm_alek} and  further main results.
Surprisingly, to prove Theorem \ref{thm_alek}
it is key to ignore the homogeneity assumption, since the (algebraic) Einstein
equation for homogeneous metrics has proved elusive over the years.
We consider instead 
 non-transitive isometric group actions with a compact orbit space. More precisely, we will be making the following 

\begin{assump}\label{as_G}
Let $(M^n,g)$ be a connected, complete Riemannian manifold, 
and $\G$   a connected Lie group acting on $(M^n,g)$ properly, isometrically, cocompactly
and with a single orbit type.
\end{assump}

By cocompact we mean of course that the orbit space $B^d=M^n/\G$ is compact. Having a single orbit type implies that all the orbits are principal, thus $B^d$ is a smooth manifold \cite{Pal61}. 


A first natural question is to  determine which groups can arise under Assumption \ref{as_G}. In this direction, our next result rules out most unimodular Lie groups even if we merely assume  negative Ricci curvature:

\begin{teointro}\label{thm_ricneg}
Let $(M^n,g)$ be a Riemannian manifold with $\ric(g) < 0$. Then, any unimodular Lie group $\G$ satisfying Assumption \ref{as_G}   must be non-compact semisimple.
\end{teointro}

Recall that a connected Lie group $\G$ is \emph{unimodular} if its left Haar measure is also right-invariant.
This is equivalent to the algebraic condition, that its Lie algebra $\ggo=T_e \G$ 
satisfies $\tr_{\ggo}( \ad X) = 0$ for all $X\in \ggo$. It is called non-unimodular otherwise.
Theorem \ref{thm_ricneg} should be compared  to \cite{Rong98}, from which the case of $\G$ abelian could be deduced.
 In  the homogeneous case,  Theorem \ref{thm_ricneg} is due to Dotti \cite{Dtt88}.

When $\G$ is semisimple, well-known properties of the Iwasawa decomposition of $\G$
imply that there is a non-unimodular solvable Lie subgroup $\A \N $ of $ \G$
 whose induced action on $(M^n,g)$ satisfies Assumption \ref{as_G}.
 Recall that for non-unimoldular $\G$,
 the \emph{nilradical} $\N$ of $\G$ (the 
maximal connected, nilpotent, normal Lie subgroup) 
has positive dimension.

The following is our main structure result on  non-compact Einstein manifolds with  symmetry:

\begin{teointro}\label{thm_rigidity}
Let $(M^n,g)$ be an Einstein manifold with \nsc{} and let $\G$ be a non-unimodular Lie group satisfying Assumption \ref{as_G}. Then, the induced action of the nilradical $\N$ of $\G$ on $M^n$ is polar and $M^n$ can be foliated into 
pairwise locally isometric, 
 minimal Einstein submanifolds $E^{\dim N+1}$. 
\end{teointro}

Alike in the famous Soul theorem for non-compact Riemannian manifolds with
non-negative sectional curvature, the \emph{Einstein leaves} $E$  are submanifolds
which inherit the same curvature condition as the ambient space (with
the same Einstein constant). Notice however that they are minimal but not
totally geodesic, as seen for instance in the transitive case of Einstein solvmanifolds. More precisely, the Einstein leaves 
are  $\N$-invariant, immersed, non-compact, locally homogeneous 
Einstein submanifolds of $M^n$.
If $M^n$ is simply-connected, then they are in fact embedded, 
equidistant Einstein solvmanifolds (Theorem \ref{thm_Einsteinsub}) and we have a diffeomorphism
$M^n \simeq E^{\dim N+1} \times P'$ with  $P=M^n/\N\simeq \RR \times P'$. This yields  immediately
topological obstructions: see Corollary \ref{cor_topobstr}. 
Finally, let us mention that if $\G$ is transitive and solvable, the Einstein foliation  is essentially Heber's rank one reduction \cite[Thm.~D]{Heb}.

The $\N$-orbits, hypersurfaces of the Einstein leaves, 
are  locally isometric to a fixed \emph{nilsoliton}, 
a left-invariant Ricci soliton on the universal cover of $\N$   (Corollary \ref{cor_nsoliton}).
The $\N$-action on $M^n$ being \emph{polar}  means that there is an immersed submanifold intersecting all $\N$-orbits orthogonally, or equivalently, the $\N$-\emph{horizontal  distribution}  is integrable,  see \cite[Thm.~A]{HLO} or \cite[Thm.~1.2]{GZ12}. 
This generalises J. Lauret's  famous result  \cite{standard} that Einstein solvmanifolds are standard. The polar condition has also appeared in the context of Ricci flat 4-manifolds with symmetries, see \cite{CP02,Lot20}.




Manifolds  satisfying the assumptions of Theorem \ref{thm_rigidity}
are    for instance given by Riemannian products of compact Einstein manifolds and Einstein solvmanifolds,
both with \nsc{}: see
\cite{Aub78}, \cite{Yau}, \cite{An06}, \cite{Ba12}, \cite{FP20} for compact examples,
and the survey \cite{cruzchica} and  references therein for Einstein solvmanifolds.

To the best of our knowledge, all previous results on Einstein manifolds with non-unimodular symmetry assume that the cohomogeneity $d=\dim M^n/\G$ is at most $1$. Recall, that for $d=0$ the Einstein equation
is algebraic, that for $d=1$ it is an ordinary differential equation, whereas for $d\geq 2$
it is an honest partial differential equation.
Thus, the main significance of Theorem \ref{thm_rigidity}
is that it allows for arbitrary cohomogeneity.
While generalisations of
Theorem \ref{thm_rigidity}  for compact,  non-smooth orbit spaces
will be treated in a forthcoming paper, we emphasize that for non-compact orbit spaces 
 such general rigidity results  are  not true,  not even in the
  cohomogeneity-one case:  see e.g.~\cite{BDGW2015, CDJL2020, Wink21}.

We  now state some consequences of Theorem \ref{thm_rigidity}:


\begin{corintro}\label{cor_noS1}
Let $(M^n,g)$, $\G$ be as in Theorem \ref{thm_rigidity}, with $\G$ acting freely and  $\dim \G = \dim\N + 1 \geq 2$.   Assume that $B^d=M^n/\G$ does not admit any smooth $S^1$-action. Then, the $\G$-orbits in $M^n$  
are Einstein solvmanifolds.
\end{corintro}


Even if $\N$ admits a nilsoliton metric, 
only one of its infinitely many possible one-dimensional (solvable) 
extensions $\G$ admits a  left-invariant  Einstein metric. Thus, Corollary \ref{cor_noS1}
provides further  obstructions not covered by Corollary \ref{cor_nsoliton}.

Simply-connected, compact spin-manifolds $B^{4k}$ with non-vanishing $\hat A$-genus
do not admit any  smooth $S^1$-actions  by  \cite{AH70}. 
For examples in dimension $6$ see  \cite{Pu95,DW17}  and in the presence of a non-trivial
fundamental group we refer to \cite{Yau77,BH81} for further topological restrictions:
e.g. compact hyperbolic manifolds provide examples  \cite{Borel83}.
Finally, recall that for a compact manifold without any $S^1$-action, the isometry group of any Riemannian metric is discrete.

A second  consequence of Theorem \ref{thm_rigidity} is 
that it allows  us to distinguish between the curvature conditions $\{\ric(g) < 0\}$ and $\{\ric(g)=-g\}$ 
among $\G$-invariant metrics:

\begin{corintro}\label{cor_negversusEinstein}
For any $k \geq 8$, there exists infinitely many $k$-dimensional, pairwise non-iso\-morphic Lie groups $\G$, such that for any compact manifold $B^{d}$, $d\geq 3$,
the manifold $M^n=\G \times B^{d}$ admits $\G$-invariant metrics with negative Ricci curvature, but no $\G$-invariant Einstein metric.
\end{corintro}

The Lie groups $\G$ are solvable and can even be chosen to 
admit  left-invariant metrics with negative sectional curvature \cite{Hei74}, but their 
codimension-one nilradicals do not admit nilsoliton metrics \cite{FC13}. In this respect, dimension $k=8$
is optimal
\cite{finding,Wll03}. If    $B^d$ admits no smooth $S^1$-action then $k \geq 3$ is enough for the statement in Corollary \ref{cor_negversusEinstein} to hold.
Notice also that despite $M^n = \G \times B^d$ being topologically a product,   
$\G$ does not have to act polarly on $(M^n,g)$ for an arbitrary  $\G$-invariant metric $g$.

\subsection{Proof outline of  Theorems \ref{thm_ricneg} and \ref{thm_rigidity}}

We endow the compact, smooth orbit space $B = M/\G$ with the quotient metric $\gB$, so that the quotient projection
\[
 \pi : (M,g) \to (B:=M/\G,\gB)\,
\] 
is a Riemannian submersion.
Assuming that $\G$ acts effectively, it follows that the nilradical $\N$ of $\G$  acts freely on $M$
(Lemma \ref{lem_Nfree}), yielding another smooth Riemannian submersion
\[
 \pi_P : (M,g) \to  (P:=M / \N,\gP)
\] 
whose fibres $(F,\gF)$ are the $\N$-orbits.
Note the space $P$ of $\N$-orbits might be non-compact. However, there is an induced isometric $\G/\N$-action on $P$, with compact orbit space $B$.

We first focus on Theorem \ref{thm_rigidity}. To show that $\N$ acts polarly,  we use the Einstein condition and O'Neill's curvature formulae for the Riemannian submersion $\pi_P$
(Theorem \ref{thm:Ric}), 
to construct a smooth, $\G/\N$-invariant vector field $Z$ on $P$ satisfying
\begin{equation}\label{eqn_divZintro}
		\divg_P Z \geq 0,
\end{equation}
where equality implies that O'Neil's $A$-tensor  vanishes. Since $Z$ is $\G/\N$-invariant and $B = P/(\G/\N)$ is compact, rigidity follows essentially from the divergence theorem (see Proposition \ref{prop_divgeq0}). 

When $\G$ is unimodular and $\N$ is abelian (in particular, when $\G= \N$ is itself abelian), $Z$ is  the gradient of the relative volume of the $\N$-orbits, cf.~ \cite{Rong98,NT18,Lot20}. (In this case one gets a contradiction, even  only assuming $\ric < 0$, see Theorem \ref{thm_rig_polarricn}.)

In general, the construction of $Z$ has three major ingredients. The first one involves an estimate for the Ricci curvature $\ric^\ve$ of the fibers, which are locally isometric to left-invariant metrics on $\N$. Using real geometric invariant theory and the Kirwan-Ness stratification of the space of brackets, Lauret \cite{standard} established the non-negativity of the $\beta$-weighted scalar curvature:
\begin{equation}\label{eqn_betascal}
		\sum_{i=1}^{\dim F}  \beta^+_i \ric^\ve_{ii} \,\, \geq \,\,  0,
\end{equation}
see Remark \ref{rmk_betascal}.
Here, $\ric^\ve_{ii} := \ric^\ve(U_i, U_i)$ for some carefully chosen vertical orthonormal frame $\{U_i\}$, and $(\beta^+_1, \ldots, \beta^+_{\dim \ngo})$ is  a vector of positive rational numbers naturally associated to the Lie algebra $\ngo$ (this is closely related to the \emph{eigenvalue type} of Einstein solvmanifolds \cite{Heb}). In order to exploit this estimate, we construct a smooth, $\G$-invariant function $\vbM:M \to \RR$, the (logarithmic) \emph{$\beta$-volume} of the $\N$-orbits: 
see Definition \ref{def_beta}. This  played a central role in the construction  of monotone quantities for homogeneous Ricci flows by the authors in \cite{BL17}. 
In this context, its first variation in a horizontal direction $X$ is given by the difference between the mean curvature and the \emph{$\beta$-weighted mean curvature} of the fibers, the latter being defined by replacing  $\ric^\ve$ in \eqref{eqn_betascal} by the shape operator in  the direction $X$.  
Using \eqref{eqn_betascal}, we can also estimate its Laplacian in terms of geometric data (Lemma \ref{lem_HessvbM}).

The second ingredient is the relative volume of the orbits. Since $\N$ is unimodular,  the mean curvature vector $N$ of the $\N$-orbits is given by $N =-\nabla \log v_\N$, where $v_\N := (\det g_{ij})^{1/2}$ is a function on $P$, and $g_{ij} = g(V_i,V_j)$ for some fixed frame $\{V_i\}$ of Killing fields in $\ngo$ (see Lemma \ref{lem_Ngrad}). The function $\log v_\N$ is  a natural candidate to yield a nice second-order PDE from which to get rigidity, as indicated by the abelian case. However, in order to apply global arguments on the compact manifold $B$, it is  crucial that $v_\N$ is $\G$-invariant. This is unfortunately not the case, if $\G$ is non-unimodular (see Lemma \ref{lem_Nw=0}).  To overcome this, we consider an \emph{equivariant, modified Helmholtz decomposition} for $N$ (Proposition \ref{prop_equivHelm}): 
\[
	N = -\nabla \log v + \Nw, \qquad \divg_P(v \Nw) = 0,
\]
where $v\in \cca^\infty(P)$ and $\Nw\in \Xg(P)$ are $\G$-invariant, and $v>0$.
(Recall the  classical Helmholtz decomposition: $X = \nabla f + X_0$ with $X_0$ divergence-free.) Essentially, the existence of such a decomposition is due to the following result, which we prove in Appendix \ref{app_PDE}: 

\begin{propintro}[Modified Helmholtz decomposition]
Given a smooth vector field $X$ on a compact Riemannian manifold $(B,g^B)$, there exists 
 a unique  (up to scaling) non-trivial smooth solution to the second order linear PDE 
\[
  \divg_B(\nabla u + u X) = 0\,,
\] 
and such a solution does not change sign on $B$. In particular, we can write
\[
    X = -\nabla \log u + X_0, \qquad \divg_B(u \, X_0) = 0, \qquad u>0.
\]
\end{propintro}


The third ingredient is the $\G$-invariant function  $\fx := \unm \Vert \Nw \Vert^2$ on $P$. The vertical and horizontal Einstein equations for the submersion $\pi_P$, together with the general formula
\begin{eqnarray}\label{eqn_genricform}
   \ric_P(E,E)&=& 
	\divgP( \nablaP_E E)    - E\divgP(E)-\tr \left(( \nablaP E ) \circ (\nablaP  E)\right)\,,
\end{eqnarray}
valid for an arbitrary vector field $E$ on any Riemannian manifold, give a nice expression for 
the Laplacian $\Delta_P \fx=\divgP( \nabla \fx)$. 
Combining all three ingredients by setting $f:=  \vbM + \log v + \fx$ we obtain 
\[
	\Delta_P f + \la \nabla \log v + \Nw, \nabla f \ra \geq 0,
\]  
which can also be written in divergence form, yielding \eqref{eqn_divZintro}.

The rest of the claims in Theorem \ref{thm_rigidity} follow essentially from a similar argument, where this time the function involves the scalar curvature of the $\N$-orbits: see Section \ref{sec_scaln}
and \ref{sec_Einsteinsubm}.  The proof of  Theorem \ref{thm_ricneg} is a simplified version of the above discussion, since unimodularity yields $\Nw = 0$: see Theorem \ref{thm_rig_polarricn}.

 \subsection{Proof outline of  Theorem \ref{thm_alek}}                      

Let $(M=\F/\Hh, g)$ be a homogeneous Einstein space with $\ric_g = -g$, and assume for simplicity that $\F$ is non-compact semisimple. After quotienting by the center of $\F$, we may assume that $\F$ is linear semisimple. This implies that, in the Iwasawa decomposition $\F = \K \A \N$, $\K$ is a maximal compact subgroup. We may pick $\K$ so that $\Hh \leq \K$. Setting $\G := \A\N$, we apply Theorem \ref{thm_rigidity} to the induced action of $\G$ on $M$, which is free and has compact quotient $\K/\Hh$. 

Unfortunately, this is still not enough for concluding and we need more structure. Using that the Lie group $\G$ is completely solvable and admits a left-invariant Einstein metric $g^\G$ (the symmetric metric on $\F/\K$), we show in Theorem \ref{thm_X=0} that the mean curvature vector $N$ of the $\N$-orbits is in fact $\G$-vertical. This follows from applying the Bochner technique, together with  subtle algebraic arguments that allow us to choose Einstein $\G$-invariant metrics on the $\G$-orbits which are `compatible' with $g$ (Proposition \ref{prop_NS}).

Notice that so far we have not used homogeneity of $M$ but only $\G$-invariance.  To exploit the full homogeneity assumption, we obtain new algebraic  formulae for computing the difference $\ric_g(U,U) - \ric^\ve(U,U)$ between the Ricci curvature of $M$ and that of the $\N$-orbits, in vertical directions $U$ (Proposition \ref{prop_ricformula}). These rely on the rigidity results from Theorem \ref{thm_rigidity}.  Tracing and using the Einstein condition one gets
\begin{equation}\label{eqn_ricformintro}
		\scal^\ve(p) + \dim \ngo  = \sum_{i=1}^n \la \nabla_{E_i} E_i, N \ra_p
\end{equation}
at the point $p := e\Hh \in \F/\Hh$, for a certain set of Killing fields $\{ E_i\}$ in $\fg$ (which are \emph{not} in a reductive complement of $\hg$ in $\fg$), which at $p$ form an orthonormal basis of $T_p M$.

Using that $N$ is $\G$-vertical, one can find a Killing field $A\in \ag := \Lie(\A)$ with $A_p = N_p$, and this allows us to  bound the right-hand-side in \eqref{eqn_ricformintro} from above by $\sum \beta^+_i$, with equality implying that the normalizer $N_\F(\G)$ of $\G$ in $\F$ acts transitively on $\F/\Hh$ (Proposition \ref{prop_estimate}). But the $\N$-orbits are locally isometric to nilsolitons, thus the left-hand-side in \eqref{eqn_ricformintro} equals $\sum \beta^+_i$. It follows that  $N_\F(\G)$ acts transitively. Since the Levi factor of $N_\F(\G)$ is compact, $M$ is a simply-connected Einstein solvmanifold by \cite{JblPet14,Jbl2015}, and in particular, diffeomorphic to a Euclidean space.

For the general case,  the structure theory for homogeneous Einstein spaces \cite{alek,JblPet14,AL16} yields a nice presentation $M = \F/\Hh$ with Levi decomposition $\F = \Ll \ltimes \Ss$, where $\Ll = \K \A \N$ is as above, and $\Ss$, the solvable radical, is completely solvable. We then set $\G := (\A\N) \ltimes \Ss$ and argue in a similar manner. The only major difference is that now it is not obvious that $\G$ admits a left-invariant Einstein metric, but we show that this is indeed the case in Theorem \ref{thm_F/K_Ein}.

\subsection{Organisation of the article}

In $\S$\ref{sec_setup}  we review the Ricci curvature formulae of a Riemannian submersion, focusing in $\S$\ref{sec_N} on isometric group actions. $\S$\ref{sec_newbetavol} describes the $\log \beta$-volume functional on left-invariant metrics on a Lie group, and this is applied to the orbits of an isometric  action in $\S$\ref{sec_betavolM}.

After estabilishing in $\S$\ref{sec_estimates} some key differential inequalities, the case of $\G$ unimodular is treated in $\S$\ref{sec_negRic}, assuming only  negative Ricci curvature.  The fact that the $\N$-action is polar and some interesting consequences of this are proved in $\S$\ref{sec:rigidity}. 
In $\S$\ref{sec_scaln} we describe the geometry of the $\N$-orbits, and  in $\S$\ref{sec_Einsteinsubm} we finish the proof of Theorem \ref{thm_rigidity}   and  its two Corollaries \ref{cor_noS1} and \ref{cor_negversusEinstein}.  

The last four sections are devoted to  proving Theorem \ref{thm_alek}. $\S$\ref{sec_NGvert} refines Theorem \ref{thm_rigidity} under some additional assumptions on $\G$. 
In $\S$\ref{sec_newalgform} we prove new Ricci curvature formulae for homogeneous spaces, and deduce an important algebraic estimate.  $\S$\ref{sec_semidirect} shows that certain semi-direct products of Einstein submanifolds are again Einstein submanifolds. Finally, Theorem \ref{thm_alek} is proved in $\S$\ref{sec_alek}.

The appendices cover the modified Helmholtz decomposition (\ref{app_PDE}, \ref{app_div}), curvature computations and estimates under an isometric group action (\ref{app_Killing}, \ref{app_beta}), and the reduction of the Alekseevskii conjecture to the simply-connected case (\ref{app_quotients}).

\subsection{Notation}
Throughout the paper and unless otherwise stated, smooth manifolds will be denoted with letters $M, P, B$, and Lie groups with $\G, \N, \F,\Ll$, etc. Typically, $\G$ acts on $M$, $\N$ is the nilradical of $\G$, $\Ll$ is semisimple, $\F$ is transitive on $M$. 

For a Riemannian manifold $(B, \gB)$, we denote by:

\begin{itemize}[wide]
\item $\nabla^B$ the Levi-Civita connection;
\item $\ric^B$ the Ricci curvature of $\gB$; $\Ric^B$ the Ricci endomorphism ($\gB(\Ric^B \cdot, \cdot) = \ric^B(\cdot,\cdot)$);
\item $\scal^B = \tr \Ric^B$ the scalar curvature;
\item $\divg_B X := \tr \nabla_{\cdot}^B X$, the divergence of a vector field $X\in \Xg(B)$;
\item $\Delta_B f := \divg_B \nabla f$, the Laplace-Beltrami operator, $f\in \cca^2(B)$.
\end{itemize}

Given a proper, isometric Lie group action of $\G$ on $(M, g)$ with a single orbit type, we endow the orbit space $B:= M/\G$, a smooth manifold, with the quotient metric $\gB$, so that
\[
    \pi : (M^n,g) \to (B^d, \gB)
\]
is a Riemannian submersion. We will also denote:
\begin{itemize}[wide]
\item $\ggo \subset \Xg(M)$ the Lie algebra of Killing fields coming from the $\G$-action;
\item $TM = \ho_\G \oplus \ve_\G$  the  orthogonal decomposition into  horizontal and vertical distributions (we omit the subscript $\G$ when it is clear from the context); 
\item $g^\ve, \ric^\ve, \scal^\ve$ the  geometric data of the $\G$-orbits in $M$ with the  submanifold geometry;
\item $L_X \in \End(\ve)$ is the  shape operator of the $\G$-orbits in the direction $X\in \ho$, see \eqref{eqn_LX};
\item $\Xg(M)^\G$ the set of $\G$-invariant vector fields on $M$ (recall that, in general,  $\Xg(M)^\G \neq \ggo$);
\item $\cca^\infty(M)^\G$ the space of $\G$-invariant smooth functions on $M$, $\cca_+^\infty(M)^\G$ the cone consisting of those which are strictly positive; 
\item If $X\in \Xg(M)^\G$ is horizontal (i.e.~\emph{basic}), the corresponding $\pi$-related vector field on $B$ is denoted by $\bar X$, and sometimes simply by $X$ when there is no risk for confusion.
\end{itemize}


\vs \noindent {\it Acknowledgements.} {We would like to thank Hans-Joachim Hein and Luis Silvestre for 
sharing with us beautiful proofs of the modified Helmholtz decomposition  and
Anand Dessai  and Claude LeBrun for very helpful comments. The first-named author was funded by the Deutsche Forschungsgemeinschaft (DFG, German Research Foundation) under Germany's Excellence Strategy EXC 2044 –390685587, Mathematics M\"unster: Dynamics-Geometry-Structure, and the Collaborative Research Centre CRC 1442, Geometry: Deformations and Rigidity. 
The second-named author is an Australian Research Council DECRA fellow (project ID DE190\-101063). }


\section{Riemannian submersions}\label{sec_setup}

We briefly recall  Riemannian submersions, based on \cite[Chapter 9]{Bss}.
Let $M^n$ and $B^d$ be smooth manifolds and
$\pi: M\to B$ 
be a smooth submersion, that is, $d\pi_p:T_p M \to T_{\pi(p)} B$ is surjective for all $p\in M$.
In this case, for all $b\in B$ the preimage $\pi^{-1}(b)=:F_b$ is an embedded submanifold of $M$.

We endow $M$  with a complete Riemannian metric $g$. Then for all $b, \tilde b\in B$ 
the fibres $F_b$ and $F_{\tilde b}$ are diffeomorphic, and
the tangent spaces to the fibers give rise to the \emph{vertical distribution} $\ve$, a subbundle of $TM$. That is
for each $p \in M$ we have $\ve_p=T_p F_{\pi(p)}$. At each $p \in M$ we set now $\ho_p:= (\ve_p)^\perp$, orthogonal
with respect to  $g$.
This leads to the smooth \emph{horizontal distribution} $\ho$, another subbundle of $TM$.
Thus 
\begin{equation}\label{eqn_TM=V+H}
  TM = \ve \oplus \ho\,.
\end{equation}
We endow $B$ with a Riemannian metric $\gB$.
The map 
\[
 \pi:(M,g)\to (B, \gB)
\] 
is a \emph{Riemannian submersion}, if for all $p \in M$ the linear map
\begin{align}\label{def_Riemsub}
  (d\pi)_p : (\ho_p,g\vert_{\ho_p}) \to \left(T_{\pi(p)}B, \, \gB_{\pi(p)} \right)
\end{align}
is an isometry between Euclidean vector spaces. In the following we will always assume this.

For every smooth vector field $E$ on $(M,g)$ we write $E=\ve E + \ho E$ according to \eqref{eqn_TM=V+H}.
As in \cite{Bss}, those vector fields on $M$ taking values in $\ve$ are called vertical and denote by letters $U,V, W$, whereas those taking values in $\ho$ are called \emph{horizontal}, and denoted by letters $X, Y, Z$.
A horizontal vector field $X$ is \emph{basic}, if it $\pi$-related to a vector field $\bar X \in \Xg(B)$, that is,
$(d\pi)_p \, X_p=\bar X_{\pi(p)}$ for all $p\in M$. 
Recall that every vector field $\bar X$ on $B$ can be uniquely lifted to a basic vector field $X$ on $M$: see \cite[9.23]{Bss}. 
To simplify notation we will sometimes write $X$ instead of $\bar X$. 
Since the Lie bracket of $\pi$-related
vector fields is $\pi$-related we conclude that $[U,X]$ is vertical, if $U$ is vertical and $X$ is basic.

We denote  by $\nabla$ the Levi-Civita connection
of $(M,g)$ and by $\nabla^B$ the Levi-Civita connection of $(B,\gB)$. 
Then, for basic vector fields $X,Y$ we have
 \begin{equation*}
     (d\pi) \cdot (\ho \nabla_X Y) =  \nabla^B_{\bar X}\bar Y\,.
 \end{equation*}
In order to compute the Ricci curvature of $(M,g)$ we recall O'Neil's $T$ and $A$ tensors:
see \cite{ON83}.
We set
\begin{equation*}
 T: TM \times TM \to TM\,\,;\,\,\,(E_1,E_2) \mapsto 
  T_{E_1}{E_2}:=\ho \nabla_{\ve E_1} \ve E_2  + \ve \nabla_{\ve E_1} \ho E_2\,.
\end{equation*}
For $U$ vertical and $X$ horizontal we
have    $T_U X = \ve \nabla_U X$ and  $T_X =0$. Moreover $T$ has the following symmetries:
\begin{align*}
   T_U  V  = T_V U 
   \quad\textrm{ and }\quad
  \langle T_U V,X\rangle =  - \langle V,T_U X\rangle.
\end{align*}
It is convenient to define the tensor
$$
  L: TM\times TM\to TM\,\,;\,\,\,(E_1,E_2) \mapsto  \ve \nabla_{\ve E_2} \ho E_1\,.
$$
Notice essentially
$L: \ho \times \ve \to \ve$
and that 
\begin{equation}\label{eqn_LX}
  L_X(U):=L(X,U)=T_U X = \ve \nabla_U X
\end{equation}
is the shape operator of the fibres in the normal direction $X\in \ho$. In short: $L$ is the vertical component of
$T$ with flipped entries.

The $A$-tensor is defined by
\begin{equation*}
 A: TM \times TM \to TM\,\,;\,\,\,(E_1,E_2) \mapsto
  A_{E_1} E_2 := \ho \nabla_{\ho E_1}\ve E_2 +\ve \nabla_{\ho E_1}\ho E_2\,.
\end{equation*}
For $U$ vertical and $X$ horizontal
we have   $A_U =0$  and $A_X  U  =\ho \nabla _X U$. Moreover,
$A$ satisfies the following properties:
\begin{align*}
    A_X Y = -A_Y X  
    \quad\textrm{ and }\quad
    \langle A_X Y,U\rangle = - \langle Y,A_X U\rangle.
\end{align*}
The $A$-tensor measures the integrability of $\ho$, since for horizontal vector fields $X,Y$ we have
$$
 A_X Y =\tfrac{1}{2} \ve [X,Y].
$$
We refer to \cite{Bss} for the proof of the above facts.  We also set
\begin{align}
  \langle A \, U, A  \, U\rangle :=& \,\, \sum_{j=1}^{d} \langle A_{X_j} U, A_{X_j}U\rangle, \label{eqn:AU}\\
  \langle A_X, A_X\rangle :=& \,\,
  \sum_{i=1}^{n-d} \langle A_X U_i,A_X U_i)= \sum_{j=1}^{d} \langle A_X X_j,A_X X_j), \label{eqn:AX} \\
    \langle T  X ,T X\rangle =& \,\, \sum_{i=1}^{n-d}
  \langle T_{U_i}X, T_{U_i} X \rangle = \Vert L_X \Vert^2\,, \label{eqn:TX}
\end{align}
for local horizontal and vertical orthonormal frames $\{ X_j\}_{1 \leq j \leq  d}$, $\{ U_i\}_{1\leq i \leq n-d}$, respectively.

\begin{definition}[Mean curvature vector]
Let  $\{U_i\}_{1 \leq i \leq n-d}$ be a local vertical orthonormal frame. Then the \emph{mean curvature vector} of the fibers $F_b$ is denoted by
\begin{eqnarray}\label{eqn:defN}
 N = \sum_{i=1}^{n-d}  T_{U_i}{U_i}  = \sum_{i=1}^{n-d}  \ho \nabla _{U_i}{U_i}\,.
\end{eqnarray}
\end{definition}

Now we can state the following well known formulae for the Ricci curvature of $(M,g)$:
see \cite[9.36]{Bss}  (cf.~ also  \cite[Prop.~8.1]{NT18}).

 \begin{theorem}[Ricci curvature]\label{thm:Ric}
 The Ricci tensor $\ric=\ric(g)$ of $(M,g)$ is given by
 \begin{eqnarray}
  \ric(U,U) 
    &=&
    \ric^\ve(U,U)
    +\langle L_N U, U\rangle
    +\langle A\,  U, A \, U\rangle
    - \sum_{j=1}^{d} \langle (\nabla_{X_j}  L)_{X_j}  U ,U\rangle , \label{Ricvv}\\
  \ric(U,X) 
    &=&  
     - \sum_{i=1}^{n-d} \la (\nabla_{U_i} T)_{U_i} U, X  \ra 
      + \la \nabla_U N, X \ra   \label{Ricvh} \\
      & &
      + \sum_{j=1}^{ d}\la (\nabla_{X_j} A)_{X_j} X, U \ra 
      - 2 \, \la A_X, T_U \ra, \nonumber\\
  \ric(X,X) 
    &=& 
       \ric^B (\bar X,\bar X)   
      -2 \, \Vert A_X \Vert^2 
      -\Vert L_X\Vert^2
      +\langle \nabla_X N, X\rangle  \label{Richh}\,.
 \end{eqnarray}
Here $U$ is vertical, $X$ is horizontal, $\ric^\ve$ denotes the Ricci tensor of the fibres $(F,\gF)$, $\gF=g\vert_{TF}$,
and $\ric^B$ denotes the Ricci tensor of $(B,\gB)$.
\end{theorem}

For the second term in \eqref{Ricvv} we have $-\la T_U U,N\ra = \la L_N U,U\ra$.
The last term in \eqref{Ricvv} has a different sign compared to 
the last term in  \cite[(9.36a)]{Bss}. The reason is simply
that by   \cite[9.32 $\&$ (9.33h)]{Bss},
\[
 (\tilde \delta T)(U,U)=\sum_{j=1}^{d} g((\nabla_{X_j} T)_U U,X_j)= -\sum_{j=1}^{d} g((\nabla_{X_j} T)_U X_j,U)=
 -\sum_{j=1}^{d} g((\nabla_{X_j} L)_{X_j} U,U)
\]
by definition of $L$. The first and the third term in \eqref{Ricvh}
come with a different sign compared to (9.36b) in \cite{Bss}, 
simply because the divergence in \cite{Bss} comes with a minus sign: see
(9.33e) and (9.33f) in \cite{Bss}.


The following properites of the mean curvature vector will be extremely useful.

\begin{lemma}\label{lem_divN}
Suppose that $N$ is a basic vector field, $\pi$-related to a vector field $\bar N$ on $B$. Then,
for every local orthonormal 
horizontal frame $\{ X_j\}_{1 \leq j\leq  d}$ we have  
\begin{equation}\label{eqn_N}
  N = - \sum_{j=1}^{ d} (\tr L_{X_j}) \cdot X_j \, , 
\qquad \quad
    \divgB ( \bar N) = - \sum_{j=1}^{ d} \tr  \big( (\nabla_{X_j} L)_{X_j}\big) \,.
\end{equation}
\end{lemma}

\begin{proof}
For $X$ horizontal we have
\begin{equation}\label{eqn_trLX}
   \langle N,X \rangle = \sum_{j=1}^{n-d} \langle T_{U_j}U_j, X\rangle = - \sum_{j=1}^{n-d} \langle U_j, \nabla_{U_j}X\rangle = -\tr L_X\,.
\end{equation}
Since $N$ is horizontal, this shows the first identity in \eqref{eqn_N}. 

Using that the expresion to prove is tensorial in $X_j$, we may assume that $X_j$ is basic for all $j$, $\pi$-related to $\bar X_j$, and $\{ \bar X_j \}$ for a local orthonormal frame in $B$. Thus, from \eqref{eqn_trLX} we get
\begin{eqnarray*}
  \divgB(\bar N)
   &= &
       \sum_{k=1}^{ d} \langle \nabla^B_{\bar X_k} \bar N,\bar X_k\rangle\\
       &= &
       \sum_{k=1}^{ d} \langle \nabla_{ X_k}  N, X_k\rangle\\
   &=&
     -\sum_{j=1}^{ d} X_j \tr L_{X_j}  
    -  \sum_{j,k=1}^{ d} \tr L_{X_j} \cdot \langle \nabla_{X_k} X_j ,X_k\rangle\\
   &=&
     -\sum_{j=1}^{ d}  \tr \big(  (\nabla_{X_j} L)_{X_j}+L_{\ho \nabla_{X_j} X_j}\big) 
      +  \sum_{k=1}^{ d} \tr L_{\ho \nabla_{X_k} X_k}\\
   &=&  
     -\sum_{j=1}^d \tr \big((\nabla_{X_j} L)_{X_j}\big)\, ,
\end{eqnarray*}
 where in the third equality we used Lemma \ref{lem_nablatrver} below.
 This shows the claim.
\end{proof}


\begin{lemma}\label{lem_nablatrver}
Let $E \in \End(\ve)$ 
and let $X$ be a horizontal vector field. Then, 
\[
  X (\tr_{\ve} E) = \tr_{\ve} (\nabla_X E). 
\]
\end{lemma}

\begin{proof}
Let $\{U_i\}$ be a local vertical orthonormal basis in $(M,g)$. Then
\begin{align*}
   X (\tr_{\ve} E) 
    &= 
     X \sum_{i=1}^{n-d} \la E U_i,U_i \ra  
    = 
    \sum_{i=1}^{n-d} \la (\nabla_X E) U_i,U_i \ra + \la E (\nabla_X U_i),  U_i \ra  + \la E U_i, \nabla_X U_i \ra \\
    &=  \tr_\ve(\nabla_X E)+ \sum_{i,j=1}^{n-d} \la \nabla_X U_i, U_j \ra \, \la  (E+E^T) U_i, U_j \ra.
\end{align*}
Since $\la\nabla_X U_i, U_j \ra$ is skew-symmetric in $i,j$ (because $\la U_i, U_j \ra$ is  constant), the lemma follows.
\end{proof}

\section{Isometric group actions}\label{sec_N}

We turn now to a special class of Riemannian submersions induced by isometric actions of Lie groups.
We assume that a connected Lie group $\G$ acts properly, almost effectively and isometrically on 
a connected Riemannian manifold $(M,g)$. 
We assume furthermore that all orbits are principal.
Then, the quotient map
\[
   \pi: (M,g) \to (B:=M/\G, \gB) \,\,;  \qquad p \mapsto \G \cdot p,
\]
is a smooth Riemannian submersion with smooth orbit space $B$, see \cite[9.12]{Bss}. Here $\gB$ is defined by 
\eqref{def_Riemsub} and will be denoted as the \emph{quotient metric}. Note the mean curvature vector $N$ 
of the $\G$-orbits in $M$ is $\G$-invariant, thus a basic vector field.

\begin{remark}\label{rem_complete}
Properness of the action ensures that 
any $\G$-orbit is a closed submanifold of $M$ \cite[Prop.~1.1.4]{Pal61}. If  $\Gamma < \G$ denotes the discrete ineffective kernel of the action, $\G/\Gamma$ acts  effectively, properly and isometrically on $(M^n,g)$, with compact isotropy groups.
Note also that completeness of $(M,g)$ is a consequence of the compactness of the orbit space $B$.
\end{remark}

\medskip

The following well-known result shows that isometric actions by unimodular Lie groups 
are analytically easier than those of non-unimodular Lie groups.

\begin{lemma}\label{lem_Ngrad}
Suppose that $\G$ acts properly, almost freely and isometrically on $(M^n,g)$ with smooth orbit space $B$.
Then, if $\G$ is unimodular, the mean curvature vector 
$N$ is the gradient vector field of a $\G$-invariant function on $M$.
\end{lemma}

\begin{proof}
Since $\G$ acts almost freely, the istropy group $\G_p$ at any point $p \in M$ is discrete.
Let now $\{\UK_i\}_{1 \leq i \leq \dim \ggo}$ 
denote a basis of $\ggo$ and set $v:=\sqrt{\det E}:M \to \RR$, with
$E:=(g(\UK_i,\UK_j))_{1 \leq i,j \leq \dim \ggo}$.
Notice that since $\G$ is unimodular the function $v$ is constant on $\G$-orbits  by Lemma \ref{lem_g_pAdx}.

For a point $p \in M$ let  $\gamma(t)$ be a horizontal, unit-speed geodesic in $(M,g)$ with 
$\gamma(0)=p$, and set $v(t):=v(\gamma(t))$, $E(t):=E(\gamma(t))$. Then,
$$
  \tfrac{d}{dt} v(t) = \unm v(t) \cdot \tr \big( E^{-1}(t) E'(t) \big)\,.
$$
Let $X$ denote a basic vector field with $X_{\gamma(t)}:=\gamma'(t)$.
Then, by Lemma \ref{lem:killeft}
$$
   X g (\UK_i,\UK_j)_{\gamma(t)}= 2 g(L_X \UK_i,\UK_j)_{\gamma(t)}\,.
$$
By the lemma below we deduce 
$$
  v'=v \cdot \tr L_X=- v \cdot \la X,N\ra
$$ 
along $\gamma(t)$ using \eqref{eqn_trLX}. This shows the claim.
\end{proof}

The assumption that $\G$ acts almost freely is actually not needed.

\begin{lemma}
Let $(V,\ip)$ be a Euclidean vector space, $L \in \End(V)$
and $\{v_1,\ldots,v_r\}$ be any basis of $V$. Then
$$
  \tr L = \tr \big( (\la v_i,v_j \ra)^{-1} \cdot (\la L v_i,v_j \ra) \big)\,.
$$
\end{lemma}

\begin{proof}
We define the matrix $G$ by $G_{ij}:=\la v_i,v_j\ra$, $1 \leq i,j \leq r$. Then $G$ is symmetric 
and positive definite. Let $P$ denote the square root of $G$. Then it is easy to check that
$\{\bar v_i:=P^{-1}v_i\}$ is an orthonormal basis of $V$. Using this we obtain
$$
 \tr L= \sum_{i=1}^r \langle L \bar v_i,\bar v_i\ra
  =\sum_{i,k,l=1}^r P_{ik}^{-1}P_{il}^{-1}\langle L v_k,v_l\ra = \sum_{k,l=1}^r G_{kl}^{-1} \langle Lv_k,v_l\rangle\,.
$$
This shows the claim.
\end{proof}

\begin{remark}\label{rem_nongrad}
If $\G$ is a non-unimodular Lie group, the mean curvature vector $N$ will in general not be a gradient vector field.
It can be shown that the skew-symmetric part of $(\nabla  N)\vert_{\ho}$ is given by $-A \mcv$, where $\mcv$ is the mean curvature vector of  the homogeneous space $\G\cdot p$: see \cite[7.32]{Bss}.
\end{remark}



We now recall another well-known fact in the context of isometric group actions:

\begin{lemma}\label{lem:killeft}  
Let $\UK\in \ggo$ be a vertical Killing field and let $X$ be basic. Then, $[\UK,X]=0$.
\end{lemma}

\begin{proof}
For a vertical vector field $V$ we have
\begin{eqnarray*}
  \langle [\UK,X],V\rangle = \langle \nabla_{\UK} X -\nabla_X \UK,V \rangle
      = \langle X, -\nabla_{\UK} V+ \nabla_V \UK\rangle
      =-\langle X,[\UK,V]\rangle=0\,,
\end{eqnarray*}
using that $\nabla \UK$ is skew-symmetric. Since $[\UK,X]$ is vertical,
the claim follows.
\end{proof}

Under some additional assumptions, the off-diagonal Ricci curvature formula from Theorem \ref{thm:Ric} can be simplified as follows:

\begin{proposition}\label{prop_ric_offdiag}
Let $\pi:(M,g) \to (B, \gB)$ be defined by a polar, free, proper, isometric action of a 
Lie group $\G$ on $(M,g)$. Then,
\[
		\ric(U,X) = 
		-\la L_X U, \mcvl  \ra - \la \nabla^\ve U, L_X \ra,
\]
for all $\G$-invariant vertical $U$ and all basic $X$. Here $\mcvl := \ve \sum_{i=1}^{n-d} \nabla_{U_i} U_i$ for a vertical $\G$-invariant orthonormal frame $\{ U_i \}_{1 \leq i \leq n-d}$.
\end{proposition}

\begin{proof}
The polar assumption, equivalent to $A=0$, implies that for every basic $Y$ and vertical $V$ we have
\[
	\ho \nabla_V Y = \ho [V,Y] +  \ho \nabla_Y V  = 0 + A_Y V = 0.
\]
Since $N$ and $X$ are basic, 
this yields $\la \nabla_U N, X \ra = -\la N,\nabla_U X\ra =0$. Hence, by Theorem \ref{thm:Ric}, we have
\[
		\ric(U,X) =- \sum_{i=1}^{n-d} \la (\nabla_{U_i} T)_{U_i} U, X  \ra.
\]
Using that $(\nabla_{U_i} T)_{U_i}$ is skew-symmetric by \cite[(9.32)]{Bss} this equals
\begin{align*}
	\ric(U,X) = & \,\, \sum_{i=1}^{n-d} \la  U,  (\nabla_{U_i} T)_{U_i}X  \ra \\
			=& \,\,  \sum_{i=1}^{n-d} U_i \la U, T_{U_i} X \ra - \la \nabla_{U_i} U, T_{U_i} X\ra  -\la U, T_{\nabla_{U_i} U_i} X \ra - \la U, T_{U_i} (\nabla_{U_i} X) \ra.
\end{align*}
Recall that $U$ and $\{ U_i\}$ are $\G$-invariant. The first term vanishes because $\la U,T_{U_i} X \ra$ is constant along orbits.  The last term also vanishes since $\ho \nabla_{U_i} X = 0$ by our first observation above. The second term clearly equals $\la \nabla^\ve U, L_X \ra$. Finally, the third term gives  $-\la U, T_\mcvl X \ra$, where $ \mcvl = \ve \sum_i \nabla_{U_i} U_i$, and the proposition follows.
\end{proof}

We conclude this section by showing that the nilradical may be assumed to  act  freely.  First a well-known result in Lie theory for which we were not able to find a reference:

\begin{lemma}\label{lem_cpt_nil}
Any compact subgroup $\K$ of a connected nilpotent Lie group $\N$ is central.
\end{lemma}

\begin{proof}
By Engel's theorem, we may choose a basis for $\ngo$ so that $\Ad(\N) \leq \U(m, \RR)$, where $\U(m, \RR)$ is the group  of $m\times m$ upper triangular real matrices with  $1$'s on the diagonal, $m = \dim \ngo$. Then, $\Ad(\K) \leq \U(m, \RR)$ is a compact Lie subgroup. But $\U(m, \RR)$ is diffeomorphic to a Euclidean space, hence $\Ad(\K)$ is trivial.
\end{proof}

\begin{lemma}\label{lem_Nfree}
Let $\G$ act properly, effectively and isometrically on $(M,g)$ with a single orbit type. Then, the induced action of the nilradical $\N$ of $\G$ on $(M,g)$ is proper and  free.
\end{lemma}

\begin{proof} 
Since $\N$ is a closed subgroup of $\G$ \cite[Thm.~3.18.13]{Varad84}, it acts properly  and effectively on $M$ \cite[Prop.~1.3.1]{Pal61}. On the other hand, since all $\G$-orbits are principal and $\N$ is normal in $\G$, also all $\N$-orbits are principal. Finally, isotropy subgroups $\N_p$ are compact, hence central by Lemma \ref{lem_cpt_nil}, and therefore trivial by effectiveness.
\end{proof}


\section{The space of left-invariant metrics and the \texorpdfstring{$\beta$}{b}-volume}\label{sec_newbetavol}



In this section we review the basic properties of the space of left-invariant Riemannian metrics $\mca^\N$ on a Lie group $\N$ with Lie algebra $\ngo$, viewing it as a symmetric space. This is the pointwise analog of
considering the $L^2$-metric on the space $\mca$ of Riemannian metrics on a compact manifold: 
see \cite[Chapter 4]{Bss}, \cite{BrCl2011}.
After choosing a background metric, we view $\mca^\N$ as a simply-connected Lie group, and define the $\Aut(\ngo)$-invariant 
$\log \betab$-volume functional on $\mca^\N$ as the potential function of a certain left-invariant, gradient vector field on $\mca^\N$. This functional will play a  key role in the proofs of our main results.  We emphasize that in this section, $\N$ is an arbitrary connected Lie group which is not necessarily nilpotent.

By evaluation at the identity $e\in \N$, $\mca^\N$ is naturally identified with the space 
$\Sym^2_+(\ngo^*)$ of positive-definite inner products on $\ngo$.  Consider on $\mca^\N \simeq \Sym^2_+(\ngo^*)$ the  $\Gl^+(\ngo)$-action given by 
\begin{equation}\label{eqn_action_ip}
    L_q(h) := (q \cdot h)(\,\cdot, \cdot) := h(q^{-1} \cdot, q^{-1} \cdot),
\end{equation}
for each $q\in  \Gl^+(\ngo)$, $h\in \Sym^2_+(\ngo^*)$.
Note that for each $q \in \Gl^+(\ngo)$, the map $L_q$ is a diffeomorphism of $\mca^\N$.
Furthermore, since the action is transitive,
fixing a background inner product $\bkg$ we see that 
$\mca^\N$ can be described as a symmetric space:
\[
    \mca^\N \simeq  \Sym^2_+(\ngo^*) \simeq \Gl^+(\ngo)/\SO(\ngo,\bkg)\,.
\]
The space $ \Sym^2_+(\ngo^*)$ is an open set in $ \Sym^2(\ngo^*)$, thus for each 
$h\in \mca^\N$ we have $T_h \mca^\N \simeq  \Sym^2(\ngo^*)$. Using the action \eqref{eqn_action_ip}
we may describe tangent vectors using action fields: 
\[
  \rho(E)h:=\tfrac{d}{dt}\big\vert_{0} L_{\exp(tE)}(h)\,,
\]
where $E \in \End(\ngo) \simeq T_e \Gl^+(\ngo)$ and
\begin{equation}\label{eqn_rho}
		(\rho(E) h) (\,\cdot, \cdot) = - h(E \,\cdot, \cdot) - h(\,\cdot, E \,\cdot).
\end{equation}
As a consequence,
\[
    T_h  \mca^\N = \{ \rho(E)h : E\in \End(\ngo) \}\,.
\]
Next, we recall the symmetric metric $\gsym$ on $\mca^\N$.
For $h \in \mca^\N$ and $k \in T_h \mca^\N$ we have
$$
   \gsym( k, k)_h = 
   \sum_{i,j=1}^{\dim \ngo} k(e_i,e_j)\cdot k(e_i,e_j)
$$
for an $h$-orthonormal basis $\{e_i\}$ of $\ngo$: see e.g. \cite{BrCl2011} and 
references therein. The symmetric metric can also be computed explicitly in terms of the above defined action fields $\rho(E)h$ by the following lemma, which also justifies its name:

\begin{lemma}\label{lem_gsym_h}
The metric $\gsym$ is $\Gl^+(\ngo)$-invariant, and  for $h  \in \mca^\N$, $E \in \End(\ngo)$ we have
\[
    \gsym(\rho( E) h, \rho(E) h)_h = 
     \tr (E+E^{T_h})(E+E^{T_h}),
\]
where $T_h$ denotes transpose with respect to $h$.
\end{lemma}

\begin{proof}
The second claim follows by definition:
\begin{align*}
 \gsym(\rho( E) h, \rho(E) h)_h 
   &=
    \sum_{i,j=1}^{\dim \ngo} (\rho( E) h)(e_i,e_j)\cdot (\rho( E) h)(e_i,e_j) \\
		&=
		\tr (E+E^{T_h})(E+E^{T_h}).
\end{align*}
Now write $h = q \cdot \bkg$. Since  $(dL_q)_{\bkg}\cdot (\rho(E)\bkg)=\rho(qEq^{-1})h$, 
it remains to be shown
that
\[
     \gsym \left(\rho( qEq^{-1}) h, \rho(qEq^{-1}) h \right)_h = \gsym \left(\rho(E) \bkg, \rho(E)\bkg \right)_{\bkg}\,.
\]
Using $h =q \cdot \bkg$, from
\[
    h((qEq^{-1})^{T_h} v, w) = \bkg(q^{-1} v, E q^{-1}w ) 
		=  \bkg(E^{T_{\bkg} } q^{-1} v, q^{-1} w)  
		= h( q E^{T_{\bkg}} q^{-1} v, w),
\]
we deduce 
\begin{equation}\label{eqn_transposes}
	(q E q^{-1})^{T_h} = q E^{T_{\bkg}} q^{-1}.
\end{equation}
It follows that $\gsym$ is $L_q$-invariant.
\end{proof}

If $\ngo$ is non-abelian, let $\betab  \in \End(\ngo)$ be the $\bkg$-self-adjoint endomorphism 
given by the Lie bracket of $\ngo$: see Lemma \ref{lem_appbeta}. Fix a  $\bkg$-orthonormal ordered basis $\bca_\betab$ of eigenvectors of $\betab$, with eigenvalues in non-decreasing order, and let $\Bbeb \leq \Gl^+(\ngo)$ be the set of those endomorphisms which are lower triangular in the basis $\bca_\betab$, with positive diagonal entries. This is a simply-connected solvable Lie subgroup, whose Lie algebra $\bg_\betab$ contains $\betab$.
 Moreover, since $\Gl^+(\ngo) = \Bbeb \cdot \SO(\ngo,\bkg)$ and $\Bbeb \cap \SO(\ngo,\bkg)=\{e\}$,
$\Bbeb$ acts simply-transitively on $\mca^\N$, yielding a diffeomorphism
\begin{equation}\label{eqn_diffeo_mca}
  \mca^\N \simeq \Bbeb, \qquad h = L_q(\bkg) \, \mapsto \,  q\,.
\end{equation}
Furthermore,
\[
    T_h  \mca^\N = \{ \rho(E)h : E\in \bg_\betab \}\,.
\]
A vector field $X$ on $\mca^\N$  is called $\Bbeb$-invariant (and from now on simply `left-invariant') 
if 
$$
 X_h=X_{L_q(\bkg)}=(dL_q)_{\bkg} \cdot X_{\bkg}
$$ 
for all $q \in \Bbeb$, $h=L_q(\bkg)\in \mca^\N$.
Writing $X_{\bkg} = \rho(E)\bkg \in T_{\bkg} \mca^\N$, $E \in \bg_\betab$, we deduce
\begin{equation}\label{eqn_Xleftinv}
    X_h = \rho(qEq^{-1})h \,.
\end{equation}


We extend $\rho(\betab)\bkg \in T_{\bkg} \mca^\N$ to a left-invariant vector field 
 on $\mca^\N$, which we denote by $X_\betab$. 
From \eqref{eqn_Xleftinv} we have for all $ q\in \Bbeb$,  $h =q \cdot \bkg$, 
\begin{align}\label{eqn_DLq}
          (X_\betab)_{h} =  (dL_q)_{\bkg}  \left( \rho\left(\betab\right)\bkg \right)
					= \rho\left( q\betab q^{-1} \right) h\,.
\end{align}

\begin{lemma}
The vector field $X_\betab$  is a gradient vector field on $(\mca^\N, \gsym)$.
\end{lemma}

\begin{proof}
We denote by $\nablasym$ the Levi-Civita connection of  $(\mca^\N\!, \gsym)$.
Since $\mca^\N$ is simply-connected, it suffices to show that $\nablasym X_\betab$ is symmetric. 
Let $Y_1, Y_2$ be two left-invariant vector fields on $\mca^\N \simeq \Bbeb$, defined by 
$(Y_i)_{\bkg} = \rho(E_i)\bkg$, $E_i \in \bg_{\betab}$, $i=1,2$. Using that the Lie bracket 
of left-invariant vector fields and the Lie bracket of Killing field differ only by a sign,
see \cite[7.21]{Bss},
we have $[Y_1,Y_2]_{\bkg} = -\rho([E_1, E_2]) \bkg$.

By Koszul's formula applied to left-invariant vector fields, we obtain
\begin{align*}
    \gsym(\nablasym_{Y_1} X_\betab, Y_2)_{\bkg} - \gsym(\nablasym_{Y_2} X_\betab, Y_1)_{\bkg} 
		=&\,\, - \gsym(X_\betab, [Y_1,Y_2])_{\bkg}   \\
		=&\,\, \gsym(\rho(\betab)\bkg, \rho([E_1,E_2])\bkg)_{\bkg}\\
    =&\,\, 4  \tr \betab [E_1,E_2]\,,
\end{align*}
since $\betab$ is $\bkg$-self-adjoint. Using the basis $\bca_\betab$, it is clear by definition of $\bg_\betab$ that the endomorphisms in $[\bg_\betab, \bg_\betab] \subset 
\bg_\betab \subset \End(\ngo)$ consist of strictly lower triangular matrices. 
Thus the last expression vanishes and the claim follows.
\end{proof}

\begin{definition}[$\log\betab$-volume]\label{def_logbetavol}
If $\N$ is a non-abelian Lie group with background left-invariant metric $\bkg$, the $(\log \betab)$-volume is the unique smooth function 
$$
 \vbg: (\mca^\N,\gsym) \to \RR
$$ 
satisfying $\vbg   (\bkg) = 0$ and
\[
   \tfrac{-1}{4\tr(\betab^2)} \,  X_\betab = \nablasym \vbg .
\] 
If $\N$ is abelian we simply set $\vbg \equiv 0$.
\end{definition}

Given that the automorphism group $\Aut(\ngo)$ of the Lie algebra $\ngo$ acts
on $\mca^\N$, and the metrics within an orbit are pairwise isometric, the following justifies the naturality of the $\log \betab$-volume: 

\begin{lemma}\label{lem_vb_autinv}
The $\log \betab$-volume $\vbg : \mca^\N \to \RR$ is $\Aut(\ngo)$-invariant.
\end{lemma}

\begin{proof}
We may assume $\ngo$ is not abelian. It is equivalent to show that $X_\betab$ 
is orthogonal to the $\Aut(\ngo)$-orbits.  Thus, given $D\in \Der(\ngo)$, $h\in \mca^\N$, 
we have to show that 
\[
    \gsym(\rho(D)h, \rho(q \betab q^{-1})h)_h = 0\,,
\] 
$q \in \Bbeb$, $h=q \cdot \bkg$.
By left-invariance of $X_\betab$ and $\gsym$ we have
\[
    \gsym(\rho(D)h, \rho(q \betab q^{-1})h)_h 
		= \gsym( \rho(q^{-1} D q) \bkg, \rho(\betab) \bkg)_{\bkg} 
		= 4 \tr (q^{-1} D q) \betab = 0,
\]
by Proposition \ref{prop_beta}, \ref{item_trDbeta}.
\end{proof}

\begin{remark}
The $\log \betab$-volume can be computed more or less explicitly, as follows. Let $h\in \mca^\N$, written as $h = \bar h(q^{-1} \cdot, q^{-1} \cdot)$ with $q = \exp(E)$, $E\in \End(\ngo)$  lower triangular with respect to the basis of eigenvectors of $\betab$ used to define $\B_\betab$, with diagonal entries $E_{ii}\in \RR$. (Recall that $\exp : \bg_\betab \to \Bbeb$ is a diffeomorphism.) Then, 
\[
    \vbg(h) = -\frac{\sum_i \betab_i E_{ii}}{\sum_i \betab_i^2},
\]  
where $\betab_1 \leq \cdots \leq \betab_{\dim \ngo}$ are the eigenvalues of $\betab$.  This can be seen by noticing that $t \mapsto  \vbg(\exp(tE) \cdot \bar h)$ is linear.
\end{remark}

\section{The \texorpdfstring{$\beta$}{b}-volume of the \texorpdfstring{$\N$}{N}-orbits in \texorpdfstring{$M$}{M}}\label{sec_betavolM}


The setup of this section is as follows: $\N$ is a connected Lie group acting properly, freely and isometrically on a Riemannian manifold $(M, g)$,
giving rise to a Riemannian submersion $\pi_P:(M,g) \to (P:=M/\N,\gP)$: see \eqref{def_Riemsub}.  
We extend the definition of the 
$(\log \betab)$-volume from Definition \ref{def_logbetavol} 
to this more general setting by constructing a smooth map 
$h : M \to \mca^\N \cong \Sym^2_+(\ngo^*)$, identifying the vertical spaces of the $\N$-action
with $\ngo$ via evaluation of Killing fields. 
We also compute the first and second variation of the $\beta$-volume on $M$.

\begin{remark}
The results in this section can be generalised to the case of isometric actions with non-trivial isotropy groups, but we restrict ourselves to free actions to simplify the presentation. 
\end{remark}



Consider $\ngo$ as a Lie algebra of Killing fields on $M$ corresponding to the $\N$-action, with Lie bracket 
$\mu_\ngo$ given by the Lie bracket of (smooth) vector fields on $M$.  Since $\N$ acts  freely,
for each $p\in M$, evaluation of Killing fields at $p$ yields a linear isomorphism  
\begin{equation}\label{eqn_identif}
  \ev_p : \ngo \to \vca_p := T_p (\N \cdot p) \subset T_p M, \qquad \UK \mapsto \UK_p.
\end{equation}

\begin{definition}\label{def_hM}
The Riemannian metric $g$ on $M$ restricted to the vertical distribution $\ve$ with respect to 
$\pi_P$ gives rise to an $M$-parameterised family of scalar products on $\ngo$, 
\[
  \hM: M \to \Sym^2_+(\ngo^*)\,, \qquad p \mapsto \g_p := \ev_p^* \gF_p 
	= g_p (\ev_p \,\cdot \,, \ev_p \, \cdot\,).
\]
By composing with the diffeomorphism from $ \Sym^2_+(\ngo^*)$ to
$\Bbeb$, see  \eqref{eqn_diffeo_mca}, we also get a map
\[
    \qM : M \to \Bbeb, \qquad p \mapsto q_p, \quad q_p \cdot \bkg = \g_p.
\]
\end{definition}

It is clear that $h$ and $q$ are smooth maps.

By viewing $\N$ as a subgroup of $\Iso(M,g)$, we say that an isometry $f$ of $(M,g)$ normalises $\N$, if $f \N f^{-1} = \N$. Conjugation by $f$ gives rise to the Lie algebra map $\Ad_f \in \Aut(\ngo)$.


\begin{lemma}\label{lem_g_pAdx}
Let $f \in \Iso(M,g)$ be an isometry normalising $\N$. 
Then, for each $p\in M$,
\[
    \g_{f(p)} = \Ad_f \cdot \,  \g_p.
\]
\end{lemma}

\begin{proof}
Let $\UK\in \ngo$ be a Killing field on $(M,g)$ with
flow $(\varphi(t))_{t \in \RR}$, $\varphi(t) \in {\rm Isom}(M,g)$. 
The one-parameter group of isometries $t\mapsto  f \circ \varphi(t) \circ f^{-1}$ 
is the flow of the Killing field $\Ad_f \UK$. In particular,
\[
    (\Ad_f \UK)_{f(p)} =   \ddt\big|_0\, (f \circ \varphi(t))  (p) = ({\rm d} f)_p \UK_p.
\]
Thus,
\[
    \g_p (\UK,\UK) 
		= g(\UK_p, \UK_p)_p 
		= g(({\rm d} f)_p \UK_p, ({\rm d} f)_p \UK_p)_{f(p)} 
    = \g_{f(p)} (\Ad_f \UK, \Ad_f \UK).
\]
This shows the claim.
\end{proof}

Given an endomorphism $E\in \End(\ve)$ of the vertical distribution,  using \eqref{eqn_identif}
we define a  corresponding smooth family  of endomorphisms of $\ngo$ parameterised by $M$:
\begin{equation}\label{eqn_defE^ngo}
		E^\ngo : M \to \End(\ngo), \qquad p\mapsto E^\ngo_p := \ev_p ^{-1} \circ E_p \circ\ev_p.
\end{equation}
This in particular may be applied to $L_X \in \End(\ve)$, where $X$ is a horizontal vector field.


\begin{lemma}\label{lem_DXhqmu}
At a point $p\in M$ and for a horizontal vector field $X$ we have
\begin{align*}
   ({\rm d} h)_p X_p &= -\rho( L_{X_p}^\ngo) \g_p\,,  \\
    ({\rm d} q)_p X_p &= -L_{X_p}^\ngo q_p + R_{X_p}q_p\,,
\end{align*}
for some $R_{X_p} \in \sog(\ngo,\g_p)$.
\end{lemma}

\begin{proof}
We may assume that $X$ is basic.
Let $p \in M$ and let $\gamma(t)$ be the (horizontal) integral curve of $X$ with $\gamma(0)=p$.
For Killing fields $\UK,\VK \in \ngo$, considered as smooth vertical vector fields on $M$,
 we compute using Lemma \ref{lem:killeft}
\begin{align*}
  \big( ({\rm d} \hM)_p X_p \big) (\UK,\VK) 
   &= 
  ( X g(\UK,\VK))_p \\
  &= 
   g(\nabla_X \UK, \VK)_p + g(\UK, \nabla_X \VK)_p   \\ 
  &= 
   g(\nabla_\UK X, \VK)_p + g(\UK, \nabla_\VK X)_p  \\
  &= 
    g(L_X \UK, \VK)_p + g(\UK, L_X \VK)_p  \\
  &= -\big( \rho(L_{X_p}^\ngo) \g_p \big) (\UK,\VK)\,.
\end{align*}
On the other hand, we set $q(t) := q_{\gamma(t)}$, with $q(0) = q_p$ and $q'(0) = ({\rm d} q)_p X_p$. 
Using the formula for differentiating the
action $ (q\cdot \g)' = \rho(q' q^{-1})(q\cdot \g)$, and the fact that $q(t) \cdot \bkg = h_{\gamma(t)}$ for all $t$, we get
\[
  ({\rm d} \hM)_p X_p = \rho\left(  q'(0) q_p^{-1} \right) \g_p.
\]
This shows
\[
 \rho \big(L_{X_p}^\ngo +  q'(0) q_p^{-1} \big) \g_p = 0\,.
\]
The second formula follows immediately from this.
\end{proof}




We now extend the definition of the endomorphism $\betab$ and the $(\log \betab)$-volume to all of $M$.

\begin{definition}\label{def_beta}
We define the $(\log \beta)$-volume of the $\N$-orbits in $M$ by
\[
  \vbM : M \to \RR, \qquad  \vbM := \vbg \circ \g,
\]
where $\vbg$ is the $(\log \betab)$-volume of inner products on $\ngo$ (Definition \ref{def_logbetavol}). 

Inspired by \eqref{eqn_DLq}, in case $\ngo$ is not abelian, for each $p \in M$ we set
\[
     \bM_p :=  q_p \betab q_p^{-1} \in \End(\ngo)\, ,
\]
we also define $\beta\in \End(\ve)$ by
\[
   \bv_p:\ve_p \to \ve_p \,\,; \qquad \bv_p :=i_p \circ  \bM_p \circ (i_p)^{-1} 
\]
and we introduce the following convenient notation:
\[
    \betab^+ := \frac{\betab}{\tr (\betab^2)} + \Id_\ngo, \qquad \bM_p^+ :=  \frac{\bM_p}{\tr (\betab^2)}  + \Id_\ngo, \qquad \bv_p^+ :=  \frac{\bv_p}{\tr (\betab^2)} + \Id_{\ve_p}
\]
For abelian $\ngo$ we set $\betab^+ = \bM_p^+ = \Id_\ngo$, $\bv_p^+ = \Id_{\ve_p}$ for all $p\in M$.
\end{definition}

Notice that $\bM_p$ is $\g_p$-self-adjoint, and $\bv$ is $\gF$-self-adjoint.
Also, by the next lemma, the function $\vbM$ is invariant under isometries of $M$ normalising $\N$:

\begin{lemma}\label{lem_vbeta_Ginv}
Let $\G$ be a Lie group acting almost effectively, properly and
isometrically on $(M,g)$ with one orbit type, 
such that $\N$ is normal in $\G$. Then, $\vbM$ is $\G$-invariant.
\end{lemma}

\begin{proof}
Let $g\in \G$ and $p \in M$. By Lemma \ref{lem_g_pAdx} we know that 
\[
    h_{g\cdot p} = \Ad_g \cdot h_p,
\]
with $\Ad_g \in \Aut(\ngo)$. Hence, $\vbM(g\cdot p) = \vbM(p)$ by Lemma \ref{lem_vb_autinv}.
\end{proof}

It is important to understand how the endomorphisms of the vertical distribution vary horizontally. It turns out that this covariant derivative corresponds via 
the linear isomorphism $\ev_p:\ngo \to \ve_p$ 
to $D_X (\bM)$, where $D$ denotes the standard flat connection on the trivial vector bundle $M \times \End(\ngo)$ (cf.~\cite[p.117]{EscWng00}):

\begin{lemma}\label{lem_nablaXE}
Let $E\in \End(\ve)$ be $\gF$-self-adjoint and $X$ be a horizontal vector field. Then, 
\[
		(\nabla_X E)_p = \ev_p \circ (D_X (E^\ngo))_p \circ \ev_p^{-1}, \qquad \forall p\in M.
\]
\end{lemma}

\begin{proof}
We may assume that $X$ is basic.
Let $U\in \ngo$ be a Killing field, and recall that 
$[X,U] = 0$ by Lemma \ref{lem:killeft}. We first compute using covariant differentiation on $(M,g)$:
\begin{align*}
  g((\nabla_X E) U, U)_p 
	 =& \,\,  g(\nabla_X (E U)-E(\nabla_X U), U)_p\\
	  =& \,\,  X_p  \, g(E U, U) - g(E U, \nabla_X U)_p - g(E \nabla_X U, U)_p  \\
      =& \,\, X_p \, g(E U, U) - \, g(E U, L_X U)_p - \, g(E L_X U, U)_p .
\end{align*}
On the other hand, we use that $g( E_x U_x, U_x) = \g_x(E^\ngo_x U, U)$ for all $x\in M$, and compute in $\End(\ngo)$ by applying Lemma \ref{lem_DXhqmu}, the very definition of $\rho(E)h$ and that
$ L_{X_p}^\ngo$ is $h_p$-self-adjoint
\begin{align*}
  X_p \, g( E U, U) =& \,\, X_p \, \g(E^\ngo U, U) =  - \big( \rho(L_{X_p}^\ngo) h_p\big) (E^\ngo_p U, U) + h_p \big( (D_X (E^\ngo))_p U, U \big) \\
    =& \,\, h_p \big( (D_X (E^\ngo))_p U, U \big) + 2\, h_p( E^\ngo_p U, L_{X_p}^\ngo U)\,.
\end{align*}
Now, $2 \, h_p( E^\ngo_p U, L_{X_p}^\ngo U) = 2 \, g(E U, L_X U)_p$, because by definition all the objects involved are related via the identification $\ev_p$. Thus, the previous computations yield
\[
    g((\nabla_X E) U, U)_p = h_p \big( (D_X (E^\ngo))_p U, U \big),
\]
from which the claim follows immediately.
\end{proof}

In the particular case of $E=\bv^+$, $D_X E^\ngo$ can be computed more explicitly:

\begin{lemma}\label{lem_nablaXbeta}
For a horizontal vector field $X$ we have 
\[
     (\nabla_X \bv^+)_p = (D_X (\bM^+))_p = [ -L_{X_p}^\ngo + R_{X_p} ,\bM^+_p],
\]
where $R_{X_p}$ is defined in Lemma \ref{lem_DXhqmu}.
\end{lemma}

\begin{proof}
If $\ngo$ is abelian the claim is clear, so assume in what follows that this is not the case. It is enough to prove the above for the endomorphisms $\bv$ and $\bM$ (without the $^+$). Let $q(t) := q_{\gamma(t)}$ for $\gamma(t)$ be an integral curve of $X$ with $\gamma(0)=p$.  By Lemma \ref{lem_DXhqmu} we have
\[
  q(0) = q_p, \qquad    q'(0) \, q_p^{-1}=  -L_{X_p}^\ngo + R_{X_p} .
\] 
Since $\bM_{\gamma(t)} = \Ad_{q(t) q_p^{-1}} \big(\bM_p \big)=(q(t) q_p^{-1})\big(\bM_p \big)(q(t) q_p^{-1})^{-1}$, 
we have
\[
    \ddt\big|_0 \bM_{\gamma(t)}  = [-L_{X_p}^\ngo + R_{X_p}, \bM_p].
\]
This shows the claim.
\end{proof}

Recall that $\vbM$ is an $\N$-invariant function by Lemma \ref{lem_vbeta_Ginv}, thus it descends to a smooth function on $P=M/\N$.  We are now in a position to estimate its Laplacian on $P$: Recall
that we have the convention 
$$
 \Delta_P f=\tr \Hess_P (f)=\tr (\nablaP_\cdot (\nablaP f))= \divgP (\nablaP f)
$$ 
for a smooth function $f:P\to \RR$.

\begin{lemma}\label{lem_HessvbM}
For a horizontal vector field $X$ we have
\begin{align*}
  ({\rm d } \, \vbM) X =& \,\,  \tr \big(  L_X  \cdot (\bv^+ - \IdV) \big), \\
  \Hess_P (\vbM) (X,X)  \geq&  \,\,   \tr \big(  (\nabla_X L)_X  \cdot ( \bv^+ - \IdV) \big), 
\end{align*}
with equality if and only if $[L_X,\bv^+] = 0$. In particular,
\[
    \Delta_P ( \vbM )  \geq  \,\, \sum_{j=1}^{d} \tr 
     \big( \left( \nabla_{X_j} L\right)_{X_j}  \cdot  (\bv^+ - \IdV) \big), 
\]
with equality if and only if $[L_X, \bv^+] = 0$ for all horizontal vector fields $X$.
\end{lemma}

\begin{proof}
The claims are trivial if $\ngo$ is abelian, so let us assume it is not. 
We compute using chain rule, Definitions \ref{def_logbetavol}, \ref{def_beta} 
and Lemmas \ref{lem_gsym_h} and \ref{lem_DXhqmu}, at a point $p \in M$
\begin{align*}
   ({\rm d} \, \vbM) X 
  &=
  ({\rm d}  \, \vbg) ({\rm d} \hM) X \\
  &= 
     \tfrac{-1}{4\tr(\betab^2)}  \, \gsym( (X_\betab)_h, -\rho(L^\ngo_{X}) h) \\
  &= 
       \tfrac{1}{4\tr(\betab^2)}  \, \gsym( \rho(\bM) h, \rho(L^\ngo_{X}) h)  \\  
  &=
     \tfrac{1}{\tr(\betab^2)}  \, \tr((\bM) L^\ngo_{X} ) \\ 
  &= 
      \tr \big( L_X (\bv^+ - \IdV) \big).
\end{align*}
Regarding the Hessian, we assume without loss of generality that $X$ is basic and 
$\nabla^P_X X = 0$ at the point $p$. Then, using the previous formula 
and Lemma \ref{lem_nablatrver} we deduce
\begin{align*}
    \Hess_P (\vbM) (X,X) =& \,\,  \la \nablaP_X \nablaP \vbM, X \ra =  X \la \nablaP \vbM, X\ra  \\
      =& \,\,  \tr \big( (\nabla_X L)_X (\bv^+ - \IdV) \big)  +  \tr \big( L_X  \nabla_X (\bv^+ - \IdV) \big). 
\end{align*}
The lemma will follow once we show that $ \tr \big( L_X  \nabla_X (\bv^+ - \IdV) \big) \geq 0$. Using that $\tr (\betab^2)$ is constant, and $\nabla_X \IdV = 0$ as well, this amounts to show that
\[
      \tr \big( L_X \nabla_X \bv \big) \geq 0.
\]
Pulling everything back to $\ngo$ with the evaluation map $\ev_p$ and using Lemma \ref{lem_nablaXbeta} we get
\begin{align*}
  \tr \big( L_X \nabla_X \bv \big)_p = \tr \big( L_{X_p}^\ngo  [ -L_{X_p}^\ngo + R_{X_p} ,\bM_p] \big) = \tr \big([L_{X_p}^\ngo, R_{X_p}] \bM_p \big).
\end{align*}
Set $S:= q_p^{-1} L_{X_p}^\ngo q_p$, $R := q_p^{-1} R_{X_p} q_p$, and recall that $\betab = (q_p^{-1}) (\bM)_p q_p$. We have that $S$ is $\bkg$-self-adjoint, $R\in \sog(\ngo,\bkg)$, and $E:= -S + R \in \bg_\betab$. Thus, 
\[
    \tr \big([L_{X_p}^\ngo, R_{X_p}] \bM_p \big) = \tr \big([S, R] \betab \big) = \tfrac12 \, \tr \big[E,E^T\big] \betab \geq 0,
\]
by  Proposition \ref{prop_beta}, \ref{item_trEE^Tbeta}, with equality if and only if $[S,\betab] = 0$, equivalent to $[L_{X_p}, \bv] = 0$. 
\end{proof}

Finally, the following useful formula is a direct consequence of Lemma 
 \ref{lem_HessvbM} and \eqref{eqn_trLX}:

\begin{corollary}\label{cor_LXbv+}
For a horizontal vector field $X$ we have
\[
    \tr (L_X \bv^+) = \la \nablaP \vbM - N , X \ra.
\]
\end{corollary}



\section{Main setup and estimates}\label{sec_estimates}

In this section we introduce the main setup for proving Theorem \ref{thm_rigidity}. 
Our assumptions will be as follows: $(M,g)$ is a Riemannian manifold admitting a proper isometric action by a connected, non-semisimple Lie group $\G$ with a single orbit type and compact (smooth) orbit space $M/\G =: B$. It follows that $(M,g)$  must be complete (see Remark \ref{rem_complete}). 
We emphasize that we are making no curvature assumptions on $g$ at this point.

Let $\N$ be the nilradical of $\G$. We may of course assume that $\G$ acts effectively. Then, by Lemma \ref{lem_Nfree}, $\N$ acts on $(M,g)$ properly, isometrically and freely, thus the results from Section \ref{sec_betavolM} apply. The orbit space $P:= M/\N$ is  a smooth manifold, which might be non-compact. We endow both $B$  and $P$ with the respective quotient metrics $\gB$ and $\gP$ so that the quotient maps
\begin{equation}\label{eqn_PBsubm}
  \pi:(M,g) \to (B, \gB) , \qquad \pi_P: (M,g) \to  (P,\gP),
\end{equation}
are Riemannian submersions. 

Since $\N$ is normal in $\G$, the action of $\G$ on $M$ maps $\N$-orbits to $\N$-orbits, and therefore  induces an action of $\G/\N$ on $P$. Of course,  also the corresponding quotient map 
\begin{equation}\label{eqn_G/Nsub}
  (P,\gP) \to \big(P/(\G/\N) = B,\gB \big)
\end{equation}
is a Riemannian submersion, whose fibers are the $\G/\N$-orbits. Observe  that the mean curvature vector $N$ of the $\N$-orbits in $M$ is not only $\N$-invariant but $\G$-invariant as well. Thus, the corresponding vector field on $P$, also denoted by $N$, is $\G/\N$-invariant.

\begin{remark}
The mean curvature vector of the $\G$-orbits will in general be different from the mean curvature vector $N$
of the $\N$-orbits. Moreover, $N$ will in general not be horizontal with respect to the submersion \eqref{eqn_G/Nsub}.
\end{remark}

By Proposition \ref{prop_equivHelm} applied to the submersion \eqref{eqn_G/Nsub}, there is an equivariant generalised Helmholtz decomposition 
\begin{equation}\label{eqn_NAnsatz}
  N = - \nablaP \log v + \Nw, \qquad \divgP (v \Nw) = 0, \qquad v \in \cca_+^\infty(P)^{\G/\N}, \quad \Nw \in \Xg(P)^{\G/\N} .
\end{equation}

\begin{remark}
Recall for a smooth vector field $X$ on $P$ we have $\divgP (X)=\tr (\nablaP_\cdot X)$
and for a smooth function $f:P \to \RR$ we have $\Delta_P f= \divgP (\nablaP f)$. Note
\begin{align}\label{eqn_divpfX}
   \divgP(fX)=\langle \nablaP f,X\ra + f \divgP (X).
\end{align}
\end{remark}

Consider the smooth, $\G/\N$-invariant function
\[    
      \fx : (P,\gP) \to \RR, \qquad \fx := \unm \, \Vert \Nw \Vert^2.
\]
Using  the horizontal Ricci curvature equation \eqref{Richh} for the $\N$-submersion, we show

\begin{lemma}\label{lem_Delx}
The function $\fx \in \cca^\infty(P)^{\G/\N}$ satisfies $\nablaP \fx  \,\, = \,\,   \nablaP_{\Nw} \Nw $, and
 \begin{align*}
  \Delta_P \fx  \, \, = \, \,  
      &   
        -  \la \nabla \log v + \Nw, \nabla \fx \ra  + {\Einstein \ric_M(\Nw, \Nw)}  
        +  \left \Vert \nablaP \Nw \right\Vert^2   
        + 2 \, \Vert A_{\Nw}\Vert^2 
        + \Vert L_{\Nw} \Vert^2  \, .
\end{align*}
\end{lemma}

\begin{proof}
We first compute the gradient of $\fx$. Since $\N$ is unimodular, the mean curvature vector $N \in \Xg(P)$
is a gradient vector field by Lemma \ref{lem_Ngrad}. Thus, by \eqref{eqn_NAnsatz}  $\nablaP \Nw$ is symmetric
and for $Y \in \Xg(P)$ we deduce
\[
    \la \nablaP \fx, Y \ra = Y(\fx) = \la \nablaP_Y \Nw, \Nw \ra 
    = \la Y, \nablaP_\Nw \Nw \ra.
\]
Regarding $\Delta_P \fx = \divgP ( \nablaP \fx)$, \eqref{eqn_genricform} and the fact that $\nablaP \Nw$ is symmetric yield
\[
  \divgP \left(  \nablaP_{\Nw} \Nw\right) 
  = 
  \Vert  \nablaP \Nw \Vert^2 
  +  \Nw \left( \divgP (\Nw) \right)  
  +   \ric_P(\Nw, \Nw).
\] 
The $\N$-horizontal Ricci curvature equation \eqref{Richh} now gives 
\begin{align*}
     \ric_P(\Nw, \Nw) 
      = & 
       \,\,  {\Einstein \ric_M(\Nw, \Nw)} 
       + 2 \, \Vert A_{\Nw}\Vert^2 + \Vert L_{\Nw} \Vert^2   
      - \la  \nablaP_{\Nw} N, \Nw\ra  
\end{align*}
and for the last term we have
\begin{align*}
	 - \la  \nablaP_{\Nw} N, \Nw\ra =& \,\, - \la \nablaP_\Nw \Nw,\Nw\ra + \la \nablaP_\Nw \nablaP \log  v ,\Nw\ra \\
	 =& \,\, - \la  \nablaP \fx, \Nw \ra  +  \Nw  \la \nablaP  \log  v, \Nw\ra- \la \nablaP \log  v, \nablaP \fx \ra \\
	 =& \,\, -\la \nablaP \log v + \Nw, \nablaP \fx \ra   - \Nw(\divgP (\Nw)),
\end{align*}
where the last equality uses $\Nw(\log v) = -\divgP (\Nw)$, which follows from $\divgP(v\Nw) = 0$
and \eqref{eqn_divpfX}.
Putting all this together we obtain the stated formula.
\end{proof}

Recall  that by Lemma \ref{lem_vbeta_Ginv}, $\vbM$ is a $\G$-invariant function on $M$, thus it induces a $\G/\N$-invariant function on $P$. We come now to our first key estimate,  a consequence of the vertical Ricci curvature equation \eqref{Ricvv}:

\begin{lemma}\label{lem_Dellogvb-u}
We have that 
\begin{eqnarray*}
     \Delta_P \, \vbMv    \,\,   \geq 	 \,\,  - \left\la \nabla \log v +\Nw, \nabla \vbMv	  \right\ra 	{\Einstein - \la \Ric_M |_\ve, \bv^+ \ra }
     + 2 \,\la L_\Nw, \bv^+ \ra  + 2 \, \fx,
\end{eqnarray*}
and equality holds if and only if for all horizontal $X$
\[
    \left[  L_X, \bv^+ \right] = 0,
		\qquad (\bM)^+ \in \Der(\ngo) \quad \hbox{and} \quad A = 0 \,.
\]
\end{lemma}

\begin{proof}
Lemmas \ref{lem_divN} and \ref{lem_HessvbM} imply that 
\[
 \Delta_P \, \vbMv  - \divgP (\Nw) = \divgP ( \nablaP \vbM - N )  \geq  \sum_j \tr  \big( \left( \nabla_{X_j} L\right)_{X_j}   \bv^+  \big), 
\]
with equality if and only if $\left[ L_X, \bv^+ \right] = 0$ for all horizontal $X$.
Recall that the vertical Einstein condition \eqref{Ricvv} in endormorphism form reads as
\[
     \sum_{j=1}^{d}  \left( \nabla_{X_j} L \right)_{X_j} 
     =
     \Ricci^\ve   {\Einstein - \Ric_M|_\ve } + L_N + A^*A \,,
\]
where $\la (A^*A) U, U \ra := \la AU, AU\ra$ for vertical $U$.
Thus,
\begin{align*}\label{eqn_Deltalogvbu}
		\Delta_P \, \vbMv  - \divgP (\Nw)  \geq &  \,\, 
		 \la \Ricci^\ve {\Einstein - \Ric_M |_\ve } + L_N + A^* A , \bv^+ \ra   \nonumber \\
       =& \,\, 
			\la \Ricci^\ve ,\bv^+\ra 
			+ \la A^* A ,\bv^+\ra 
			 { \Einstein - \la \Ric_M |_\ve, \bv^+ \ra } 
			+ \la  L_N, \bv^+ \ra \,.  \nonumber
\end{align*}
By Proposition \ref{prop_GITestimate} 
we have 
the pointwise estimate $\la \Ricci^\ve, \bv^+\ra \geq 0$, with equality if and only if $(\bM)^+\in \Der(\ngo)$. Also, by Proposition \ref{prop_beta+>0},  $\bv^+$ is positive-definite at each point. Since   $A^*A$ is clearly positive semi-definite at each point, this yields $\la \bv^+,A^* A \ra \geq 0$, with equality if and only if $A$ vanishes identically. Thus, we may drop these terms in the right-hand-side. Regarding the last term, we write $N = 2\Nw - (\Nw + \nablaP \log v)$.  Corollary \ref{cor_LXbv+} and \eqref{eqn_NAnsatz} give
\begin{align*}
    \la L_N, \bv^+\ra 
    =& \,\, 2\, \la L_\Nw, \bv^+ \ra - \la \Nw + \nablaP \log v, \nablaP \vbM - N \ra\\
    =& \,\, 2\, \la L_\Nw, \bv^+ \ra  + 2 \, \fx + \Nw(\log v) - \la \Nw + \nablaP \log v, \nablaP \vbMv	  \ra.
\end{align*}
The lemma follows by combinig the above and using again $\Nw(\log v) = -\divgP(\Nw)$.
\end{proof}

\section{Proof of Theorem \ref{thm_ricneg}}\label{sec_negRic}



For a Riemannian manifold $(M^n,g)$ with an isometric $\G$-action we say that
\[
  \ric_g \leq 0  \quad (\hbox{resp.}~ =0)\quad \hbox{along }\G \, \hbox{-orbits,} 
\]
if $\ric_g(U,U) \leq 0$ (resp.~$=0$) for all $U \in T_p(\G\cdot p)$ and all $p \in M$. The main result of this section is the following:

\begin{theorem}\label{thm_rig_polarricn}
Let $(M^n,g)$ be a Riemannian manifold admitting a proper isometric action of a connected, unimodular Lie group $\G$
with non-trivial nilradical $\N$, a single orbit  type and compact orbit space. 
If $\ric_g \leq 0$ along $\N$-orbits, then the following hold:
\begin{enumerate}[(i),wide]
  \item $\ric_g = 0$ along $\N$-orbits;
    \item The horizontal distribution defined by the action of $\N$ on $M$ is integrable;
    \item The following conditions hold pointwise on $M$  for all $\N$-horizontal vector fields $X$:
\[
    [L_X, \bv^+] = 0, \qquad \textrm{and} \qquad (\bM)^+ \in \Der(\ngo).
\]
\end{enumerate}
\end{theorem}

\begin{proof}
Since $\G$ is unimodular we have $\Nw=0$ by Lemma \ref{lem_Nw=0}. Moreover,
the $\N$-vertical endomorphism $\bv^+ \in \End(\ve)$ is positive definite. Thus,
by Lemma \ref{lem_Dellogvb-u} and the Ricci curvature assumption, 
the function $f := \vbMv \in \cca^\infty(P)^{\G/\N}$ satisfies the estimate
\begin{eqnarray*}
     \Delta_P  f    + \la \nabla \log v, \nabla f \ra \geq  0.
\end{eqnarray*}
By \eqref{eqn_divpfX} we obtain
\[
    \divgP (v \nabla f) = v \big( \Delta_P f  + \la \nabla \log v, \nabla f \ra \big) \geq 0.
\]
Proposition \ref{prop_divgeq0} yields equality everywhere, and as a consequence we deduce that $\ric_g = 0$ along $\N$-orbits (since $\bv^+ > 0$). Moreover, items (ii) and (iii) hold by  Lemma \ref{lem_Dellogvb-u}.
\end{proof}

\begin{proof}[Proof of Theorem \ref{thm_ricneg}]
Assume that $\G$ is unimodular and acts on $(M^n,g)$ satisfying \eqref{as_G}. If $\ric_g < 0$ then by Theorem \ref{thm_rig_polarricn} the $\N$-orbits must be trivial. By Lemma \ref{lem_Nfree}, $\N$ itself must be trivial, hence $\G$ is semisimple.
\end{proof}

Observe that in the unimodular case, \eqref{eqn_NAnsatz} is nothing but the classical expression of the mean curvature vector in terms of volume element of the orbits, see Lemma \ref{lem_Ngrad}. Moreover:

\begin{lemma}\label{lem_Nw=0}
We have that $\Nw = 0$ if and only if $\G$ is unimodular.
\end{lemma}
\begin{proof}
The necessity follows from Lemma \ref{lem_Ngrad}: in this case, the potential for the gradient vector field $N$ is no only $\N$-invariant but $\G$-invariant as well.
Conversely, if $\Nw = 0$, the mean curvature vector of the $\N$-orbits can be written as $N = -\nablaP \log v$, with $v$ a $\G$-invariant function. Since $M/\G$ is compact, $v$ has critical points in $P$, and the $\N$-orbits corresponding to those points are minimal in $M$. Of course they are also minimal as submanifolds of the corresponding $\G$-orbits. By Lemma \ref{lem_Nminimal} below, $\G$ must be unimodular.
\end{proof}

\begin{lemma}\label{lem_Nminimal}
Let $\G/\Hh$ be a homogeneous space and let $\N$ denote the nilradical of $\G$. Then, $\G$ is unimodular if and only if the orbit $\N\cdot e\Hh \subset \G/\Hh$ is a minimal submanifold.
\end{lemma}

\begin{proof}
Set $p := e\Hh$, and let $X\in\ggo$ be a Killing field with $X_p \perp \N\cdot p$. Since $\ngo := \Lie(\N)$ is an ideal in $\ggo$, $\ad_\ggo X$ preserves it, and it is a well-known algebraic fact that 
\[
  \tr \ad_\ggo X = \tr (\ad_\ggo X)|_\ngo.
\]
(Indeed, $\ggo/\ngo$ is a unimodular Lie algebra.) Thus, Lemma \ref{lem_LX_Killing} and \eqref{eqn_trLX} yield
\[
    \la X,N\ra_p = - \tr L_{X_p} =  -\tr \ad_\ggo X,
\]
from which the lemma follows.
\end{proof}

\section{The \texorpdfstring{$\N$}{N}-horizontal distribution is integrable}\label{sec:rigidity}


The main goal of this section is to prove Theorem \ref{thm_rig_polar}, which contains the first rigidity results needed for the proof of Theorem \ref{thm_rigidity}.

\begin{theorem}\label{thm_rig_polar}
Let $(M^n,g)$ be an Einstein manifold with $\ric(g)=-g$ admitting a cocompact, proper, isometric action of a  Lie group $\G$ with non-trivial nilradical $\N$, with a single orbit type. Then:
\begin{enumerate}[(i),wide]
  \item The horizontal distribution defined by the action of $\N$ on $M$ is integrable;
  \item $\Nw$ is a parallel vector field on $P=M/\N$ (see \eqref{eqn_NAnsatz} for the definition of $\Nw$);
  \item  The following conditions hold pointwise on $M$  for all horizontal vector fields $X$:
 \[
    [L_X, \bv^+] = 0, \qquad (\bM)^+ \in \Der(\ngo),    \qquad L_\Nw = -\bv^+.
\]
\end{enumerate}
\end{theorem}

In order to prove this, we work under the setup introduced in Section \ref{sec_estimates}. Combining Lemmas \ref{lem_Delx} and \ref{lem_Dellogvb-u} we obtain the key estimate in the Einstein case:


\begin{lemma}\label{lem_divgradf-N}
Assume that $(M^n,g)$ is Einstein with $\ric(g) = -g$, and let  $f:= \vbMv + \fx \in \cca^\infty(P)^{\G/\N}.$
Then,
\[
   \Delta_P \, f \, \geq \, -\left\la  \nabla \log v + \Nw,  \nabla f \right \ra,
\]
with equality if and only if we have equality in Lemma \ref{lem_Dellogvb-u} and, in addition,
\begin{align*}
 \nablaP \Nw = 0 \qquad \textrm{and}\quad L_{\Nw}  = - \bv^+ .
\end{align*}
\end{lemma}

\begin{proof}
The Einstein condition and Proposition \ref{prop_beta}, \ref{item_trbv+}, imply
\[
	- \la \Ric_M |_\ve, \bv^+ \ra =  \tr \bv^+ = \Vert \bv^+ \Vert^2, \qquad \ric_M(\Nw,\Nw) = -2 \fx.
\]
Thus, combining Lemmas \ref{lem_Delx} with \ref{lem_Dellogvb-u} and dropping the non-negative terms  $\left \Vert \nablaP \Nw \right\Vert^2$ and $2 \, \Vert A_{\Nw}\Vert^2$  we obtain
\begin{align*}
\Delta_P f \,\,   \geq &   \,\,   -\left\la \nabla \log v + \Nw, \nabla f \right\ra   + \Vert \bv^+ \Vert^2
    + 2\, \la L_\Nw,  \bv^+ \ra\,  + \Vert L_\Nw \Vert^2 \\
    	= & \,\, -\left\la \nabla \log v + \Nw, \nabla f \right\ra + \Vert \bv^+  + L_\Nw \Vert^2,
\end{align*}
and the lemma follows.
\end{proof}

\begin{proof}[Proof of Theorem \ref{thm_rig_polar}]
Let $f = \vbMv + \fx\in \cca^\infty(P)^{\G/\N}$. By \eqref{eqn_divpfX},
Lemma \ref{lem_divgradf-N} and the fact that $v>0$, we obtain
\begin{align*}
    \divgP (v \, \nablaP f) =& \, \,  \la \nablaP v, \nablaP f \ra + v \,  \Delta_P f \\
		\geq & \,\,  v \, \la \nablaP \log v, \nablaP f \ra -  v \,  \la \nablaP \log v + \Nw, \nablaP f \ra \\
		=&  \,\,   -v \Nw(f) \, = \,   - \divgP(f v \Nw),
\end{align*}
where the last equality follows from \eqref{eqn_divpfX} and $\divgP (v \Nw) = 0$.  
Thus, the vector field $Z := v (\nablaP f  + f   \Nw) \in \Xg(P)^{\G/\N}$ satisfies 
\[
      \divgP (Z) \geq 0.
\] 
By Proposition \ref{prop_divgeq0} applied to the submersion \eqref{eqn_G/Nsub}, this implies that $\divgP (Z) \equiv 0$, and  equality must hold in all the above estimates. In particular, the following  hold pointwise on $M$:
\[
    [L_X, \bv^+] = 0, \qquad (\bM)^+ \in \Der(\ngo), \qquad A=0, \qquad \nablaP \Nw =0 ,  \qquad L_\Nw = -\bv^+.
\]
This shows the claims in (i), (ii) and (iii). 
\end{proof}

The following are some further consequences of the rigidity obtained in Theorem \ref{thm_rig_polar}:

\begin{corollary}\label{cor_Nw}
Under the assumptions of Theorem \ref{thm_rig_polar} we have: 
\begin{enumerate}[(i), wide]
  \item $\Nw$ is parallel on $P$;
  \item $\Nw(v) = 0$;
  \item $\Vert \Nw\Vert^2 =\tr \bv^+$;
  \item With respect to $TM= \ve \oplus \ho$ we have
\[
  \nabla \Nw = \left(\begin{array}{cc}
     -\bv^+ & \\ & 0 
  \end{array}\right).
\]
\end{enumerate}
\end{corollary}

\begin{proof}
The first claim follow from Theorem \ref{thm_rig_polar} (ii). In particular, $\divg_P \Nw = 0$. Using  \eqref{eqn_divpfX} and $\divgP(v \Nw) = 0$ it also follows that $\Nw(v) = 0$.
The third claim follows from Lemma \ref{lem_Delx}: We deduce $0=-2\fx +\Vert L_\Nw\Vert^2$, 
since all the other terms vanish. Now by Theorem \ref{thm_rig_polar} (iii)
we have $\Vert L_\Nw\Vert^2=\Vert \bv^+\Vert^2$, and  $\Vert \bv^+\Vert^2=\tr \bv^+$ by Proposition \ref{prop_beta}.
To show (iv), by Theorem  \ref{thm_rig_polar}
it is sufficient to show that the mixed terms of $\nabla \Nw$ vanish. So let 
$X \in \ho$ and $U \in \ve$. Then, since $A=0$ we have $\la \nabla_X \Nw, U\ra=0$. We assume 
furthermore that $X$ is basic and that $U$ is a (vertical) Killing field. 
Then using that $X$ and $\Nw$ are $\N$-invariant, by Lemma \ref{lem:killeft}
\[
 \la \nabla_U \Nw, X\ra = -\la \Nw, \nabla_U X\ra = -\la \Nw, \nabla_X U\ra =\la \nabla_X \Nw, U\ra =0,
\]
which shows (iv).
\end{proof}


\begin{corollary}\label{cor_bv+par}
We have that $\nabla_X \bv^+ = 0$ for any horizontal vector field $X$. 
\end{corollary}

\begin{proof}
By Lemma \ref{lem_nablaXbeta} it suffices to show that $[-L_X^\ngo + R_X, \bM^+] = 0$. The rigidity in Theorem \ref{thm_rig_polar} gives $[L_{X_p}, \bv_p^+] = 0$ for all $p\in M$, which is equivalent to $[L_{X_p}^\ngo, \bM_p^+] = 0$. Arguing as in the proof of Lemma \ref{lem_HessvbM}, we set $S:= q_p^{-1} L_{X_p}^\ngo q_p$, $R:= q_p^{-1}R_{X_p} q_p$, $E:= -S+R$, and we observe that $[S,\betab] = 0$. Thus, 
\[
  0 = \tr ([S,R] \betab) = \tfrac12 \, \tr [E,E^T]\betab.
\]
Proposition \ref{prop_beta}, \ref{item_trEE^Tbeta} yields $[E,\betab] = 0$ which is equivalent to $[-L_X^\ngo + R_X, \bM^+] = 0$. 
\end{proof}

\begin{corollary}\label{cor_nablaNLN=0}
We have that $(\nabla_\Nw L)_\Nw = 0$.
\end{corollary}

\begin{proof}
By Theorem \ref{thm_rig_polar} we know that $L_\Nw = -\bv^+$ and 
$0=\nablaP_\Nw \Nw = \ho \nabla_\Nw \Nw$. Thus,
\[
    (\nabla_\Nw L)_\Nw  = \nabla_\Nw (L_\Nw) - L_{\nabla_\Nw \Nw} 
    = -\nabla_\Nw \bv^+ - L_{\ho \nabla_\Nw \Nw} = 0,
\]
by Corollary  \ref{cor_bv+par}. 
\end{proof}

\begin{corollary}\label{cor_verEin}
The Ricci curvature of the $\N$-orbits satisfies
\begin{equation}\label{eqn_verEin}
    \Ric^\ve + \Id_\ve - \bv^+ =  L_{\nabla \log v} + \sum_{j=1}^d \left( \nabla_{X_j} L \right)_{X_j},
\end{equation}
where $\{X_j\}$ is a local horizontal orthonormal frame.
\end{corollary}

\begin{proof}
This follows at once from the rigidity in Theorem \ref{thm_rig_polar} applied to  \eqref{Ricvv}.
\end{proof}

Given  $p\in M$ we consider those Killing fields in $\ggo$ which are $\N$-horizontal at $p$:
\[
  \ag_p := \{  U\in \ggo : U_p \perp \N\cdot p \}.
\]
We next deduce from Theorem \ref{thm_rig_polar}  the standardness of all the $\G$-orbits: $[\ag_p, \ag_p] \subset \ag_p$ (cf.~ \cite{standard}). Of course $\ag_p$ depends on $p$, but remarkably it does not change if one moves horizontally. Indeed, let $P_p$ denote an integral submanifold through $p$ of the $\N$-horizontal distribution in $M$.

\begin{corollary}\label{cor_std_agconst}
Under the assumptions of Theorem \ref{thm_rig_polar},
$\ag_p$ is a Lie subalgebra of $\ggo$  for all $p\in M$. Moreover, $\ag_p = \ag_q$ for all $q\in P_p$.
\end{corollary}

\begin{proof}
Let $A_1, A_2 \in \ag_p$, $U\in \ngo$. 
Since $A_2$ is a Killing field, we have
\[
  -\la \nabla_{A_1} A_2, U\ra_p = \la A_1, \nabla_U A_2 \ra_p 
    =  \la A_1, [U,A_2] \ra_p + \la A_1, \nabla_{A_2} U \ra_p = 0,
\]
using that $\ngo$ is an ideal, and the integrability of the $\N$-horizontal distribution (Theorem \ref{thm_rig_polar}).
Hence, $[A_1,A_2]$ is a Killing field in $\ggo$ which is horizontal at $p$, and  this shows the first claim.

Let now $\gamma(t)$ be a horizontal geodesic contained in $P_p$, joining $p$ to $q\in P_p$. We have 
\[
  \gamma'(t) \la A, U \ra = \la \nabla_{\gamma'(t)} A, U \ra + \la A, \nabla_{\gamma'(t)} U \ra = 0,
\]
by the same reasoning as above, replacing $A_1$ by $\gamma'(t)$. This shows that $A\in \ag_p$ remains horizontal along $\gamma(t)$, from which it follows that $\ag_q = \ag_p$.
\end{proof}

Applying Theorem \ref{thm_rig_polar} to an Einstein solvmanifold (in which case $M/\G$ is just a point) yields

\begin{corollary}\label{cor_NS}
Let $(S, \gS)$ be an Einstein manifold with $\ric(\gS) = -\gS$, admitting an isometric and simply-transitive action of a simply-connected solvable Lie group $\G$ with nilradical $\N$. Then, there exists a unique $\G$-invariant, $\N$-horizontal vector field $\NbS$ on $S$  satisfying
\[
    (\nabla^S \NbS) |_{\ve_\N} = - \bS, \qquad (\nabla^S \NbS) |_{\ho_\N} = 0.
\]
where $\bS \in \End(\ve_\G) = \End(TS)$ was introduced in Definition \ref{def_beta} (applied to $(S,\gS)$).
\end{corollary}

\begin{proof}
We have $M=\G$ is this case. By Corollary \ref{cor_Nw} the vector field $\NbS$ may be taken to be 
$\Nw=N$, which is $\G$-invariant since $\G$ normalises $\N$. 
Uniqueness follows now immediately from the fact that $(S,\gS)$ does not admit any non-trivial parallel vector fields (otherwise it would split an $\RR$-factor, contradicting the Einstein condition).
\end{proof}

\section{The scalar curvature of the \texorpdfstring{$\N$}{N}-orbits}\label{sec_scaln}

In this section we continue working towards a proof of Theorem \ref{thm_rigidity}. 
We show that the induced metrics on the $\N$-orbits are pairwise isometric and  locally isometric to  a nilsoliton.
To that end, we obtain estimates for the Laplacian of the scalar 
curvature of the $\N$-orbits as a function on $P = M/\N$.  These are mainly based on the moment map formulation for the scalar curvature of nilmanifolds (see Proposition \ref{prop_ricmm} below), due to J.~Lauret. 

We first recall a number of properties of the Ricci curvature of left-invariant metrics on $\N$, using the notation from Section \ref{sec_newbetavol}.
By homogeneity, after evaluating at the identity and identifying $T_e \N \simeq \ngo$,  the Ricci curvature of left-invariant metrics on $\N$ and its trace can be viewed as smooth maps
\begin{equation}\label{eqn_Ricscaln}
    \Ric : \mca^\N \simeq \Sym^2_+(\ngo^*) \to \End(\ngo), \qquad \scal := \tr \Ric : \mca^\N \to \RR.
\end{equation}
Regarding the first variation of the  scalar curvature, we have the following well known formula (valid for any unimodular Lie group):

\begin{lemma}\cite{Jen70}\label{lem_dscal}
Given $h\in \mca^\N$, for any $E\in \End(\ngo)$ we have that
\[
    ({\rm d} \scal)_h (\rho(E) h) = 2\tr \Ric(h)E, 
\]
\end{lemma}

Recall here that $(\rho(E)h)(X,Y)=-h(EX,Y)-h(X,EY)$ which explains the factor $2$
instead of the expected factor $-1$. Note  by Lemma \ref{lem_gsym_h}
$$
(\nablasym \scal)_h=\tfrac{1}{2}\rho(\Ric(h))h=-h(\Ric(h) \, \cdot \,, \cdot ) =-\ric(h)
$$
using that $\Ric(h)$ is $h$-self-adjoint: see \cite[Proposition 4.17]{Bss}.

The Lie bracket $\mu_\ngo$ 
of $\ngo$ is an element of the vector space $V_\ngo := \Lambda^2(\ngo^*) \otimes\ngo$. Any inner product $h$  on $\ngo$ induces in a natural and obvious way an inner product $\ipp_h$ on $V_\ngo$: given any $h$-orthonormal basis $\{e_i\}$, an orthonormal basis for $\ipp_h$ is given by $\{(e^i \wedge e^j)\otimes e_k\}$. Moreover, there is a natural $\Gl(\ngo)$-action on $V_\ngo$  and a corresponding $\glg(\ngo)$-representation $ \tau : \glg(\ngo) \simeq \End(\ngo) \to \End(V_\ngo)$, given by 
\begin{equation}\label{eqn_action_Vn}
  q\cdot \mu(\cdot, \cdot) := q \mu(q^{-1} \cdot, q^{-1} \cdot)  , \qquad \tau(E) \mu(\cdot, \cdot) := E \mu(\cdot,\cdot) - \mu(E \cdot, \cdot) - \mu(\cdot, E\cdot),
\end{equation}
for $q\in \Gl(\ngo)$, $E\in \End(\ngo)$, $\mu \in V_\ngo$. Notice that $\tau(E)\mu = 0$ if and only if $E$ is a derivation of $\mu$. The $\Gl(\ngo)$-action induces an inclusion $\Gl(\ngo) \to \Gl(V_\ngo)$, and in this way, $\Gl(\ngo)$ acts on the space $\Sym^2_+(V_\ngo^*)$ of inner products in $V_\ngo$.

\begin{lemma}\label{lem_ipp_equiv}
The map $\Sym^2_+(\ngo^*) \to \Sym^2_+(V_\ngo^*)$, $h\mapsto \ipp_h$, is $\Gl(\ngo)$-equivariant, that is,
\[
  \lla \mu,\lambda \rra_{q\cdot h} = \lla q^{-1} \cdot \mu \,, \,\,  q^{-1} \cdot \lambda \rra_h,
\]
for all $q\in \Gl(\ngo)$, $\mu,\lambda \in V_\ngo$.
\end{lemma}
\begin{proof}
Let $\{e_i\}$ be an $h$-orthonormal basis. Then $\{q e_i\}$ is a $(q\cdot h)$-orthonormal basis. By definition of the action, 
\[
  (q\cdot h) \big( \mu(q e_i, q e_j), q e_k \big) = h \big( (q^{-1}\cdot \mu)(e_i,e_j), e_k \big),
\]
thus the structure coefficients of $\mu$ with respect to $(q\cdot h)$, are the same as those of $q^{-1} \cdot \mu$ with respect to $h$. The lemma follows.  
\end{proof}

An immediate consequence is the following description of the first variation of $\ipp_h$:
\begin{corollary}\label{cor_dipp}
The first variation of the map $\ipp : \Sym^2_+(\ngo^*) \to \Sym^2_+(V_\ngo^*)$ is given by
\[
    {\rm d} \ipp  (v) = - \lla \tau(E) \cdot, \cdot \rra_h - \lla \cdot, \tau(E) \cdot \rra_h, \qquad v = \rho(E) h.
\]
\end{corollary}

The main reason we are interested in the space of brackets $V_\ngo$  is the following key formula due to J.~Lauret, providing a real GIT moment map interpretation of the Ricci curvature of nilmanifolds:

\begin{proposition}\cite[Prop.~3.5]{Lau06}\label{prop_ricmm}
If $\ngo$ is a nilpotent Lie algebra then for  any $E\in \End(\ngo)$, $h\in  \Sym^2_+(\ngo^*)$, we have 
\[
    \tr \Ric(h)  E = \unc  \lla \tau(E) \mu_\ngo, \mu_\ngo  \rra_h  \, .
\]
\end{proposition}





This yields a very useful formula for the first variation of $\Ric(h)$:

\begin{lemma}\label{lem_dRic}
The first variation of $\Ric : \Sym^2_+(\ngo^*) \to \End(\ngo)$ satisfies
\[
   \tr  \big( ({\rm d} \Ric)(v) \big) E  = -\unm \, \left\Vert \tau(E) \mu_\ngo \right\Vert^2_h  ,\qquad v = \rho(E)h,
\]
for any $h$-self-adjoint $E\in \End(\ngo)$.  In particular, 
\[
 -   \tr  \big( ({\rm d} \Ric)(v) \big) E \geq  0,\qquad v = \rho(E)h,
\]
with equality if and only if $E\in \Der(\ngo)$.
\end{lemma}

\begin{proof}
This follows directly from Corollary \ref{cor_dipp} and Proposition \ref{prop_ricmm}, using the fact that $\tau(E)$ is $\ipp_h$-self-adjoint if $E$ is $h$-self-adjoint.
\end{proof}

On the manifold $M$ we also consider the Ricci curvature $\Ric^\ve_p$
of the $\N$-orbits $\N \cdot p$ as a family of endomorphisms of $\ngo$. With respect to the notation introduced in \eqref{eqn_Ricscaln}, we have
\begin{align*}
  & \Ric \circ \, h : M \to \End(\ngo), \qquad \Ric(h_p)=  \ev_p^{-1} \circ \Ric^\ve_p \circ\,\,\ev_p  .
\end{align*}
Here $\ev_p$ is simply the evaluation of Killing fields map \eqref{eqn_identif}, and $h: M \to \Sym^2_+(\ngo^*)$ is given as in Definition \ref{def_hM}. Taking traces yields
\begin{align*}	
  & \scal \circ \, h : M \to \RR,  \qquad \scal(h_p) := \tr \Ric^\ve_p = \scal^\ve_p.
\end{align*}

Since $\G$ acts on $M$ isometrically and preserving $\N$-orbits, $\scal^\ve \in \cca^\infty(M)^\G$, thus it induces a smooth function $\scal^\ve \in \cca^\infty(P)^{\G/\N}$, $P=M/\N$. Recall that vertical and horizontal
is meant with respect to $\pi_P:M \to P$.

The first main result of this section is the following estimate:

\begin{proposition}\label{prop_Deltascaln}
For a horizontal vector field $X\in \ho$ and a local horizontal orthonormal frame $\{X_j \}$ we have that
\begin{align*}
    \la L_X, \Ric^\ve \ra 
    =& 
    \,\,     -\tfrac{1}{2} \big\la \nablaP \!\scal^\ve , X  \big\ra, \\
   - \sum_{j=1}^d \left\la (\nabla_{X_j} L)_{X_j}, \Ric^\ve  \right\ra 
     \geq& 
     \,\,  \tfrac{1}{2}  \Delta_P \scal^\ve,
\end{align*}
with equality if and only if $L_X^\ngo \in \Der(\ngo)$ for all horizontal $X$.
\end{proposition}

\begin{proof}
Using chain rule, Lemma \ref{lem_DXhqmu} and Lemma \ref{lem_dscal} we compute:
\[
   \la \nablaP \!\scal^\ve, X \ra = ({\rm d} \scal) ({\rm d} \g)(X) = -({\rm d} \scal) (\rho(L_X^\ngo) \g) = -2\tr \Ric(\g) L_X^\ngo = -2\tr L_X \Ric^\ve.
\]
Assume without loss of generality that $\nablaP_X X = 0$ at the point. Differentiating again and using Lemma \ref{lem_nablatrver} we get
\[     
    -\la (\nabla_X L)_X, \Ric^\ve \ra  
    - \la L_X, \nabla_X \Ric^\ve\ra
            =
		 \tfrac{1}{2}\la \nablaP _X \nablaP \!\scal^\ve, X  \ra  . 
\]
The estimate will follow once we show that
\[
  0 \leq  \la L_X, \nabla_X \Ric^\ve\ra .
\]
 To see that, first notice that by Lemma \ref{lem_nablaXE} we have 
\[
  \nabla_X \Ric^\ve  = \ev_p \circ \big(D_X (\Ric\circ h) \big) \circ \ev_p^{-1},
\]
and by chain rule, $D_X (\Ric \circ h) = ({\rm d} \Ric) ({\rm d} \g) X = - ({\rm d} \Ric) (\rho(L_X^\ngo) \g)$.  Thus, 
\[
      \la L_X, \nabla_X \Ric^\ve\ra   = \tr L_X^\ngo (D_X (\Ric\circ h) \big) = -\tr L_X^\ngo \big(({\rm d} \Ric) \rho(L_X^\ngo)\g \big) \geq 0,
\]
by Lemma \ref{lem_dRic}, with equality if and only if $L_X^\ngo \in \Der(\ngo)$.
\end{proof}

An immediate consequence of Theorem \ref{thm_rig_polar} and Proposition \ref{prop_Deltascaln} is

\begin{proposition}\label{prop_Ric_N}
The Ricci curvature of the $\N$-orbits is given by 
$$
 \Ric^\ve = -\Id_\ve + \beta^+.
$$
 Moreover, $L^\ngo_Y \in \Der(\ngo)$ for all $\N$-horizontal $Y$. 
\end{proposition}

\begin{proof}
Since $[L_X,\bv^+]=0$ by Theorem \ref{thm_rig_polar}, we have equality in Lemma \ref{lem_HessvbM}:
\begin{equation}\label{eqn_DeltavbM=}
    \Delta_P ( \vbM )  =  \sum_{j=1}^d \tr  \big( \left( \nabla_{X_j} L\right)_{X_j}   (\bv^+ - \Id_\ve) \big).
\end{equation}
We trace \eqref{eqn_verEin} against $\Ric^\ve + \Id_\ve -  \bv^+$  
and use Lemma \ref{lem_HessvbM}, Proposition \ref{prop_Deltascaln} and \eqref{eqn_DeltavbM=}, thus obtaining
\begin{align*}
  0 &\leq \left\Vert \Ric^\ve +\Id_\ve - \bv^+  \right \Vert^2 \\ 
  &= \big\la L_{\nablaP \log v} + \sum_j (\nabla_{X_j} L)_{X_j} , \Ric^\ve + \Id_\ve -  \bv^+ \big\ra \\
    &\leq  -\Delta_P ( \tfrac{1}{2}\scal^\ve +\vbM)  
		- \big\la \nablaP (\tfrac{1}{2}\scal^\ve + \vbM ), \nablaP \log v \big\ra \\
    & = -v^{-1} \, \divgP \left(v \,  \nablaP (\tfrac{1}{2}\scal^\ve + \vbM) \right).
\end{align*}
Proposition \ref{prop_divgeq0} implies that  equality must hold everywhere. In particular, 
$\Ric^\ve =  -\Id_\ve + \bv^+ $ and $L_Y^\ngo \in \Der(\ngo)$ for all horizontal $Y$.
\end{proof}

\begin{corollary}\label{cor_nsoliton}
The $\N$-orbits are pairwise isometric and locally isometric to nilsolitons. 
\end{corollary}

\begin{proof}
We may assume that $\N$ acts effectively on $(M^n,g)$. By Lemma \ref{lem_Nfree}, this implies that the action is free, so that all $\N$-orbits are isometric to left-invariant metrics on a fixed Lie group $\N$. We now claim that $h(M)$ (Definition \ref{def_hM}) is contained in a single $\Aut(\ngo)$-orbit in $\Sym^2_+(\ngo^*)$, which clearly implies that the orbits are pairwise isometric. But since $M$ is connected, this is immediate from Lemma \ref{lem_DXhqmu} and the fact that $L_Y^\ngo \in \Der(\ngo)$ for all $\N$-horizontal $Y$, which holds thanks to Proposition \ref{prop_Ric_N}. Recall that $\Der(\ngo)$ is the Lie algebra of $\Aut(\ngo)$.

The fact that these metrics are locally isometric to a nilsoliton follows from Proposition \ref{prop_Ric_N} and Theorem \ref{thm_rig_polar}, (iii): indeed, after conjugating with $\ev_p$ we obtain
\[
	\Ric(h_p) = -\Id_\ngo + (\bM)^+, \qquad (\bM)^+ \in \Der(\ngo),
\]
which is the definition of a nilsoliton inner product \cite{soliton}.
\end{proof}

\begin{corollary}\label{cor_vngoconst}
The functions $\log v, \vbM :P \to \RR$ are constant and $N = \Nw$.
\end{corollary}

\begin{proof}
Tracing \eqref{eqn_verEin} and using \eqref{eqn_trLX}, Lemma \ref{lem_divN} and Proposition \ref{prop_Ric_N}, we obtain
\[
   0 = - \la \nablaP \log v, N \ra  - \divgP N = \Vert \nablaP \log v \Vert^2 + \Delta_P (\log v).
\]
Proposition \eqref{prop_divgeq0} yields $\nablaP \log v \equiv 0$, from which $v$ is constant on $P$.

Regarding $\vbM$, by Lemma \ref{lem_HessvbM} we know that for any $X\in \Xg(P)$,
\[
		\la \nablaP \vbM, X \ra = \tr L_X^\ngo (\bM^+ - \Id_\ngo)=\tfrac{1}{\Vert \betab \Vert^2}  \tr L_X^\ngo \bM.
\]
The latter vanishes by Proposition \ref{prop_Ric_N} and 
Proposition \ref{prop_beta}, \ref{item_trDbeta}; if $\ngo$
is abelian then by definition $\bM^+ - \Id_\ngo = 0 $ so the result also follows.
\end{proof}

\section{Proof of Theorem \ref{thm_rigidity}: Einstein submanifolds}\label{sec_Einsteinsubm}

In this section we complete the proof of Theorem \ref{thm_rigidity}.
Recall that, by Corollary \ref{cor_nsoliton},
the $\N$-orbits are pairwise isometric and locally isometric to nilsolitons
and that $N=\Nw$ by Corollary \ref{cor_vngoconst}. Moreover, by Theorem \ref{thm_rig_polar} and Corollary \ref{cor_Nw},
the mean curvature vector $N$ of the $\N$-orbits 
is parallel on $(P,\gP)$ with $\Vert N\Vert^2=\tr \bv^+=\Vert \bv^+\Vert^2>0$. 

\medskip

Since $N$ is $\N$-basic, it commutes with Killing fields in $\ngo$ by Lemma \ref{lem:killeft}.

\begin{definition}
Given the integrable distribution $\Eca \subset TM$  spanned by the Killing fields in $\ngo$ and
$N$, its maximal leaves are called the \emph{Einstein leaves}.
\end{definition}

 Let $p\in M$ and denote by $E$ the Einstein leaf of $\Eca$ through $p$, endowed with the induced Riemannian metric
$g^E :=g\vert_{TE}$. Note $\dim E=\dim \N+1$ (Lemma \ref{lem_Nfree}).

 \begin{theorem}\label{thm_Einsteinsub}
Let $(M^n,g)$ be an Einstein manifold as in Theorem \ref{thm_rigidity} with $\ric(g)=-g$.
Then, each Einstein leaf $(E,g^E)$ is complete and locally isometric to an 
Einstein solvmanifold with $\ric(g^E)=-g^E$.
Moreover, the Einstein leaves are minimal, pairwise locally 
isometric, and the orthogonal distribution to $\Eca$ is integrable.
If in addition $M^n$ is simply-connected, then the Einstein leaves are embedded, equidistant submanifolds
which are pairwise isometric and intrinsically homogeneous. Moreover,
$M^n= E \times P'$ with  $P=M^n/\N=\RR \times P'$.
 \end{theorem}

\begin{proof}
We first show that Einstein leaves $(E,g^E)$ are locally isometric to an Einstein solvmanifold and satisfy $\ric(g^E)=-g^E$.
The group $\N$ acts isometrically and with cohomogeneity-one on $E$ with unit-normal  $N_1:=N/\Vert N\Vert$.   Clearly, $N$ is the mean curvature vector of the $\N$-orbits in $(E,g^E)$. By Lemma \ref{lem_Nfree} we may assume that $\N$ acts freely. If we let $\gamma(t)$ be a unit speed geodesic orthogonal to the $\N$-orbits in $E$ with $\gamma(0) = p$, then for a small $\epsilon >0$ the $\epsilon$-tubular neighborhood $E_\epsilon := B_\epsilon(\N\cdot p)$ in $E$ around the orbit $\N \cdot p$ is diffeomorphic to $I \times \N$, 
where $I := (-\epsilon, \epsilon)$, and $\gamma(t)$ corresponds to $(t,e)$,  $e\in \N$ the identity element. Furthermore, the projection onto $\N$ induces a diffeomorphism
\[
	\varphi : I \times \N \to E_\epsilon,
\]
 and metrically we have $\varphi^* g^E = dt^2 + g^\N(t)$, where $g^N(t)$ is a curve of left-invariant metrics on 
$\N$ defined by their values at the identity,
 \[
 	g^N(t)_e = \hM_{\gamma(t)},
 \]
 identifying $T_e \N \simeq \ngo$.  Now, Theorem \ref{thm_rig_polar} yields 
 \[
 	\big(L_{N_1}^\ngo\big)_{\gamma(t)} 
	= 
	\Vert N \Vert^{-1} \cdot \big(L^\ngo_N\big)_{\gamma(t)} 
	= 
	- \Vert \bv^+ \Vert^{-1} \cdot \bM_{\gamma(t)}^+ =:  - D_t \in \Der(\ngo)
 \] 
 with $\ddt D_t = 0$ by Corollary \ref{cor_bv+par} and Lemma \ref{lem_nablaXE}. Hence, $D_t \equiv D \in \Der(\ngo)$ is constant.  Thus, by Lemma \ref{lem_DXhqmu}, the curve $\hM_{\gamma(t)}$ solves the linear initial value problem
\begin{equation}\label{eqn_ivphM}
 	\ddt \hM_{\gamma(t)} = -\rho(D) \hM_{\gamma(t)}.
\end{equation}
In other words, we have that
\[
	\hM_{\gamma(t)} = \exp(-t D) \cdot \hM_p.
\]
Note that if $\tilde \N$ denotes the universal covering of $\N$ then we can lift
the metric $\varphi^* g^E$ to a metric $dt^2+\tilde g_1(t)$ on $I \times \tilde \N$
with $\tilde g_1(t)=g^\N(t)$.

On the other hand, let $\Ss$ be the simply-connected solvable Lie group with Lie algebra 
$\sg := \RR \xi \ltimes \ngo$,  where $\ad_\sg(\xi):= D$, and  extend the inner product $\hM_p$ on $\ngo$ 
to an inner product on $\sg$ making $\xi \perp \ngo$ and $\xi$ unit norm. This induces a left-invariant metric $g^\Ss$ on $\Ss$, and $(\Ss,g^\Ss)$ is called the \emph{one-dimensional extension} of $(\tilde \N, \tilde g^\N(0))$ (by the derivation $\bM_p^+$, with constant  $\alpha = \Vert \bv^+ \Vert^{-1}$), see \cite[$\S$2]{HePtrWyl}. 
 Consider the cohomogeneity-one action of $\tilde \N$ on $(\Ss, g^\Ss)$ by left-multiplication and let 
$\tilde \gamma(t)$ be a unit-speed geodesic orthogonal to the $\tilde \N$-orbits, $\tilde \gamma(0) = e$. 
 Since $\bM_p^+$ is self-adjoint, by \cite[Prop.~2.7]{HePtrWyl} the second fundamental form of 
$\tilde \N \cdot \tilde \gamma(t)$ is also equal to $-D$ for $t=0$, and constant along $\tilde \gamma(t)$ under the canonical identifications $T_{\tilde \gamma(t)} (\tilde \N \cdot \tilde \gamma(t)) 
\simeq T_e \tilde \N \simeq \ngo$. It follows that the metric $g^\Ss$ is given by
\[
	g^\Ss = dt^2 + \tilde g_2(t),
\]
with $\tilde g_2(t)$ a curve of left-invariant metrics on $\tilde \N$ coming from a curve of inner products on 
$\ngo$ that also satisfy \eqref{eqn_ivphM} (cf.~also \cite{NikoAlek21}). Thus $\tilde g_1(t)=\tilde g_2(t)$ and
consequently $g^\Ss = \varphi^* g^E$. This shows that $(E,g^E)$ is locally isometric to the simply-connected solvmanifold $(\Ss, g^\Ss)$.

Regarding curvature, by Proposition \ref{prop_Ric_N},  $(\tilde \N, \tilde g_1(0))$ 
is a nilsoliton with derivation $\bM_p^+$. Therefore, \cite[Thm.~3.2]{HePtrWyl} yields $\ric(g^\Ss) = - g^\Ss$.

The mean curvature vector $N_E$ of an Einstein leaf $(E,g^E)$ is given by the sum of the component
of $N$ normal to $E$, which is zero by definition of $\Eca$, and the normal component of $\nabla_{N_1} N_1$, 
 which is zero by Corollary \ref{cor_Nw}.
Thus, the Einstein leaves are minimal submanifolds of $(M,g)$. 

The distribution in $M$ orthogonal to $\Eca$ corresponds under $\pi_P$ to the distribution in $P$ orthogonal to $N$. The latter is integrable because $ N$ is parallel in $(P,\gP)$.



We now discuss completeness. As above, we consider the normal, unit speed geodesic $\gamma$ in $E$
with $\gamma(0)=p$ and $\gamma'(0)=N_1(p)$. Clearly $\gamma$ is a geodesic in the
horizontal leaf $P_p=\pi_P^{-1}(\pi_P(p))$, since $N_1$ is parallel. Let $\bar \gamma(t):=\pi_P(\gamma(t))$. There are now three different cases:

Case (1): $\bar \gamma(t)$ is not injective.
Then it is periodic since it is an integral curve of the smooth vector field $\bar N_1$ on $P$.
Thus $E \simeq \N\cdot p \times S^1$ and $E$ is a closed, embedded submanifold of $M$ (recall that $\N\cdot p$  is a closed subset of $M$), hence complete.

Case (2): $\bar \gamma(t)$ is injective but not an embedded curve in $P$. Then, 
$E$ is an immersed, intrinsically complete (but not embedded) submanifold of $M$,
diffeomorphic to $\N \cdot p \times \RR$. The diffeomorphism is
realized by the map $(n \cdot p,t) \mapsto \{n\cdot \gamma(t):n \in \N\,,\,\,t \in \RR\}$. 
Note the induced metric $g^E=g\vert_{TE}$ is complete, since
$\N \cdot p$ is a closed subset and $\gamma(t)$ is defined for all $t \in \RR$.

Case (3): $\bar\gamma(t)$ is an injective, embedded 
geodesic in $P$. As in the second case, $E$ is diffeomorphic to $\N \cdot p \times \RR$,
but in this case it is a closed, embedded submanifold of $M$.

We assume now that $M$ is simply-connected. By Lemma \ref{lem_covering},
 $M$ is diffeomorphic to $\N \cdot p \times P_p$, since $P_p$ intersects $\N\cdot p$ only at $p$,
and $P$ is simply-connected. As a consequence of Lemma \ref{lem_covering}, $P_p$ is diffeomorphic to 
$P$ and $\N \cdot p$ must be diffeomorphic to $\N$ 
(otherwise $\N \cdot p$ would not be simply-connected).
Thus we deduce that $M =\N \times P$ and that $\N$ is simply-connected.

Since $(P,\gP)$ admits a non-vanishing parallel vector field
$ N_1$ and since $P$ is simply-connected, by the De Rham decomposition theorem we have that 
$(P,\gP)= (\RR^k,g_{\sf Eucl}) \times (\tilde P,\tilde g)$ isometrically, with 
$\bar N_1$ tangent to the Euclidean factor. It follows that we are  in case (3), thus
$E$ is a closed, embedded submanifold of $M$. Moreover,
using $P=\RR \times P'$ we obtain
$$
  M= \N \times P=\N \times \RR \times P' = E\times P'\,.
$$
Let  $E_1 \neq E_2$ be two Einstein leaves, $p_1 \in E_1$ and $c:p_1\leadsto p_2 \in E_2$ 
be a shortest curve in $(M,g)$ between 
$p_1$ and $E_2$ with $c(0)=p_1$. It intersects $E_2$ perpendiculary, and consequently it must
intersect $\N \cdot p_2$ and $\gamma_{p_2}$ perpendicularly too, $\gamma_{p_i}$ integral curves of $N_1$ 
with $\gamma_{p_i}(0)=p_i$, $i=1,2$. Since $c$ intersects $\N \cdot p_2$ perpendicularly in $p_2$ and 
since the horizontal leaf $P_{p_2}$ is
totally geodesic in $(M,g)$, $c$ is a geodesic in the horizontal
leaf $P_{p_1}=P_{p_2}$ being perpendicular to the parallel vector field $N_1\vert_{P_{p_1}}$ 
at the point $p_2$. 
Note $P_{p_1} =\RR \times P'_{p_1} $ isometrically, 
the flat factor corresponding to $N_1\vert_{P_{p_1}}$. Thus $c$ is a geodesic in $P'_{p_1}$
intersecting also $\gamma_{p_1}$ perpendiculary.
In particular $d^M(p_1,E_2)$ equals to the distance between $\gamma_{p_1} = E_{1} \cap P_{p_1}$
and $\gamma_{p_2} = E_{2} \cap P_{p_1}$ in $P_{p_1}$. 

It remains to show that $d^M(\tilde p_1,E_2)$ does not depend on $\tilde p_1 \in E_1$.
So let $\tilde c:\tilde p_1 \to \tilde p_2$ be
a shortest curve between $\tilde p_1 \in E_1$ and $E_2$. Then using the $\N$-action we may assume 
that $\tilde p_1 = \gamma_{p_1}(t_0)$ for some $t_0 \in \RR$. This shows that the distance 
between two Einstein leaves is constant.
\end{proof}

\begin{remark}
Integral curves of a parallel vector field might have different lengths.
This happens for instance for the Klein bottle or the Moebius strip. Therefore,
in case (2) we cannot conclude that all Einstein leaves are isometric.
\end{remark}

We denoted by $P_p$ the horizontal leaf of $\pi_P:(M,g)\to (P,\gP)$ through a point $p\in M$. 
Moreover, let $\N_{P_p} := \{ x\in \N\ : x \cdot  P_p = P_p \}$, $Z_{P_p} := \{ x\in \N  : x\cdot q = q, \, \forall q\in P_p\}$, be respectively the stabilizer and centraliser of the slice $P_p$ in $\N$. Then, 
$$
 \Pi_p := \N_{P_p}/Z_{P_p}
 $$
  is the so called \emph{polar group}, see \cite{GZ12}.  

\begin{lemma}\label{lem_covering}
The map $\pi_p:= \pi_P\vert_{P_p}:P_p \to P$ 
is a covering map with $P_p \cap \N\cdot p=\pi_p^{-1}(\pi_P(p))$.
Moreover, if $M$ is simply-connected then so is $P$ and 
 $P_p \cap \N \cdot p =\{p\}$.
\end{lemma}

\begin{proof}
By \cite[Thm.~A]{HLO} $P_p$ is complete ($M$ is complete),  and in addition $\pi_P$ is a local isometry. It is a well-known fact that this forces $\pi_p$ to be a covering map.  This shows the first claim.
To show the second claim, not that by the long exact homotopy sequence of
the fibration $F \to M \to P$ with $F\cong \N \cdot p$ and $P=M/\N$ we have that
$0 =\pi_1(M)\to \pi_1(P) \to \pi_0(F)=0$ is exact, showing that $P$ is simply-connected.
As a consequence, $\pi_p : P_p \to P$ is a diffeomorphism, 
thus the fiber $P_p \cap N \cdot p$ consist of exactly one point.
\end{proof}

We now describe some further consequences of Theorem \ref{thm_rigidity}. Firstly, we note that in some cases, the horizontal distribution cannot be integrable by topological reasons. For these spaces, invariant Einstein metrics are obstructed:

\begin{corollary}\label{cor_topobstr}
Let $S^1 \to Q \to B$ be a compact, principal $S^1$-bundle and $\G'$ a connected Lie group. Assume that both $Q$ and $\G'$ have finite fundamental group.  Then, $M^n=  \G' \times Q$ does not admit a $(\G'\times S^1)$-invariant Einstein metric with \nsc{}.
\end{corollary}

\begin{proof}
The nilradical $\N$ of $\G:= \G'\times S^1$ is given by $\N = \N' \times S^1,$ where $\N'$ is the nilradical of $\G'$. By assumption, $\N$ is connected and acts freely on $M$. 
The polar group $\Pi_p$ acts on $\N\times P_p$ via
\[
	y \cdot (x,q) := (x y^{-1}, y \cdot q), \qquad y\in \Pi_p, \quad x\in \N, \quad q\in P_p.
\]
By \cite[Prop.~1.3]{GZ12}, this action is properly disconuous. The action map $(n,q)\mapsto n\cdot q$ induces a diffeomorphism 
\[
	(\N \times P_p)/\Pi_p \simeq M.
\]
In particular, $\N\times P_p$ covers  $M$. Since $\pi_1(M)$ is finite but $\pi_1(\N)$ is 
infinite, this yields a contradiction.
\end{proof}

 Recall that for simply-connected $B$ we have infinitely many such principal $S^1$-bundles, provided that
 $H^2(B,\ZZ)\cong \ZZ^l$, with $l \geq 1$.  A concrete example is 
 $M^{2m+3}=\RR^2 \times S^{2m+1}$, $m \geq 1$.
Here the sphere $S^{2m+1}$ is a principal $S^1$-bundle over the complex projective
space $\CC\PP^m$, the Hopf bundle, and for 
$\G'$  we choose the group of the upper triangular matrices in $\Sl(2,\RR)$ with positive diagonal entries
mentioned in the introduction.
Note that this particular $M^{2m+3}$  carries in fact an Einstein metric with \nsc{} by \cite{BDGW2015,BDW2015}, since $\CC\PP^m$ admits a K\"ahler-Einstein metric with \psc{}.
 However, for principal $S^1$-bundles over arbitrary bases $B$ this is wide open.

We finally prove the two corollaries to Theorem \ref{thm_rigidity} mentioned in the introduction.

\begin{proof}[Proof of Corollary \ref{cor_noS1}]
Let $N = \ve_\G N + \ho_\G N$ be the decomposition of the mean curvature vector $N$ of the $\N$-orbits with respect to the orthogonal decomposition $TM = \ve_\G \oplus \ho_\G$ induced by the $\G$-action. The $\G$-invariant vector field $\ho_\G N$ induces a vector field on $B$, denoted with the same name. We claim that $\ho_\G N$ is a Killing field on $B$. Indeed, let $X,Y$ be vector fields on $B$ and lift them to basic horizontal vector fields $X,Y$ on $M$. Then,
\[
    \la \nabla^B_X (\ho_\G N), Y \ra = \la \nabla_X N, Y \ra - \la \nabla_X (\ve_\G N), Y \ra =- \la \nabla_X (\ve_\G N), Y \ra,
\]
since $N$ is parallel on $B$ by Theorem \ref{thm_rig_polar}. This is skew-symmetric in $X,Y$ by the  properties  of the $A$-tensor, thus $\ho_\G N$ is a Killing field. 

If $\ho_\G N \neq 0$ then the (compact) isometry group of $(B,g^B)$ 
has positive dimension, contradicting the assumption of no $S^1$-actions. Thus, $N$ is $\G$-vertical, and since $\dim \G = \dim \N + 1$, we have that $\ve_\G = \Eca$. It follows from Theorem \ref{thm_Einsteinsub} that the $\G$-orbits are Einstein.
\end{proof}

\begin{proof}[Proof of Corollary \ref{cor_negversusEinstein}]
Let $\G$ be a simply-connected solvable Lie group whose  nilradical  $\N$  has codimension one. Assume that $\G$ is a \emph{Ricci-negative Lie group} -- it admits left-invariant metrics with negative Ricci curvature --  and that $\N$ is non-abelian and  not an \emph{Einstein  nilradical} -- it does not admit a nilsoliton metric. Then, $M = B^d \times \G$ admits $\G$-invariant metrics with negative Ricci (the product of a left-invariant metric on $\G$ with negative Ricci, and any negative Ricci metric on $B$ \cite{Loh94}), but  no $\G$-invariant  Einstein metric, since by Corollary \ref{cor_nsoliton} such a metric would induce  a nilsoliton metric on the $\N$-orbits (here $\N$ is simply-connected and  acts freely).

The existence and abundance of such $\G$'s is well known. Their Lie algebra is a semi-direct product $\ggo \simeq \RR \ltimes \ngo$, determined by the nilradical $\ngo$ and a derivation $D\in \Der(\ngo)$. If $D$ is positive definite then by \cite{Hei74} $\G$ admits metrics with negative Ricci curvature (even with negative sectional curvature). Moreover, any two extensions of a  fixed $\ngo$ where the $D$'s do not share the same eigenvalues (up to scaling) yield non-isomorphic Lie groups. It follows that any $\ngo$ having at least two linearly independent semisimple derivations (i.e.~ of rank $\geq 2$), one of which has positive eigenvalues, yields infinitely many examples of Lie groups $\G$ admitting negatively curved left-invariant metrics. One can in addition choose $\ngo$ so that it is not an Einstein nilradical. The lowest dimension for such an $\ngo$ is  $7$ \cite{cruzchica}. From the classification of $7$-dimensional Einstein nilradicals in \cite{FC13} it follows that there are plenty of possibilities. Let us take for instance the one labelled $\ggo_{3.1(iii)}$ in \cite{FC13}, whose non-zero brackets in a basis $\{e_i\}_{i=1}^7$ are given by
\[
  [e_1,e_2] = e_4, \quad [e_1,e_3] = e_5, \quad [e_1,e_6] = e_7, \quad [e_2,e_5] = e_7, \quad [e_3,e_4] = e_7\, .
\] 
It is easy to check that $D_{a,b,c} = \diag(a,b,c,a+b,a+c,b+c,a+b+c) \in \Der(\ggo_{3.1(iii)})$ for any $a,b,c\in \RR$, so that $\rank \geq 3$. On the other hand, the basis $\{e_i\}_{i=1}^7$ is \emph{nice} \cite[Def.~3]{Nik11}, thus one can quickly apply the criteria from \cite[Thm.~3]{Nik11} to conclude that $\ggo_{3.1(iii)}$ is not an Einstein nilradical.

Examples in higher dimensions are obtained by considering the direct sum of the above $\ngo$ with an abelian Lie algebra (letting $D$ act as the identity on the abelian factor).

Regarding the case where $B^d$ does not admit any smooth $S^1$-action, applying Corollary \ref{cor_noS1}  we see that it is enough to consider groups $\G$ as above, with $\ngo\simeq \RR^m$ abelian, $m\geq 2$, and $D>0$ but not a multiple of the identity.
\end{proof}

\section{The mean curvature vector of the \texorpdfstring{$\N$}{N}-orbits is \texorpdfstring{$\G$}{G}-vertical} \label{sec_NGvert}


We now start working towards a proof of Theorem \ref{thm_alek}. But before assuming homogeneity of $M$  we will show that, under some additional assumptions on $\G$, the mean curvature vector $N$ of the $\N$-orbits is $\G$-vertical (Theorem  \ref{thm_X=0}). Our setup for this section is as follows: $(M^n,g)$, $\G$ and $\N$ are as in Theorem \ref{thm_rigidity}, and $\bv^+ \in \End(\ve)$ is as in Definition \ref{def_beta}.
In addition we make the following

\begin{assumption}\label{as_G_Einstein}
The Lie group $\G$ is completely solvable, admits a non-flat left-invariant Einstein metric $g^\G$, and  acts almost freely on $M$. 
\end{assumption}

Recall that a solvable Lie group $\G$ is called \emph{completely solvable}, if the eigenvaules of $\ad_\ggo X$ are real for all $X\in \ggo$. 

\begin{remark}\label{rmk_Gfree}
A Riemannian manifold $(\G, g^\G)$  as in Assumption \ref{as_G_Einstein} is called an \emph{Einstein solvmanifold}.  These provide a rich class of non-compact Einstein spaces, containing families depending on several continuous parameters, see e.g.~\cite{cruzchica} and the references therein. 

Recall that by \cite[Thm.~1.1]{Jbl2015}, $\G$ must be simply-connected and centerless. (Under the completely solvable assumption, this follows also from the fact that the exponential map $\exp : \ggo \to \G$ is a diffeomorphism.) In particular, $\G$ acts effectively on $M$: indeed, the ineffective kernel is a discrete normal subgroup, hence central. This implies that the isotropy groups $\G_p$ are compact, and since $\G$ is diffeomorphic to $\RR^n$, they must be trivial. Thus, the action of $\G$ on $M$ is also free. 
\end{remark}

The main goal of this section is to prove the following

\begin{theorem}\label{thm_X=0}
Under the assumptions of Theorem \ref{thm_rigidity} and \ref{as_G_Einstein}, the mean curvature vector $N$ of the $\N$-orbits in $M$ is $\G$-vertical.
\end{theorem}

Let $p \in M$ and set $S:=\G \cdot p$. Since $\G$ acts freely (Remark \ref{rmk_Gfree}), as a manifold, $S$ is diffeomorphic to $\G$. If $\gS$ is an Einstein metric on $S$, we let
  $\bS$ be the endomorphism of the $\N$-vertical distribution $\ve\vert_S$  defined 
by applying Definition \ref{def_beta} to $(S,\gS)$. Note that $\bS$ depends on
$\gS$.


\begin{proposition}\label{prop_NS}
Let $(M^n,g)$ be as in Theorem \ref{thm_rigidity},
let $p\in M$, set $S:= \G\cdot p$ and assume that \ref{as_G_Einstein} holds.  
Then, there is a $\G$-invariant Einstein metric $\gS$ on $S$ such that $\bS = \bv^+\vert_S$. In particular, 
there exists a unique $\G$-invariant, $\N$-horizontal 
vector field $\NS$ on $S$ such that 
\[
    \left(\nabla \NS\right)|_{\ve} = -\bv^+.
\]
Here $\nabla$ denotes the Levi-Civita connection of $g|_S$ (and not that of $\gS$).
\end{proposition}

\begin{remark}
It will be made clear in the proof of Proposition \ref{prop_NS} that both metrics $g$ and $\gS$ give rise to the same $\N$-horizontal distribution on $S$.
\end{remark}

\begin{remark}
By Proposition \ref{prop_Ric_N} and the fact that the action is free, the $\N$-orbits are nilsolitons.  Recall that the corresponding inner products on $\ngo$ are unique up to $\Aut(\ngo)$. Thus, it might seem natural to try to extend these to Einstein metrics on $\G$, by simply modifying the metric on the intersection of the $\G$-vertical and $\N$-horizontal distributions. However, this  is not always possible: some nilsoliton inner products on $\ngo$ do not extend to Einstein inner products on $\ggo$. Indeed, consider $\ggo = {\rm span}_\RR\{e_1,e_2,e_3,e_4\}$, with Lie bracket 
\[
	[e_1,e_2] = e_2, \qquad [e_1,e_3] = e_3, \qquad [e_1,e_4] = 2\, e_4, \qquad [e_2,e_3] = e_4. 
\]
The inner product on $\ngo_3 = {\rm span}_\RR\{e_2,e_3,e_4 \}$ with orthonormal basis $\{ e_2, e_3+ e_4, e_4 \}$ is not the restriction of an Einstein inner product, because $(\ad e_1)|_\ngo$ is not a normal operator \cite[Thm.~B]{Heb}. (Recall that any inner product on  the Heisenberg Lie algebra $\ngo_3$ 
yields a  nilsoliton left-invariant metric.)
\end{remark}

The proof of Proposition \ref{prop_NS}  requires the framework described in Sections \ref{sec_newbetavol} and \ref{sec_betavolM}, applied to the Lie group $\G$ instead of $\N$. Recall that by
Assumption \ref{as_G_Einstein} and Remark \ref{rmk_Gfree}, $\G$ acts freely on $M$.
The notation for the objects corresponding to $\G$ will be the same as that used for $\N$ in the previous sections, but with a bold font instead. Unless otherwise explicitly stated, for the rest of  the section we assume that $\N$ is non-abelian.

Let us extend  $\bkg \in \Sym^2_+(\ngo^*)$ to a background inner product $\bkgg \in \Sym^2_+(\ggo^*)$ on $\ggo$. The $\beta$-endomorphism associated to $\ggo$ (see Appendix \ref{app_beta}) is denoted by $\betabg$. By Proposition \ref{prop_beta+>0} it is related to the one associated with $\ngo$ via
\begin{equation}\label{eqn_bg_bn}
      \betabg^+|_{\ngo} = \betab^+, \qquad \betabg^+|_{\ngo^\perp} = 0,
\end{equation}
where $\ngo^\perp$ is the $\bkgg$-orthogonal complement of $\ngo$ in $\ggo$. The maps
\[
    \gg : M \to \Sym^2_+(\ggo^*), \qquad \qMg : M \to \B_{\betabg}, \qquad \bMg^+ : M \to \End(\ggo),
\]
and the tensor $\bvg^+ \in \End(\ve^\G)$ are defined as in Section \ref{sec_betavolM}, but using now
the identifications $\ggo \simeq \ve^\G_p=T_p(\G \cdot p)$ 
given by evaluation of Killing fields in $\ggo$.


 
The next  consequences of Theorem \ref{thm_rig_polar} uses for the first time that $\G$ is solvable. 
Recall that for a given  $p\in M$ we
consider those Killing fields in $\ggo$ which are $\N$-horizontal at $p$:
\[
	\ag_p := \{  U\in \ggo : U_p \perp \N\cdot p \}.
\]
In other words, $\ag_p$ is the $\gg_p$-orthogonal complement of $\ngo$ in $\ggo$.




\begin{corollary}\label{cor_bMgDer}
Suppose that $\G$ is solvable and acts almost freely on $(M,g)$. Then,
under the assumptions of Theorem \ref{thm_rig_polar}, for all $p\in M$ we have $\bMg^+_p \in \Der(\ggo)$.
\end{corollary}

\begin{proof}

By definition we have $\bMg^+_p = \qMg_p \betabg^+ \qMg^{-1}_p$. Since $\qMg_p \in \B_\betabg$,  we know that $\qMg_p$  preserves $\ngo$ by Proposition \ref{prop_beta+>0}, and of course $\qMg_p  |_\ngo = \qM_p$. Thus, by \eqref{eqn_bg_bn},
\[
    \big(\bMg^+_p\big) \big|_\ngo = q_p  \betab^+ q_p^{-1} = \bM^+_p, \qquad \big(\bMg^+_p \big) \big|_{\ag_p} = 0,
\]
with $\bM^+_p\in \Der(\ngo)$ by Theorem \ref{thm_rig_polar}.  

The above and Corollary \ref{cor_std_agconst} reduce the claim $\bMg^+_p \in \Der(\ggo)$ to proving 
that  for all $ A\in\ag_p$
\begin{equation}\label{eqn_[adA,beta]=0}
   D:= \big[(\ad_\ggo A)|_\ngo, \bM^+_p \big] = 0.
\end{equation}
To show that, recall that by the rigidity from Theorem \ref{thm_rig_polar} on $M$ we have $[L_{A_p}, \bv_p^+] = 0$, 
which yields $[L^\ngo_{A_p}, \bM^+_p] = 0$ on $\End(\ngo)$. Thus, by Lemma \ref{lem_LX_Killing},
 $D$ is a symmetric derivation of $\ngo$ with
\[
0 \leq  \tr D^2 = \tr \big[(\ad_\ggo A)|_\ngo, \bM^+_p \big] D = - \tr \bM^+_p \big[ (\ad_\ggo A)|_\ngo, D \big] = 0,
\]
by Proposition \ref{prop_beta}, \ref{item_trDbeta}. This shows the claim.
\end{proof}

Note that so far we have not used that $\G$ admits left-invariant Einstein metrics.

\begin{proof}[Proof of Proposition \ref{prop_NS}]
Let $\gS$ be a $\G$-invariant Einstein metric on $S = \G \cdot p$. At $p$ there is a corresponding inner product $\gg_S = \qMg_S \cdot \bkgg$ on $\ggo$. By Theorem \ref{thm_rig_polar} and Corollary \ref{cor_bMgDer}, both applied to the manifold $(S,\gS)$ (the orbit space for the $\G$-action is just a point), we have that 
\[
      \bMg^+_S := \qMg_S  \, \betabg^+ \, \qMg_S^{-1} \in \Der(\ggo).
\]
On the other hand, from those two results, this time applied to $(M^n,g)$, we also know that
\[
    \bMg^+_p  = \qMg_p \, \betabg^+ \, \qMg_p^{-1} \in \Der(\ggo).
\]
Thus, by Lemma \ref{lem_GITb1b2}, there exists $\ab  \in  \Aut(\ggo)$ such that 
\[
      \bMg^+_S = \ab \, \bMg^+_p \, \ab^{-1}.
\]
The inner product $(\ab \, \qMg_S) \cdot \bkgg$ gives rise to a $\G$-invariant metric $\tilde g^S$ on $S$ which is isometric to $\gS$, and in particular also Einstein. By construction, it is clear that the corresponding  $\beta$-endomorphism will satisfy  $\tbMg^+_{S}  = \bMg^+_p$. From Proposition \ref{prop_beta+>0}, it follows on one hand that both $\tilde g^S$ and $g|_S$ define the same $\N$-horizontal distribution on $S$ (this is the kernel of the endomorphism $\tbMg^+_S$). On the other hand we also have $\bS = \bv^+|_S$ as desired, by restricting to the $\N$-vertical space. 

Applying Corollary \ref{cor_NS} to $(S,\gS)$ we obtain a $\G$-invariant, $\N$-horizontal vector field $\NS$ on $S$ satisfying 
\[
	\nabla^S \NS \,  |_\ve = -\bv^+,
\] 
where $\nabla^S$ denotes the Levi-Civita connection of $(S,\gS)$. To conclude, we now claim that
\[
	\nabla^S \NS \,  |_\ve = \nabla \NS \, |_\ve.
\]
By $\G$-invariance, it suffices to prove this at $p$, and via $\ev_p$ \eqref{eqn_identif} we may work in $\ngo$ instead of $\ve_p$. Let $A\in \ag_p$ be a Killing field with $A_p = (\NS)_p$ and set $E:= (\ad_\ggo A)|_\ngo$.  Lemma \ref{lem_LX_Killing} reduces the claim to showing that 
\begin{equation}\label{eqn_S(adA)=-bM+}
		 - \bM^+_p = \unm (E+E^{T_g}) ,
\end{equation}
where the tranpose is with respect to $\g_p = \ev_p^* (g_p|_{\ve \times \ve})$, the inner product on $\ngo$ corresponding to $g$ at $p$. Using again Lemma \ref{lem_LX_Killing}, now applied to $(S,\gS)$, we know that 
\[
	-\bM^+_p = \unm (E+E^{T_S}), 
\]
transpose with respect to $\g^S_p:= \ev_p^* \, \gS_p$. By \eqref{eqn_[adA,beta]=0}, $E$ is a normal operator with respect to $\g^S_p$.  Since $\ggo$ is completely solvable, this implies that $E$ is self-adjoint and $E = -\bM^+_p$. Furthermore, by construction the two inner products $\g_p$ and $\g^S_p$ give the same $\beta$-endomorphism, thus if we write $\g_p = q \cdot \g^S_p$ for some $q\in \B_{\betab}$, then $[q,\bM^+_p] = 0$. Therefore, by \eqref{eqn_transposes} we have that $E^{T_g} = E^{T_S} = E$ and \eqref{eqn_S(adA)=-bM+} follows. 
\end{proof}



We now see how applying Proposition \ref{prop_NS} to each $\G$-orbit in $M$  gives rise to a smooth vector field on $M$, which  in addition induces a Killing field on $P$:

\begin{proposition}\label{prop_NS_M}
Let $(M^n,g)$ be as in Theorem \ref{thm_rigidity} and assume that \ref{as_G_Einstein} holds.
Then, there exists a unique smooth,  $\G$-invariant, 
$\G$-vertical,  $\N$-horizontal vector field $\NS$ on $M$ with
\[
    L_{\NS}   = - \bv^+.
\] 
Moreover, the corresponding vector field on $P$ is a Killing field. 
\end{proposition}

\begin{proof}
By Proposition \ref{prop_NS}, existence and uniqueness of $\NS$ are clear.
It is also evident that $\NS$ varies smoothly in $\G$-vertical directions. To prove that it is smooth as a vector field on $M$, let $p\in M$, let $\gamma(t)$ be an $\N$-horizontal curve with $\gamma(0) = p$ (in $P_p$), and let $A_t\in \ggo$ be Killing fields defined by 
\[
  -(A_t)_{\gamma(t)} = (\NS)_{\gamma(t)}. 
\]
By  Corollary 
\ref{cor_std_agconst} we have $A_t \in \ag_p$, thus $A_t - A_0 \in \ag_p$ for all $t$.
We claim $A_t \equiv A_0$ is constant, thus in $\N$-horizontal 
directions, $\NS$ behaves like a Killing field in $\ggo$.
To see that, notice that by definition of $\NS$ and Lemma \ref{lem_LX_Killing}, the second fundamental form 
of $\N\cdot \gamma(t)$ satisfies
\begin{align}\label{eqn_adAt}
     -\la  \bv^+ V, V\ra_{\gamma(t)} = \la  L_{\NS} V, V \ra_{\gamma(t)} =  
-\la [A_t, V], V \ra_{\gamma(t)}\,.
\end{align}
Since $\ggo$ is completely solvable and $\ag_p$ is a subalgebra of $\ggo$ by
Corollary  \ref{cor_std_agconst}, the eigenvalues of $(\ad_\ggo A_t)|_\ngo$ are all real.    
Moreover, by \eqref{eqn_[adA,beta]=0} 
the symmetric endomorphism $\bM^+_{\gamma(t)}$ 
and  $(\ad_\ggo A_t)\vert_\ngo$ commute.
Thus, by \eqref{eqn_adAt} we get $\bM^+_{\gamma(t)}=(\ad_\ggo A_t)\vert_\ngo$ for all $t$,
which shows that
$(\ad_\ggo A_t)\vert_\ngo$ depends smoothly on $t$.
Differentiating this indentity, by using Lemma \ref{lem_nablaXE} and
Corollary \ref{cor_bv+par}, we conclude
 that $(\ad_\ggo A_t)|_\ngo \equiv (\ad_\ggo
 A_0)|_\ngo$. Thus, $A_t - A_0$ centralises $\ngo$. Given that the nilradical $\ngo$ is its own 
centraliser, $A_t - A_0 \in \ngo \cap \ag_p=\{0\}$. It follows  $A_t \equiv A_0$, thus $\NS\vert_{P_p}=-\A_0
\vert_{P_p}$ and in particular, $\NS$ is smooth on $M$. Another consequence 
is that the vector field on $P$ corresponding to $\NS$ is a Killing field, 
since  by Corollary \ref{cor_std_agconst}
$A_0\vert_{P_p}$ is a Killing field on $P_p$ (being locally isometry to 
$P$).
\end{proof}


We now apply the Bochner technique to prove the main result of this section:

\begin{proof}[Proof of Theorem \ref{thm_X=0}]
We  will show that $N=\NS$, where $\NS$ is the vector field from Proposition \ref{prop_NS_M}.
To that end, consider the horizontal Einstein equation \eqref{Richh}
for the $\N$-submersion evaluated at $N - \NS$:
\[
    - \Vert N - \NS \Vert^2 = \Ricci_P(N - \NS, N - \NS)  = \Ricci_P (\NS,\NS),
\]
where we have used that $A=0$ by Theorem \ref{thm_rigidity}, $L_{N-\NS} = -\beta^+ + \beta^+ = 0$ by Theorem \ref{thm_rigidity} and Proposition \ref{prop_NS_M}, and the fact that $N=\Nw$ is parallel on $P$ (Corollaries \ref{cor_Nw} and \ref{cor_vngoconst}). Since $\NS$ is a Killing field (Proposition \ref{prop_NS_M}), we have the well-known Bochner formula
\[
    \Ricci_P(\NS, \NS) = \Vert \nabla \NS \Vert^2 - \Delta_P (\unm \Vert \NS\Vert^2),
\]
an immediate consquence of \eqref{eqn_genricform}, using that $\nabla \NS$ is skew-symmetric.
Combining both equations yields
\[
  \Delta_P (\unm \Vert \NS\Vert^2) = \Vert \nabla \NS \Vert^2  + \Vert N - \NS \Vert^2 \geq 0.
\]
By Proposition \ref{prop_divgeq0} applied to the $\G/\N$-action on $P$, we deduce that equality must hold everywhere on $P$ and in particular $N \equiv  \NS$.
\end{proof}


\section{New algebraic formulae for the Ricci curvature of a homogeneous space}\label{sec_newalgform}

In order to use the rigidity from Theorem \ref{thm_rigidity} in the homogeneous setting, in this section we obtain formulae that relate the Ricci curvature of a homogeneous space with that of the $\N$-orbits, under the assumption that the $\N$-action is polar. The main result is Proposition \ref{prop_ricformula}. Its main advantage compared to other known formulae is the fact that it does not require the Killing fields to be in a reductive complement.

\begin{lemma}\label{lem_ricformula}
Let $(M, g)$ be a homogeneous manifold and let $\bca:= \{E_i \}_{i=1}^n$ be Killing fields which at $p$ form an orthonormal basis of $T_p M$. Then, at $p$ we have 
\begin{equation}\label{eqn_ricformula}  
  \ric(X,X) = 2\, \sum_i \la \nabla_{E_i} X, [X,E_i] \ra + \Vert \nabla X \Vert^2  - \sum_i \la (\ad X)^2 E_i , E_i \ra -   \sum_i \la (\nabla_{E_i} E_i), \nabla_X X \ra,
\end{equation}
for any Killing field $X$. 
\end{lemma}

\begin{proof}
By Lemma \ref{lem_R(X,Y)} 
for Killing fields $X,Y$ we have $ R(X,Y) = [\nabla X, \nabla Y] + \nabla [X,Y]$.
Using this, we compute:
\begin{align*}
 \lefteqn{ \ric(X,X)=}& \\
  =& \,\, \sum_i \la [\nabla X, \nabla E_i] E_i, X \ra + \la \nabla_{E_i} [X, E_i], X \ra \\
    =& \,\, \sum_i  -\la \nabla_{E_i} E_i, \nabla_X X \ra + \la \nabla_{E_i} X, \nabla_X E_i\ra - \la E_i, \nabla_X [X, E_i] \ra \\
    =& \,\,  - \sum_i \la \nabla_{E_i} E_i, \nabla_X X \ra + \sum_i \la \nabla_{E_i} X, [X,E_i] + \nabla_{E_i} X \ra - \la E_i, [X,[X,E_i]] + \nabla_{[X,E_i]} X\ra \\
    =& \,\, -\sum_i \la \nabla_{E_i} E_i, \nabla_X X\ra + 2\, \sum_i \la \nabla_{E_i} X, [X,E_i] \ra + \Vert \nabla X\Vert^2 - \sum_i \la (\ad X)^2 E_i, E_i \ra.
\end{align*}
This shows the claim.
\end{proof}

\begin{remark}
For a given homogeneous space $\F/\Hh$, with reductive decomposition $\fg = \hg \oplus \mg$, if all the $E_i's$ were in $\mg$ then the fourth term would be $\la \mcv, \nabla_X X \ra = -\la [\mcv,X],X\ra$. 
Here, $\mcv= -\sum_i (\nabla_{E_i}E_i)_{\mg}$ 
is  the mean curvature vector of the homogeneous space $(\F/\Hh, g)$ (\cite[(7.32)]{Bss})
satisfying for $Y \in \mg$
\begin{align}   \label{eqn_mcv}
 \langle \mcv,Y\rangle = \tr (\ad_{\fg} Y).
\end{align}
 In general this is not the case: consider an irreducible symmetric space of the non-compact type $M= \F/\K$. Choosing Killing fields which span the simply transitive $\A \N$-action ($\F = \K \A \N$ is an Iwasawa decomposition), it can be seen  that $\mcv_\G\neq 0$, 
since the group $\G=\A \N$ is not unimodular. However, $\F$ is unimodular, thus the mean curvature vector of $\F/\K$  vanishes.

In the Lie group case, that is $\hg=\{0\}$, it is easy to see that
\eqref{eqn_ricformula}  gives \cite{Bss}[(7.38)]. The last terms
in both formulae agree by the above. Moreover setting $\ad(X)=A+S$, $A$ skew-symmetric, $S$ symmetric,
we have $\nabla X=-A +J$, with $\la J\cdot Y,Z\ra=\tfrac{1}{2}\la [Y,Z],X\ra$. A computation
shows that the first three terms in \eqref{eqn_ricformula}  equal to $-\tr S^2 -\tr J^2$ and it is
easy to see that this equals to the first three terms in  \cite{Bss}[(7.38)].
\end{remark}

\begin{corollary}\label{cor_ricformula}
Let $(\N,g^\N)$ be a nilpotent Lie group with left-invariant metric, and $U\in \ngo$ a Killing field. Then, for any set of Killing fields $\{E_i\}$ forming an orthonormal basis at $p \in \N$, we have
\[
  \ric^\N(U,U) = 2\, \sum_i \la \nabla_{E_i} U, [U,E_i] \ra + \Vert \nabla U \Vert^2.
\]
\end{corollary}

\begin{proof}
The third term in \eqref{eqn_ricformula} equals $-\tr (\ad_\ngo U)^2$, and vanishes because $\ad U$ is nilpotent. Moreover, letting $V\in \ngo$ with $V_p = (\nabla_X X)_p$ and using that $\{E_i\}$ spanns $\ngo$, the fourth term in \eqref{eqn_ricformula} equals 
   $ \tr \ad_\ngo V  = 0$ by \eqref{eqn_mcv}.
\end{proof}

We say a connected Lie subgroup $\N \leq \F$ is \emph{nilpotently embedded}, if for all $U\in \ngo$, $\ad_\fg U \in \End(\fg)$ is a nilpotent endomorphism.

\begin{proposition}\label{prop_ricformula}
Let $(M = \F/\Hh,g)$ be a homogeneous space, $\N\leq \F$ a nilpotently embedded subgroup acting almost freely and polarly on $\F/\Hh$.
Given $p\in M$, let $\{E_i\}_{i=1}^n = \{ U_r\} \cup \{ Y_k\}$ be any set of Killing fields in $\fg$ which at $p$ form an orthonormal basis of $T_p M$, with $\{ U_r\}$ a basis of $\ngo$. Then, we have that
\[ 
    \scal^\ve(p) - \sum_r \ric_g(U_r, U_r)_p =  \sum_i \la \nabla_{E_i} E_i, N \ra_p + \sum_{i,r} \la [U_r,[U_r, E_i]\, ] , E_i \ra_p.
\]
Here, $\scal^\ve(p)$ denotes the scalar curvature of the orbit $\N \cdot p$ with the induced metric, and $N$ its mean curvature vector. 
\end{proposition}

\begin{proof}
For any $U_r\in \ngo$, Lemma \ref{lem_ricformula} and Corollary \ref{cor_ricformula} (the latter applied to the orbit $\N\cdot p$, which is locally isometric to a left-invariant metric on $\N$) give
\begin{align*}
\ric_g(U_r,U_r) +& \sum_i \la [U_r, [U_r, E_i]\, ] , E_i \ra + \sum_i \la \nabla_{E_i} E_i, \nabla_{U_r} U_r \ra  = 2\, \sum_i \la \nabla_{E_i} U_r, [U_r,E_i] \ra + \Vert \nabla U_r \Vert^2 \\
 =& \,\,   2\, \sum_s \la \nabla_{U_s} U_r, [U_r,U_s] \ra + \sum_{s,t}\la \nabla_{U_s} U_r, U_t \ra^2  \\
  & \,\, + 2\, \sum_k \la \nabla_{Y_k} U_r, [U_r,Y_k] \ra + 2\, \sum_{s,k} \la \nabla_{Y_k} U_r, U_s \ra^2  \\ 
 =&  \,\,  \ric^\ve(U_r,U_r) + 2\, \sum_{k,s} \la \nabla_{Y_k} U_r, U_s\ra \la [U_r, Y_k], U_s \ra +  2\, \sum_{s,k} \la \nabla_{Y_k} U_r, U_s \ra^2  \\
 =& \,\,  \ric^\ve(U_r,U_r) + 2\, \sum_{k,s} \la \nabla_{Y_k} U_r, U_s\ra \la \nabla_{U_r} Y_k, U_s \ra.
\end{align*}
We have omitted any terms of the form $\la \nabla_{Y_l} U_r, Y_k \ra$ due to the polar assumption (the $Y_k$'s are not horizontal vector fields, but they are horizontal at $p$, and the expression under consideration is tensorial in those entries). Summing over $r$, we notice that the summands of the second term in the right-hand-side have two factors: one of them,
\[
  \la \nabla_{Y_k} U_r, U_s \ra = \unm\big( \la [Y_k, U_r], U_s \ra + \la [Y_k, U_s], U_r \ra \big),
\]
symmetric in $r,s$ (recall that $[U_r,U_s] \perp Y_k$ at $p$), and the other one, $\la \nabla_{U_r} Y_k , U_s\ra$,
 skew-symmetric in $r,s$ because $Y_k$ is a Killing field. Thus, this term vanishes when summing over all $r,s$, and we obtain  
\begin{equation}\label{eqn_proof_scalgn}
    \sum_r \ric_g(U_r, U_r) +  \sum_{i,r} \la [U_r, [U_r, E_i]\, ] , E_i \ra + \sum_{i,r} \la \nabla_{E_i} E_i, \nabla_{U_r} U_r \ra = \scal^\ve(p).
\end{equation}
Finally, 
for any Killing field $U\in \ngo$ we have 
$ \sum_r \la \nabla_{U_r} U_r, U\ra  = - \tr \ad_\ngo(U) = 0$ by \eqref{eqn_mcv}.
Thus, $(\sum_r \nabla_{U_r} U_r)_p = N_p$.
\end{proof}

An important consequence of this formula is the following estimate, a fundamental ingredient in the proof of Theorem \ref{thm_alek}:

\begin{proposition}\label{prop_estimate}
Let $(M^n,g)$ be as in Proposition \ref{prop_ricformula}, and choose a Killing field $A\in \fg$  with $A_p = -N_p$, $p:= e\Hh\in \F/\Hh$. Assume that $D:= \ad_\fg A \in \End(\fg)$ is diagonalisable with real eigenvalues $(\lambda_i)_{i=1}^{\dim \fg}$ and respective eigenspaces $\fg_{\lambda_i}$. Set 
\[
  \fg = \fg_-\oplus \fg_0 \oplus \fg_+, \qquad \fg_-  := \bigoplus_{\lambda<0}{\fg_\lambda}, \qquad\fg_+ := \ngo = \bigoplus_{\lambda >0}  \fg_\lambda, \qquad \fg_0 := \ker D, \qquad \sigma_+ :=\sum_{\lambda_i >0}\lambda_i.
\]
Then, the Killing fields $\{E_i\}_{i=1}^n$ from Proposition \ref{prop_ricformula} may be chosen so that
\[
    \sum_i \la \nabla_{E_i} E_i, N \ra_p \leq \sigma_+, \qquad \sum_{i,r}\la [U_r, [U_r,E_i]], E_i \ra_p =0.
\]
In particular, if $(M^n,g)$ is Einstein with $\ric(g) = -g$, then
\[
  \scal^\ve(p)+ \dim \ngo \leq \sigma_+
\]
and equality holds if and only if $(\fg_0 \oplus \fg_+)\cdot p= T_p M$.
\end{proposition}

\begin{proof}
Set $\mg:= {\rm span}\{ E_i : 1 \leq i \leq n\} \subset \fg$, with $E_i$'s yet to be determined. Notice that $\ngo\subset \mg$, by assumption on the set $\{ E_i\}$. If $D\mg \subset \mg$, then $\mg$ would be a sum of eigenspaces, and  by Lemma \ref{lem_727} we would clearly have
\[
    \sum_i \la \nabla_{E_i} E_i, N \ra_p = \sum_i \la [A,E_i], E_i \ra_p = \tr D|_\mg  \leq \sigma_+,
\]
with equality if and only if no eigenspaces with negative eigenvalue are contained in $\mg$.  This is of course equivalent to $\mg \subset \fg_0 \oplus \fg_+$, and also to $(\fg_0 \oplus \fg_+)\cdot p= T_p M$, for $\mg \cdot p = T_p M$.

Regarding the second formula, notice that 
\[
    \sum_{i,r}\la [U_r, [U_r,E_i]], E_i \ra_p = \sum_r \tr \, \proj_\mg\circ (\ad_\fg U_r)^2 |_\mg,
\]
where $\proj_\mg: \fg = \hg\oplus \mg \to \mg$ denotes the projection onto the second factor. Hence, it suffices to choose $\{E_i\}$ so that $\mg$ is $D$-invariant and  $\proj_\mg\circ (\ad_\fg U)^2 |_\mg$ is traceless for each $U\in \ngo$. 

Let $\{ F_i\}_{i=1}^{\dim \fg}$ denote an eigenbasis for $D$ with eigenvalues sorted in non-increasing order, so that in particular $\{ F_r\}_{r=1}^{\dim\ngo}$ spans $\ngo$. Notice that
\begin{equation}\label{eqn_[F_k+1,ngo]}
  [F_{k+1},\ngo] \subset {\rm span}\{ F_i : 1 \leq i \leq k \}.
\end{equation}
Indeed, for each eigenvector $F_r \in \ngo$ with  eigenvalue $\lambda_r > 0$, by the Jacobi identity we have that 
\[
    [D, [F_{k+1},F_r]] = [[D,F_{k+1}],F_r] + [F_{k+1}, [D,F_r]] = (\lambda_{k+1}+ \lambda_r)[F_{k+1},F_r],
\]
so either $[F_{k+1},F_r] = 0$, or it is an eigenvector with eigenvalue $\lambda_{k+1}+ \lambda_r > \lambda_{k+1}$.

 We now construct $\mg$ inductively. Set $\mg_0 = 0$, and for each $k\geq 0$  define 
\[
  \mg_{k+1} := \begin{cases} \mg_k,   & \hbox{if } F_{k+1} \in \hg + \mg_k \, ;  \\ 
  \mg_k\oplus \RR F_{k+1} , & \hbox{otherwise.}
  \end{cases}
\]
Notice that the dimension of $\mg_k$ increases by at most one in each step. There is a corresponding subset of indices $1 = i_1 < i_2 < \cdots < i_n$  with 
\[
  \dim \mg_{i_r} = r, \qquad  \hbox{and} \qquad \mg_{i_r} = \mg_{i_{r+1} - 1} \subsetneq \mg_{i_{r+1}},
\]
for each $r = 1, \ldots, n-1$. Observe that $i_r = r$ for all $r\leq \dim \ngo$, since $\ngo \cap \hg =0$ and $\ngo$ (which is non-trivial) is spanned by the first $\dim \ngo$ vectors in the eigenbasis $\{ F_i\}$.

We then set $\mg := \mg_{i_n}$.  By induction, one can show that 
\[
  \mg_k \cap \hg = 0, \qquad {\rm  span } \{ F_i : 1 \leq i \leq k\} \subset \hg \oplus \mg_k , 
\]
for all $k\geq 1$. In particular, $\fg = \hg \oplus \mg$. Moreover, using \eqref{eqn_[F_k+1,ngo]} we deduce that
\[
    [F_{k+1}, \ngo] \subset \hg \oplus \mg_k,  \qquad  \forall \, k\geq 1.
\]  
Refining this slightly, we can write it as
\[
    [F_{i_{r+1}}, \ngo] \subset \hg \oplus \mg_{i_r}, \qquad \forall \, r\geq 1.
\] 
Notice that $\mg$ need not be a reductive complement. (In fact, if $\fg$ is not solvable, then no reductive complement for the homogeneous space $\F/\Hh$ will contain $\ngo$.)  However, by construction we still have that $\mg$ is $D$-invariant, as it is a sum of $D$-eigenspaces. In addition,
\[
  (\ad_\fg U)^2 \, \mg_{i_{r+1}} \subset  \hg \oplus \mg_{i_r}, \qquad  \forall \, U\in \ngo, \quad  \forall \, r\geq 1.
\]
Thus, 
\[
    \proj_\mg \circ (\ad_\fg U)^2 \,  \mg_{i_{r+1}} \subset \mg_{i_r}, \qquad \forall \, r\geq 1.
\]
from which it follows that $\proj_\mg \circ (\ad_\fg U)^2 |_\mg $ is a nilpotent endomorphism of $\mg$, and in particular traceless. Hence $\mg$ satisfies the required conditions, and it is now enough to choose a basis $\{ E_i\}$ for  it so that it is orthonormal with respect to the inner product induced by $g$ under the isomorphism $\mg \simeq  \mg \cdot p \simeq T_p M$.
\end{proof}

\section{Semi-direct product of Einstein solvmanifolds}\label{sec_semidirect}

Let $\F/\K$ be a homogeneous space with effective presentation and global Levi decomposition $\F = \Ll \ltimes \Ss$, where $\Ll$ is a maximal connected semisimple subgroup with Iwasawa decomposition $\Ll = \K \A \N$, $\K$ a maximal compact subgroup, $\Ss$ the solvable radical, and $\Ll \cap \Ss = \{e\}$.  
The semi-direct product is defined by a Lie group homomorphism $\Phi : \Ll \to \Aut(\Ss)$ with
\[
	 \phi := d \Phi|_e : \lgo \to \Der(\sg)
	 \subset \End(\sg), \qquad \phi(Y) := (\ad_\fg Y)|_\sg, \quad Y\in \lgo,
\]
the corresponding Lie algebra homomorphism defining the semi-direct product $\fg = \lgo \ltimes \sg$. 


Since it acts simply-transitively on the symmetric space $\Ll/\K$, the solvable Lie group $\A\N$ admits a left-invariant Einstein metric. Assuming that $\Ss$ also admits a left-invariant Einstein metric $g^\Ss$, and that a certain compatibility condition for $\phi$ and $g^\Ss$  holds, our goal in this section is to prove that $(\A\N) \ltimes \Ss$ also admits a left-invariant Einstein metric. Moreover:

\begin{theorem}\label{thm_F/K_Ein}
Let $g^\Ss$ be a left-invariant Einstein metric on $\Ss$ defined by an inner product $g^\Ss_e$ on $\sg$, and assume that $\phi(\lgo)^T = \phi(\lgo)$, transpose with respect to $g^\Ss_e$. Then, $\F/\K$ admits an $\F$-invariant Einstein metric.
\end{theorem}

Let $\lgo = \kg \oplus \pg$ be a Cartan decomposition corresponding to a Cartan involution $\theta$. Then, $\fg = \kg \oplus (\pg \oplus \sg)$ is a reductive decomposition for $\F/\K$. Observe that, since $(\Ss,g^\Ss)$ is an Einstein solvmanifold, $\Ss$ is simply-connected by \cite[Thm.~1.1]{Jbl2015}.

\begin{lemma}\label{lem_phiKX}
After pulling-back $g^\Ss$ by an automorphism of $\Ss$, we may assume that 
\[
	\phi(K)^T = -\phi(K), \quad \forall K\in \kg, \qquad\quad \phi(X)^T = \phi(X), \quad \forall X\in \pg.
\]
\end{lemma}

\begin{proof}
The kernel of $\phi$ is an ideal in the semisimple Lie algebra $\lgo$. By working on a complementary semisimple ideal, we may assume without loss of generality that $\phi$ is injective. Thus, we have a Lie algebra isomorphism 
\[
	\phi : \lgo \to \phi(\lgo) \subset \End(\sg).
\] 
The assumption $\phi(\lgo)^T = \phi(\lgo)$ implies that the map $E \mapsto -E^T$  is a Cartan involution on $\phi(\lgo)$. Hence, the corresponding map on $\lgo$, given by 
\[
	 \tilde \theta : \lgo \to \lgo, \qquad \tilde \theta(Y) :=  -Y', \qquad \hbox{where} \quad  \phi(Y') = \phi(Y)^T,
\]
is a Cartan involution of $\lgo$. By definition, it satisfies 
\begin{equation}\label{eqn_phitildetheta}
  \phi(\tilde \theta (Y)) = -\phi(Y)^T, \qquad Y\in \lgo.
\end{equation}
By uniqueness of Cartan involutions \cite[Cor.~6.19]{Knapp2e}, $\theta$ and $\tilde \theta$ are conjugate by an inner automorphism $\Ad_\Ll(x) =: a \in \Aut(\lgo)$, $x\in \Ll$. The corresponding automorphism  $\Ad_\F(x)$ of $\fg$  preserves the Levi decomposition, and is given by
\[
		\Ad_\F(x) |_\lgo =  \Ad_\Ll(x) = a, \qquad \Ad_\F(x)|_\sg = \Phi(x) =:q.
\]
Here we view $q = \Phi(x) \in \Aut(\sg)$ under the natural isomorphism $\Aut(\Ss)\simeq \Aut(\sg)$, using the fact that $\Ss$ is simply-connected. 

We claim that the left-invariant Einstein metric on $\Ss$ defined by the inner product $q \cdot g^\Ss_e$ on $\sg$ satisfies the required properties. To see that, we first notice that 
\begin{equation}\label{eqn_phiKX1}
	\phi(a^{-1} K) = -\phi(a^{-1} K)^T, \quad \forall K\in \kg, \qquad \quad \phi(a^{-1} X) = -\phi(a^{-1} X)^T, \quad \forall X\in \pg.
\end{equation}
Indeed, if $K\in \kg$ then $K = \theta K = a \tilde \theta a^{-1} K$, from which $a^{-1} K $ is fixed
 by $\tilde \theta$. Thus, by \eqref{eqn_phitildetheta},
\[
	\phi(a^{-1} K) = \phi \big(\tilde \theta a^{-1} K \big) = 
  -\phi\left(a^{-1}K \right)^T,
\]
and analogously for $X\in \pg$.

Secondly, we observe that 
\begin{equation}\label{eqn_phiKX2}
	q^{-1} \phi(Y) q = \phi(a^{-1} Y), \qquad \forall Y\in \lgo.
\end{equation}
This follows from the fact that $\Ad_\F(x)^{-1} \in \Aut(\fg)$, which yields
\[
	\Ad_\F(x)^{-1} \circ \ad_\fg (Y) \circ \Ad_\F(x) = \ad_\fg (a^{-1} Y), \qquad \forall Y\in \lgo,	
\]
and then one simply restricts  to $\sg$.

Let $\tilde T$ denote transpose in $\End(\sg)$ with respect to $q \cdot g^\Ss_e$. Using \eqref{eqn_transposes},
 \eqref{eqn_phiKX1} and \eqref{eqn_phiKX2} we obtain
\[
	\phi(K)^{\tilde T} \overset{\eqref{eqn_transposes}}= q ( q^{-1} \phi(K) q)^T q^{-1} \overset{\eqref{eqn_phiKX2}} = 
		q \phi(a^{-1} K)^T q^{-1}  \overset{\eqref{eqn_phiKX1}} = - q \phi(a^{-1} K) q^{-1}  \overset{\eqref{eqn_phiKX2}} =  - \phi(K),
\]
and analogously for $X\in \pg$.
\end{proof}

From now on we choose $g^\Ss$ so that the conclusion of Lemma \ref{lem_phiKX} is satisfied. In particular, since $\K$ is connected, the inner product defined by $g^\Ss$ on $\sg \simeq T_e \Ss$ is $\Ad(\K)$-invariant. We extend it to an $\Ad(\K)$-invariant inner-product  $g^E$ on $\pg\oplus \sg$ by setting
\begin{equation}\label{eqn_def_g_E}
		g^E|_{\sg \times \sg} := g^\Ss_e, \qquad g^E(\pg,\sg) =  0, \qquad g^E |_{\pg \times \pg} := \big(\kf_\fg - \unm \kf_\lgo\big)|_{\pg\times \pg},
\end{equation}
where $\kf_\ggo$ denotes the Killing form of $\ggo$. Notice that Lemma \ref{lem_phiKX} gives
\[
	\kf_\fg(X,X) = \kf_\lgo(X,X) + \tr \phi(X)^2 \geq \kf_\lgo(X,X),  \qquad \forall X\in \pg, \quad X\neq 0.
\] 
Thus, $\kf_\fg - \unm \kf_\lgo \geq \unm \kf_\lgo$ is positive-definite on $\pg$, by definition of Cartan decomposition.

This inner product extends to an $\F$-invariant Riemannian metric on $\F/\K$, also denoted by $g^E$. Notice that the action of $\Ss$ on $\F/\K$ by left-multiplication is free and isometric, thus it induces a Riemannian submersion cf.~\cite[$\S$9]{BB}
\[
		\pi : \F/\K \to \Ss\backslash \F/\K \,  .
\]
Since $\Ss$ is normal in $\F$,  $\F/\Ss \simeq \Ll$ acts on the base,  and this action is of course isometric and transitive. Thus, the base is isometric to $\Ll/\K$, endowed with the $\Ll$-invariant metric defined by the $\Ad(\K)$-invariant inner product $\big(\kf_\fg - \unm \kf_\lgo\big)|_{\pg\times \pg}$ on $\pg$.

The Riemannian submersion $\pi$ has integrable horizontal distribution. Indeed, since $\pg \perp \sg$, an integral submanifold through $e\K$ is given by the $\Ll$-orbit $\Ll\cdot e\K \subset \F/\K$. Through different points $s\K$, $s\in \Ss$, the integral submanifold will be the orbit of the Levi subgroup $s \Ll s^{-1}$.

Furthermore, the $\Ss$-orbits are minimal submanifolds. It suffices to show this at the point $p:= e\K$, since by normality of $\Ss$ in $\F$, different $\Ss$-orbits are isometric by an ambient isometry of $\F/\K$. To that end, recall that $\la X,N \ra = -\tr L_X$ for all horizontal $X$, by \eqref{eqn_trLX}. By Lemmas \ref{lem_phiKX} and \ref{lem_LX_Killing} we have
\begin{equation}\label{eqn_LsgX}
	 L^\sg_{X_p} = \phi(X),
\end{equation}
in the notation of Section \ref{sec_betavolM}. This implies that $\tr L_X = \tr L^\sg_{X_p} =  \tr \phi(X) = 0$, since $\phi$ is a representation of a semisimple Lie algebra and so its image consists of traceless endomorphisms.

Using that $\phi(X)$ is self-adjoint for $X\in \pg$, it also follows that 
\[
		\Vert L_{X_p} \Vert^2 =  \tr \phi(X)^2 = \kf_\fg(X,X) - \kf_\lgo(X,X).
\]
We are now in a position to prove Theorem \ref{thm_F/K_Ein}:

\begin{proof}[Proof of Theorem \ref{thm_F/K_Ein}]
We claim that the $\F$-invariant metric $g^E$ on $\F/\K$ defined in \eqref{eqn_def_g_E} is Einstein. To see this, we will use the Riemannian submersion $\pi : \F/\K \to \Ll/\K$ determined by the $\Ss$-action, and compute its Ricci curvature at $p:= e\K$, which is enough by homogeneity.

By Theorem \ref{thm:Ric} and Proposition \ref{prop_ric_offdiag}, together with the above observations, we have that
\begin{align*}
	\ric_{g^E} (U,U) =&\,\,  -g^E(U,U) - \sum \la (\nabla_{X_k} L)_{X_k} U, U\ra, \\
	\ric_{g^E} (U,X) =&\,\, -\la L_X U, \mcvl \ra - \la \nabla^\ve U, L_X \ra, \\ 
	\ric_{g^E} (X,X) =&\,\, \ric_{\Ll/\K}(X,X) - \kf_\fg(X,X) + \kf_\lgo(X,X),
\end{align*}
at the point $p \in \F/\K$, and for all $U\in \sg$, $X\in \pg$, 
(After rescaling, we may assume without loss of generality that $\ric_{g^\Ss} = -g^\Ss$.) In the off-diagonal equation, $U$ is not a Killing field but is $\Ss$-invariant: see Proposition \ref{prop_ric_offdiag}.

Regarding the horizontal part, since $[\pg,\pg] \subset \kg$, it is well-known that 
\[
	\ric_{\Ll/\K} = -\unm \kf_\lgo,
\]
see e.g.~ \cite[7.38]{Bss}.  Thus, 
\[
	\ric_{g^E} (X,X) = - \big( \kf_\fg(X,X) -\unm \kf_\lgo(X,X)) = -g^E(X,X).
\]

For the off-diagonal terms, we identify --only for this paragraph-- 
\[
	\ve_p = T_p (\Ss\cdot p) \simeq \sg 
\]
via evaluation of $\Ss$-left-invariant vector fields (as opposed to Killing fields as in the rest of the paper). The Koszul formula for the Levi-Civita connection implies that, under this identification, $L_X \in \End(\ve_p)$ gets identified with $-\phi(X)\in \Der(\sg)$, $X\in \pg$, and $\nabla^\ve U$ gets identified with $-S(\ad_\sg U)$, where $S(E) = \unm(E+E^T)$ denotes the symmetric part. Thus, using that $[\fg,\sg] \subset \sg $, we deduce that
\[
	\la \nabla^\ve U, L_X\ra = - \tr_\fg \ad_\fg U  \ad_\fg X = -\kf_\fg (U,X) = 0,
\]
since the Levi decomposition is orthogonal with respect to the Killing form  \cite[Ch.~I, $\S$5, Prop.~5, b)]{Bour71}. For the first term $-\la L_X U, \mcvl \ra$, recall that the nilradical $\ngo_\fg$ of $\sg$ is unimodular, and from this it follows that, again under the above identification, $(\mcvl)_p \in \ngo_\fg ^\perp$. Since $\phi(X)$ is a self-adjoint derivation of $\sg$, it preserves $\ngo_\fg$ and thus vanishes on $\ngo^\fg$. These observations yield
\[
		\ric_{g^E}(U,X) = 0, \qquad \forall \, U\in \sg, \quad X\in \pg.
\]

Finally, for the vertical equation to hold, we must show that
\[
	\sum_k \la (\nabla_{X_k} L)_{X_k} U, U\ra_p = 0, \qquad \forall \,  U\in \sg.
\]
Here $\{X_k\}$ is a basis of Killing fields in $\pg$, orthonormal at $p$. Notice that $\sum \nabla_{X_k} X_k = 0$, as this is the mean curvature vector of the homogeneous space $\Ll/\K$. Thus,
\[
	\sum_k \la (\nabla_{X_k} L)_{X_k} U, U\ra_p = \left\la \big(\nabla_{X_k} (L_{X_k}) \big) U, U  \right\ra. 
\]
We now claim that $\nabla_X (L_X) = 0$ for all Killing fields $X\in \pg$. Indeed, by Lemmas \ref{lem_nablaXE} and \ref{lem_nablaXLX} and the formula \eqref{eqn_LsgX} we have 
\[
	\big(\nabla_X(L_X) \big)_p = D_{X_p} (L^\sg_X) = [\phi(X), L^\sg_{X_p}] = 0, \qquad \phi(X) = (\ad_\fg X)|_\sg,
\] 
where in the first equality we are using the identification via Killing fields evaluation. 
\end{proof}

\begin{lemma}\label{lem_nablaXLX}
Let $X$ be a Killing field with $X_p\perp \Ss\cdot p$, whose flow normalises $\Ss$. Then, 
\[
		D_{X_p}(L^\sg_X) = [\phi(X), L^\sg_{X_p}], \qquad \phi(X) = (\ad X)|_\sg.
\] 
\end{lemma}

\begin{proof}
We use the notation introduced in $\S$\ref{sec_betavolM} for the free isometric action of $\Ss$ on $M$.

Let $f: M \to M$ be an isometry of $M$ normalising $\Ss$. On one hand, $f$ maps the orbit $\Ss\cdot p$ to $\Ss \cdot q$, $q := f(p)$. Hence, for the second fundamental form we have
\[
	d f^{-1} \, L_{X_q}  \, d f = L_{d f^{-1} X_q}.
\]
On the other hand, for a Killing field $X$, it is well-known that $x \mapsto d f^{-1} X_{f(x)}$ defines the Killing field $\tilde X = \Ad f^{-1} X$ (see e.g.~\cite[(20)]{BL18}). This implies firstly that we may rewrite the above formula as
\begin{equation}\label{eqn_NXLX1}
	d f^{-1} \, L_{X_{f(p)}}  \, d f = L_{(\Ad f^{-1} X)_p}.
\end{equation}
And secondly, under the identifications $\ev_x : \sg \to \ve_x$, $x = p, f(p)$, $df|_p : \ve_p \to \ve_q$ corresponds to $\Ad f$. That is,
\begin{equation}\label{eqn_NXLX2}
		df |_p = \ev_{f(p)} \circ \Ad f \circ \ev_p^{-1}.
\end{equation}
Thus, we may use \eqref{eqn_NXLX2} to translate \eqref{eqn_NXLX1} into an equation on $\End(\sg)$ by means of \eqref{eqn_defE^ngo}:
\begin{equation}\label{eqn_NXLX3}
	L^\sg_{X_{f(p)}} = \Ad f \, \circ \, L^\sg_{(\Ad f^{-1} X)_p} \, \circ \, \Ad f^{-1}.
\end{equation}

Now if $\varphi_t = \exp(t X)$ denotes the flow of $X$, then by definition we have 
\[
		(D_X(L^\sg_X))_p = \ddt\big|_0 L^\sg_{X_{\varphi_t(p)}}. 
\]
By assumption, $\varphi_t$ consists of isometries which normalise $\Ss$. Applying \eqref{eqn_NXLX3}, we obtain the stated formula after a  straightforward computation.
\end{proof}

\section{Proof of Theorem \ref{thm_alek}: the Alekseevskii conjecture}\label{sec_alek}

Let $(M_1^n,g_1)$ be a connected homogeneous Einstein space with Einstein constant $-1$. By Jablonski's Theorem \ref{thm_Einsolvquot}, for proving Theorem \ref{thm_alek} we may assume that $M^n_1$ is simply-connected, provided we end up showing that it admits a transitive solvable group of isometries. 

The proof of Theorem \ref{thm_alek} is organised into the following steps:

\begin{enumerate}
  \item There exists a homogeneous quotient $(M^n,g)$ of $(M_1^n, g_1)$ admitting a presentation $M=\F/\Hh$ such that $g$ is $\F$-invariant, $\F = \K \G$ with $\Hh\leq \K$, $\K$ compact, and $\G$ is completely solvable and admits an Einstein left-invariant metric.  \label{item_Mpresentation}
  \item The action of the nilradical $\N$ of $\G$ on  $\F/\Hh$ is polar.  \label{item_Npolar}
  \item The mean curvature vector $N$ of the $\N$-orbits is $\G$-vertical. \label{item_NGvert}
  \item  The normaliser $N_\F(\G)$ of $\G$ in $\F$ acts transitively on $\F/\Hh$.  \label{item_NFG_trans}
  \item There exists a closed solvable Lie subgroup of $ \F$ acting transitively on $\F/\Hh$. \label{item_F/Hsolv}
\end{enumerate}

\subsection{Step \ref{item_Mpresentation}: Choosing the presentation as homogeneous space}

Since $(M_1,g_1)$ is a simply-connected homogeneous  Einstein space with $\Ricci_{g_1} = - g_1$, by the structure theory for these spaces developed in \cite{alek,JblPet14,AL16} (see \cite[Thm.~0.2]{JblPet14} and \cite[Thm.~2.1, Thm.~2.4 and Cor.~2.8]{AL16} for the precise statements), there exists a presentation $M_1 = \F/\Hh$ with the following properties: 
\begin{enumerate}[(i)]
  \item $\F$ acts effectively on $\F/\Hh$ and the isotropy $\Hh$ is a compact subgroup ; \label{item_Hcpt}
  \item There is a global Levi decomposition $\F = \Ll \ltimes\Ss$ (in particular, $\Ll \cap \Ss = \{ e\}$) with $\Ll$ a maximal connected semisimple Lie subgroup, $\Ss$ the solvable radical, which is simply-connected and completely solvable, $\Hh \leq \Ll$, and $\Ll$ has no compact simple factors; \label{item_Levi}
  \item The orbits of $\Ll$ and $\Ss$ are orthogonal at $p:= e\Hh$;  \label{item_perp}
  \item The induced metric $g^\Ss$ on $\Ss \cdot p$ is Einstein with $\ric(g^\Ss) = -g^{\Ss}$ . \label{item_SEinstein}
  \setcounter{enum_counter}{\value{enumi}}
\end{enumerate}

Moreover, by \cite[Prop.~3.9]{JblPet14}, the Lie algebra representation $\phi : \lgo \to \End(\sg)$ defining the semi-direct product structure of $\fg = \lgo \ltimes \sg$ (that is, $\phi(X) = (\ad_\fg X)|_\sg$ for $X\in \lgo$) satisfies the following important  compatibility with the geometry of $\Ss\cdot p$:
\begin{enumerate}[(i)]
  \setcounter{enumi}{\value{enum_counter}}
  \item We have that $\phi(\lgo) = \phi(\lgo)^T$, where the transpose is taken with respect to the inner product on $\sg$ induced by the Einstein metric $g^\Ss$ at $p$. \label{item_compat}
  \setcounter{enum_counter}{\value{enumi}}
\end{enumerate}



Let now $\Ll = \K \A \N$ be an Iwasawa decomposition with $\Hh \leq \K$, and denote the corresponding Borel subgroup by $\B:= \A \N$. Conditions \ref{item_SEinstein} and \ref{item_compat} allow us to apply Theorem \ref{thm_F/K_Ein} and conclude that the following also  holds:

\begin{enumerate}[(i)]
  \setcounter{enumi}{\value{enum_counter}}
  \item The homogeneous space $\F/\K$ admits an $\F$-invariant Einstein metric. \label{item_F/K_Ein}
  \setcounter{enum_counter}{\value{enumi}}
\end{enumerate}

Thus far we have constructed a presentation $ M_1 = \F/\Hh$ satisfying \ref{item_Hcpt}--\ref{item_F/K_Ein}.  Consider now $\Lambda := Z(\F)\cap \Ll$, a discrete central subgroup of $\F$ and $\Ll$. The quotient $M:= \Lambda \backslash  M_1$ is homogeneous, locally isometric to $(M_1,g_1)$,  and has a presentation $(\F/\Lambda) / (\Hh/\Hh\cap \Lambda)$ which clearly also satisfies \ref{item_Hcpt}--\ref{item_F/K_Ein}. Moreover, this presentation has the advantage that the Levi factor $\Ll/\Lambda$ has finite center. Indeed, by construction $\Ad_{\F/\Lambda} (\Ll/\Lambda) \subset \Gl(\fg)$ is a linear semisimple Lie group isomorphic to $\Ll/\Lambda$, because $\ker \Ad_\F |_\Ll = \Lambda$. 

 In what follows we work on $M:=  M_1 / \Lambda$, and by abuse of notation we drop the $\Lambda$ quotients in the groups notation. In addition, we may assume that 

\begin{enumerate}[(i)]
  \setcounter{enumi}{\value{enum_counter}}
  \item The center of $\Ll$ is finite, and in particular $\K$ is compact \cite[Thm.~6.31]{Knapp2e}. \label{item_K_cpt}
  \setcounter{enum_counter}{\value{enumi}}
\end{enumerate}

To summarise, we re-state what we have proved in this first step:

\begin{proposition}
For any simply-connected homogeneous Einstein manifold $(M^n_1, g_1)$ with Einstein constant $-1$, there exists a homogeneous Einstein space $(M^n := \F/\Hh, g)$ satisfying \ref{item_Hcpt}--\ref{item_K_cpt} above, whose universal cover is $(M^n_1, g_1)$.
\end{proposition}

For the remaining of the proof, we focus on proving that there exists a closed solvable Lie subgroup of $\F$ acting transitively on $M$. This would imply that $(M^n_1,g_1)$ itself is a solvmanifold, and by Theorem \ref{thm_Einsolvquot}  this is enough for proving Theorem \ref{thm_alek}.

\subsection{Step \ref{item_Npolar}: The $\N$-action is polar}


Consider the subgroup $\G \leq \F$ given by
\[
    \G = \B \ltimes \Ss.
\]
By well-known properties of the Iwasawa decomposition (see for instance \cite[Ch.~VI, $\S$4]{Knapp2e}), $\B$ is a closed, simply-connected,  completely solvable subgroup of $\Ll$ intersecting $\K$ trivially. Thus, $\G$ is a closed, completely solvable, simply-connected Lie subgroup of $\F$, which by  \ref{item_Levi} does not intersect $\K$, and we furthermore have $\F = \K \G$ with $\K$ compact by \ref{item_K_cpt}. This implies that the action of $\G$ on $\F/\Hh$ by left-multiplication is free, proper and isometric, and the quotient space is  $\G \backslash \F / \Hh$, diffeomorphic to the compact homogeneous space $\K/\Hh$.

Hence, we are in a position to apply Theorem \ref{thm_rig_polar} and conclude that the action of the nilradical $\N$ of $\G$ on $M$ induces a Riemannian submersion $M \to \N\backslash M$ with integrable horizontal distribution. Recall also that, thanks to the equivariant modified Helmholtz decomposition, the mean curvature vector $N$ of the $\N$-orbits can be decomposed as 
\[
		N = -\nabla \log v + \Nw, \qquad \Nw \in \Xg(M)^{\G}, \quad v\in \cca^\infty_+(M)^{\G}.
\]
Theorem \ref{thm_rig_polar} and Corollary \ref{cor_vngoconst} imply that 
\begin{equation}\label{eqn_N=Nw}
  N = \Nw, \qquad \nabla^P N =0, \qquad L_N = -\beta^+.
\end{equation}

\subsection{Step \ref{item_NGvert}: $N$ is $\G$-vertical}

Condition \ref{item_F/K_Ein} satisfied by the presentation $\F/\Hh$ implies that the simply-connected solvable Lie group $\G$ admits a left-invariant Einstein metric.  Thus, Assumption \ref{as_G_Einstein} holds, and we may apply Theorem  \ref{thm_X=0} to conclude that $N$ is $\G$-vertical.  



\subsection{Step \ref{item_NFG_trans}: $N_\F(\G)$ acts transitively on $\F/\Hh$}

For convenience of the reader we briefly recall our setup on Lie algebra level: $\fg = \lgo \ltimes \sg$ is a Levi decomposition; $\lgo = \kg \oplus \ag_\lgo \oplus \ngo_{\lgo}$ is an Iwasawa decomposition with $\kg$ containing the isotropy subalgebra $\hg$; by setting $\ngo_\lgo^- := \theta(\ngo_\lgo)$ and $\lgo_0 := Z_\lgo(\ag_\lgo)$, where $\theta\in \Aut(\lgo)$ is the corresponding Cartan involution, we get 
\[
  \lgo = \ngo_\lgo^- \oplus \lgo_0 \oplus \ngo_\lgo.
\] 
Regarding the solvable radical, we write $\sg = \ag_\sg \oplus \ngo_\fg$ where $\ngo_\fg$ is the nilradical of $\sg$ (also of $\fg$) and $\ag_\sg$ is the orthogonal complement of $\ngo_\fg$ in $\sg$ with respect to the inner product induced by $g^\Ss_p$ on $\sg$. The Lie bracket satisfies
\begin{equation}\label{eqn_[as,l]=0}
  [\fg,\sg] \subset \ngo_\fg, \qquad  [\ag_\sg,\ag_\sg] = 0, \qquad [\lgo,\ag_\sg] = 0.
\end{equation}
Indeed, the first condition is well-known, see e.g.~\cite[Thm.~3.8.3,(iii)]{Varad84}. The second one follows from the integrability of the $\N$-horizontal distribution in $\Ss\cdot p$. The last one follows from the first one, together with \ref{item_compat} and the fact that $[\lgo,\ngo_\fg] \subset \ngo_\fg$.

Finally, we may also write
\[
    \ggo := (\ag_\lgo \oplus \ag_\sg) \ltimes (\ngo_\lgo \ltimes \ngo_\fg) = \ag \ltimes \ngo, \qquad \ag := \ag_\lgo \oplus \ag_\sg, \qquad \ngo := \ngo_\lgo \ltimes \ngo_\fg,
\]
with $\ag$ abelian by \eqref{eqn_[as,l]=0}.

By the previous steps, $N_p$ is tangent to the orbit $\G \cdot p$. Moreover, by \cite[Cor.~2.10]{Heb}, after changing $p$ to another point $x\cdot p$ for some $x\in \N$, we may assume without loss of generality that for some Killing field $A\in \ag$ we have $-N_p = A_p$. 

\begin{lemma}\label{lem_ad_fgA}
Let $D:= \ad_\fg A \in \End(\fg)$, and decompose $\fg$ as 
\[
   \fg = \fg_- \oplus \fg_0 \oplus \fg_+, \qquad  \fg_- := \ngo_\lgo^-, \qquad \fg_0 := \lgo_0 \oplus \ag_\sg, \qquad \fg_+ := \ngo.
\]
Then, $D$ is positive-definite on $\fg_+$, negative-definite on $\fg_-$, and zero on $\fg_0$.
\end{lemma}

\begin{proof}
Lemma \ref{lem_LX_Killing} and \eqref{eqn_N=Nw} yield
\[
   \unm \left(  (\ad_\ggo A)|_\ngo + (\ad_\ggo A)|_\ngo^T \right) = (\bM_p)^+.
\]
Using \eqref{eqn_[adA,beta]=0} and the fact that $(\ad_\fg A)|_\ngo = (\ad_\ggo A)|_\ngo$ has only real eigenvalues (for $\ggo$ is completely solvable), we deduce that
\[
    (\ad_\fg A)|_\ngo = (\bM_p)^+.
\]
In particular, $D:= \ad_\fg A$ is positive-definite on $\ngo$. By \eqref{eqn_[as,l]=0} and the fact that $\ag$ is abelian, $\fg_0 \subset \ker D$.  Regarding $\fg_-$, let us write $A = A_\lgo + A_\sg$ with $A_\lgo \in \ag_\lgo$, $A_\sg \in \ag_\sg$. Again by \eqref{eqn_[as,l]=0}, $[A_\sg, \lgo] = 0$, thus $\ad_\lgo A_\lgo = D|_{\ngo_\lgo} > 0.$
Using that $\theta A_\lgo = - A_\lgo$ and the fact that $\theta = \theta^{-1} \in \Aut(\lgo)$ it is clear that if $U\in \ngo_\lgo$ is an eigenvector of $\ad_\lgo A_\lgo$ with eigenvalue $\lambda >0$, then $\theta (U) \in \ngo_\lgo^-$ is an eigenvector with eigenvalue $-\lambda < 0$. It follows that $D|_{\fg_-} < 0$. 
\end{proof}

\begin{proposition}\label{prop_NFGtrans}
The group $N_\F(\G)$ acts transitively on $M= \F/\Hh$.
\end{proposition}

\begin{proof}
By Lemma \ref{lem_ad_fgA}  we may apply the estimate in Proposition \ref{prop_estimate} yielding 
\[
  \scal^\ve(p) + \dim \ngo \leq \tr \beta^+,
\]
$p = e\Hh$. But by Proposition \ref{prop_Ric_N}, equality must hold. Hence, the rigidity in Proposition \ref{prop_estimate} implies that the Killing fields in $\fg_0 \oplus \fg_+$ span the entire tangent space $T_p M$.  Notice that $\fg_0 \oplus \fg_+$ is the Lie algebra of the normalizer $N_\F(\G)$ of $\G$ in $\F$. Since $N_\F(\G)$ is closed in $\F$, the orbit $N_\F(\G) \cdot p$ is an embedded submanifold, in particular a closed subset. It is also open, since it has full dimension. By conectedness, $N_\F(\G)$ must act transitively. 
\end{proof}

\subsection{Step \ref{item_F/Hsolv}: $\F/\Hh$ is a solvmanifold}

Proposition \ref{prop_NFGtrans} implies that the Einstein manifold $(M^n,g)$ admits a transitive group of the form $N_\F(\G)_0 = \M \ltimes  \G$, where $\G$ is solvable, and $\M \leq \Ll$, the identity component of the normalizer of $\ag$ in $\K$, is compact. We may further decompose $\M$ as $\M_{ss} \M_z$, where $\M_{ss}$ is semisimple and $\M_z$ central in $\M$. Then, the transitive group 
\[
    \M \ltimes \G = \M_{ss} \ltimes (\M_z \G)
\] 
has a compact Levi factor $\M_{ss}$. By \cite[Thm.~0.1]{JblPet14} (see also \cite[Cor.A.2]{semialglow}), the solvable group $\M_z \ltimes \G$ acts transitively on $M^n$. Thus, $(M^n,g)$ is an Einstein solvmanifold. Clearly, the same applies for its universal cover $(M^n_1, g_1)$, and  Theorem \ref{thm_alek} now follows from Theorem \ref{thm_Einsolvquot}.

\begin{appendix}

\addtocontents{toc}{\protect\setcounter{tocdepth}{1}}

\section{Modified Helmholtz decomposition}\label{app_PDE}

In this section we study the following second order linear elliptic PDE 
\begin{equation}\label{eqn_PDE}
    \lca v := \divg( \nabla v + v X) = \Delta v + \la \nabla v, X \ra + v  \divg X= 0,
\end{equation} 
on a closed Riemannian manifold $B^d$, where $X$ is a given smooth vector field (corresponding to $N$ in the previous sections) and $\divg X := \tr \nabla X$.
Our main goal is to show

\begin{proposition}\label{prop_positive}
Up to scaling, there exists a unique smooth solution to \eqref{eqn_PDE}, and this solution does not change sign.
\end{proposition}

\begin{example}
If $X = \nabla u$ is a gradient vector field, then a solution to \eqref{eqn_PDE} is given by $v = e^{-u}$. On the other hand, if $X$ is divergence free, then $v$ must be constant.
\end{example}

\begin{corollary}[Modified Helmholtz decomposition]\label{cor_Helmholtz}
Let $X$ be a vector field on a compact Riemannian manifold $B$. Then, there exists  $X_0\in \Xg(B)$ and  a positive $v \in \cca^\infty_+(B)$ such that
\[
		X = -\nabla \log v + X_0, \qquad \divg(v X_0) = 0.
\]
\end{corollary}

\begin{proof}
By Proposition \ref{prop_positive}, there exists a positive smooth solution $v>0$ to \eqref{eqn_PDE}.  We then simply set $X_0 = \nabla \log v + X$.
\end{proof}

We are indebted to H.J.~Hein for the following proof:

\begin{proof}[Proof of Proposition \ref{prop_positive}]
 Consider the
Sobolev spaces $H_k :=W_k^2(B)$, $k \in \NN$, with induced norm
\[
 \Vert v \Vert^2_k = \int_B \sum_{i=0}^k\Vert \nabla^i v\Vert^2_g \,\,d \mu_g\, ,
\]
where $\nabla^0 v :=v$, $\nabla^1 v = \nabla v$ is 
the gradient, 
and $\nabla^i v$ denotes the $(i-1)$-th covariant derivative of $\nabla v$ with respect to
the Levi-Civita connection of $(B,g)$. For each $k\in \NN$, $H_k$ is a Hilbert space.

Let $k \geq k_\ast :=\tfrac{d}{2}+11$. Then
by \cite[Ch.~4, Prop.~3.3]{Tay} we have
\begin{equation} \label{eqn_HkinC10}
  H_k \subset \cca^{10}(B).
\end{equation}
For $t \in [0,1]$ we consider  the family of uniformly elliptic operators
\[
  \lca_t :H_{k+2} \to H_k \, , \qquad v \mapsto \divg(\nabla v + t \cdot v \cdot X) \,.
\]
The corresponding adjoints $\lca^*_t:H_k \to H_{k+2}$ are also uniformly elliptic and are given by 
\[
   \lca^*_t(v) = \Delta v - t \, \la \nabla v, X\ra\,.
\]
This family yields a continuous path connecting $\lca_1 = \lca$ with the Laplace-Beltrami operator $\lca_0 = \Delta$. The proof strategy is to use the maximum principle for $\lca_t^*$ to show that $ \dim \ker \lca_t = 1$, and that there is a continuous curve $(v_t)_{t\in [0,1]}$  in $H_{k+2} \subset \cca^{10}(B)$ with $v_0 \equiv 1$ constant on $B$, and $v_t \in \ker \lca_t$ for each $t\in [0,1]$. Notice that by elliptic regularity, $v_t$ is in fact smooth for all $t\in [0,1]$ \cite[Cor.~8.11]{GT01}. The fact that $v_1 > 0$ will then follow from Harnack's inequality.

Firstly, we show that  $\lca_t, \lca_t^*$ 
are Fredholm
for all $t \in [0,1]$. By \cite[App.~I, Thm.~31]{Bss},
$\ker \lca_t$ is finite-dimensional, and
by \cite[Thm.~8.10]{GT01} we have
an estimate of the form
\[
  \Vert v\Vert_{k+2} \leq C \cdot ( \Vert \lca_t(v) \Vert_k + \Vert v\Vert_k)\,.
\]
We write $H_{k+2}= \ker \lca_t \oplus (\ker \lca_t)^\perp$ and show that
the range $\ran \lca_t$ of $\lca_t$ is closed. Let $(v_n)_{n \in \N}$ be a sequence 
in $(\ker \lca_t)^\perp$ with 
$\lim_{n\to \infty}w_n=w \in H_k$, $w_n:=\lca_t(v_n)$.
It is enough to show,
that $(v_n)$ subconverges to $v_\infty \in H_{k+2}$. Firstly, we assume that $(v_n)$ is
a bounded sequence. Then, since $H_{k+2}$ is  compactly embedded into 
$H_k$,  see \cite[Ch.~4, Prop.~3.4]{Tay},
we may
assume that the sequence $(v_n)$ converges in $H_k$.
Thus
$$
 \Vert v_n -v_m \Vert_{k+2} \leq C \cdot ( \Vert w_n-w_m\Vert_k + \Vert v_n-v_m\Vert_k)
$$
and it follows that $(v_n)$ converges in $H_{k+2}$. It remains to consider the case
$\Vert v_n\Vert_{k+2} \to \infty$ for $n\to \infty$. We set $\tilde v_n:=\tfrac{v_n}{\Vert v_n\Vert_{k+2}}$. Then $\lca_t(\tilde v_n) \to 0$ for $n\to \infty$. As above it follows that
$\tilde v_n \to \tilde v \in H_{k+2}$, along a subsequence possibly. But then
$\lca_t(\tilde v)=0$, a contradiction.
By \cite[App.~A, Prop.~5.7]{Tay}, we have
 $(\ran \lca_t)^\perp = \ker \lca_t^*$.
 Since $\ker \lca_t^*$ is finite-dimensional it follows that $\lca_t$ is
Fredholm. Moreover, since $\ran \lca_t$  is closed,
by \cite[App.~A, Prop.~5.7]{Tay} also $\lca_t^*$ is Fredholm.

Since the index does not change along a continuous path of Fredholm
 operators \cite[App.~A,  Prop.~7.4]{Tay}, and $\lca_0 = \lca_0^*$ is the Laplace-Beltrami operator, for all $t\in [0,1]$ we have
\[
 {\rm ind } \,\lca_t =  \dim \ker \lca_t - \dim \ker \lca^*_t =  0.
\] 
Notice that $\lca_t^*$ satisfies a strong maximum principle: see \cite[$\S$3,  Thm.~5]{PrWe}. Thus,
$\dim \ker \lca^*_t = 1$,  the kernel consisting of constant functions,
and we deduce that $\dim \ker \lca_t = 1$ for all $t\in [0,1]$.


We now claim that the constant function $1$ is not in the range of $\lca_t^*$. This follows by contradiction, since at a maximum of $v$ we have $\lca_t^*(v) \leq 0$. Hence, $1 \notin (\ker \lca_t)^\perp$ by the closed range theorem. By \cite{Bro65}, the orthogonal projections $P_t : H_{k+2} \to \ker \lca_t$ depend continuously on $t$. We thus have a continuous family $(v_t := P_t(1))_{t\in [0,1] } \subset H_{k+2}$,  with $\lca_t v_t  = 0$ and $v_t\neq 0$ thanks to the previous claim. Moreover, $v_0 \equiv 1$ is constant. 


Finally, we claim that $v_t > 0$ for all $t \in [0,1]$. This holds for $v_0 \equiv 1$, and
it is an open condition, since $v_t$ depends continuously on $t$ in $C^0$-topology.
Closedness follows from the Harnack inequality applied 
locally in $B$: see \cite[Thm.~8.20]{GT01}. 
\end{proof}

\begin{remark}
Another approach for proving Proposition \ref{prop_positive} is as follows. The second order linear elliptic operator $\lca u = -\Delta u + \la \nabla u, X\ra + c \, u$, $c\in \cca^\infty(B)$, acting on functions on  a closed smooth Riemannian manifold  $B$, has a \emph{principal eigenvalue} $\lambda_1$,  characterised by the following properties:
\begin{enumerate}
  \item $\lambda_1$ admits a positive eigenfunction $v$;
  \item $\lambda_1$ is a simple eigenvalue;
  \item For any other complex eigenvalue $\lambda \neq \lambda_1$ of $\lca$ one has $\Re(\lambda) > \lambda_1$.
\end{enumerate}
(Moreover, in case $c\geq 0$ and $c\not\equiv 0$, we have $\lambda_1 > 0$, but we will not need this.) This statement is well-known in the case of bounded domains with smooth boundary in $\RR^n$, see \cite[Theorem 3 in $\S$6.5]{Evans98} or \cite{Ni14} (it  also holds in more general domains, see e.g. \cite{BNV94}). According to some experts, the proof carries over to the closed manifold case without any issues.

It follows that there exists a smooth $v>0$ and $\lambda_1\in \RR$ with
\[
  \divg( \nabla v + v X)  = \lambda_1 v. 
\]
Integrating both sides and using that $v>0$ we deduce that $\lambda_1 = 0$, thus $v$ solves \eqref{eqn_PDE}.
\end{remark}

\section{An equivariant divergence theorem}\label{app_div}


The main aim of this section is to prove equivariant versions of the divergence theorem and the modified Helmholtz decomposition (Corollary \ref{cor_Helmholtz}), for vector fields on a Riemannian manifold $(M^n,g)$ with a co-compact isometric action of a unimodular Lie group $\G$ for which the orbit space $B := M/\G$ is a closed manifold.  We denote by $\Xg(M)^\G$, $\cca_+^\infty(M)^\G$ the spaces of $\G$-invariant vector fields and positive $\G$-invariant smooth functions on $M$, respectively. As usual, we endow $B$ with the quotient Riemannian metric $\gB$, so that $\pi : M\to B$ is a Riemannian submersion.

\begin{proposition}[Equivariant modified Helmholtz decomposition]\label{prop_equivHelm}
Let $(M^n,g)$, $\G$ be as above, and let $E\in \Xg(M)^\G$. Then, there exists $v\in \cca_+^\infty(M)^\G$ and $E_0\in \Xg(M)^\G$ such that
  \[
      E = -\nabla \log v + E_0, \qquad \divg_M(v E_0) = 0.
  \]
  Moreover, $E_0$ is unique with those properties, and $v$ is unique up to scaling.
\end{proposition}

\begin{proposition}\label{prop_divgeq0} 
Let $(M^n,g)$, $\G$ be as above, and let $E\in \Xg(M)^\G$ such that $\divg_M E \geq 0$. Then, $\divg_M E \equiv 0$.
\end{proposition}

We remark that the vector field $E$ in the above statements is not necessarily $\G$-horizontal.

To prove these we will first establish some basic lemmas. Let $N$ be the mean curvature vector of the $\G$-orbits, which by Lemma \ref{lem_Ngrad} can be written as 
\[
    N =  -\nabla \log v_\G, \qquad v_\G \in \cca_+^\infty(M)^\G.
\]

\begin{lemma}\label{lem_divgs}
Let $X\in \Xg(M)^\G$ be a basic (horizontal) vector field. Then,
\[
  \divg_M X = (\divgB  X) \circ \pi - \la X, N\ra.
\]
\end{lemma}

\begin{proof}
Let $\{X_i \}$, $\{ U_j\}$ be basic and vertical orthonormal frames, respectively. Then,
\begin{align*}
  \divg_M X &= \,\, \sum_i \la \nabla^M_{X_i} X, X_i \ra + \sum_j \la \nabla^M_{U_j} X, U_j \ra \\
  &= \,\, \sum_i \la \nabla^B_{X_i} X, X_i \ra - \sum_j \la X, \nabla^M_{U_j}, U_j \ra \\
    &= \,\, \divgB X - \la X, N \ra.
\end{align*}
\end{proof}

\begin{lemma}\label{lem_divU=0}
Let $U\in \Xg(M)^\G$ be a vertical, $\G$-invariant vector field. Then, $\divg_M U = 0$.
\end{lemma}

\begin{proof}
We first prove the result for a transitive $\G$-action, that is, $M^n = \G/\Hh$.  Let $d \vol_g$ denote the $\G$-invariant volume form on $(M^n,g)$, and recall that $\divg_M X = \lca_X d \vol_g$. The pull-back of $d \vol_g$ under the $\G$-equivariant projection $\G \to \G/\Hh$ is a left-invariant $n$-form on $\G$. By unimodularity, said pull-back is also right-invariant. It follows that $d\vol_g$ is right-$\N_\G(\Hh)$-invariant. In particular, since $\G$-invariant vector fields on $\G/\Hh$ are precisely those which are tangent to the right $\N_\G(\Hh)$-action on $\G/\Hh$, we have $\divg_M X = 0$.

In the general case, let $p\in M$, and let $\{ X_k\}$ and $\{ U_i\}$ be respectively horizontal and vertical orthonormal frames. Then,
\begin{align*}
    \divg_M U = \, \,  \sum_k \la \nabla_{X_k} U, X_k\ra + \sum_i \la \nabla_{U_i} U, U_i \ra  = \,\, \divg_{\G\cdot p} U = 0.
\end{align*}
where the second equality follows from skew-symmetry of the $A$-tensor in the horizontal entries, and the last one from the homogeneous case.
\end{proof}

\begin{lemma}\label{lem_divMdivB}
Let $E\in \Xg(M)^\G$. Then, 
\[
 \divg_M E = v_\G^{-1} \, \divgB (v_\G\, \ho E).
\]
\end{lemma}

\begin{proof}
By Lemmas \ref{lem_divgs} and \ref{lem_divU=0} we have 
\[
    \divg_M E = \divg_M \ho E = \divgB \ho E + v_\G^{-1} \la \ho E, \nabla v_\G \ra.
\]
Thus,
\[
   v_\G^{-1}\,  \divgB (v_\G \, \ho E)  = v_\G^{-1} \, \left( v_\G \, \divgB \ho E + \la \ho E, \nabla v_\G \ra \right) = \divg_M E.
\]
\end{proof}

We are now in a position to prove both main results of this section:

\begin{proof}[Proof of Proposition \ref{prop_equivHelm}]
Since $B$ is compact, by Corollary \ref{cor_Helmholtz} there is a generalised Helmholtz decomposition for the vector field $N+\ho E \in \Xg(B)$,
\[
    N + \ho E = -\nabla \log v_1 + X_0, \qquad \divgB(v_1 \, X_0) = 0.
\]
We lift $X_0$ and $v_1$ to $\G$-invariant objects on $M$, and notice that
\[
    E =  -N + (N+\ho E) + \ve E =  -\nabla \log (v_1/v_\G) + X_0 + \ve E.
\]
We thus set $v:= v_1/v_G$, $E_0 \in X_0 + \ve E$, both clearly $\G$-invariant. Finally, by Lemmas \ref{lem_divU=0} and \ref{lem_divMdivB} we have
\[
    \divg_M \left( (v_1/ v_G) \, (X_0 + \ve E) \right) =  \divg_M \left( (v_1/ v_G) \, X_0 ) \right) = v_\G^{-1} \, \divgB (v_1 X_0) = 0.
\]

Finally, since $\ho E = -\nabla \log v + \ho E_0$ is a generalised Helmholtz decomposition for $\ho E \in \Xg(B)$, uniqueness follows from Corollary \ref{cor_Helmholtz}.  
\end{proof}

\begin{proof}[Proof of Proposition \ref{prop_divgeq0}]
By  Lemma \ref{lem_divMdivB} the assumptions imply 
\[
    \divgB (v_\G \ho E) = v_\G \, \divg_M E \geq 0.
\]
Integrating over $B$, the divergence theorem yields equality everywhere.
\end{proof}

\section{Some formulae involving Killing fields}\label{app_Killing}

In this section we recall some useful formulae for computing with Killing fields.

\begin{lemma}\cite[Lemma 7.27]{Bss}\label{lem_727}
Let $X,Y,Z$ be Killing fields on a Riemannian manifold $(M^n,g)$. Then, the Levi-Civita connection $\nabla$ satisfies
\begin{equation}\label{eqn_727}
  2 \, \la \nabla_X Y, Z \ra = \la [X,Y], Z \ra + \la [X,Z], Y \ra + \la X, [Y,Z] \ra.
\end{equation}
\end{lemma}

From this, one can deduce the following:

\begin{lemma}\label{lem_LX_Killing}
Let $\N$ act isometrically on $(M^n,g)$, and let $X$ be a Killing field of $(M,g)$ with $X_p \perp (\N\cdot p)$. Then, for all Killing fields $U,V\in \ngo$, the shape operator satisfies
\[
    \la L_X U, V \ra_p = \unm \big( \la [X,U], V \ra_p + \la U, [X,V] \ra_p \big).
\] 
In particular, if $[X,\ngo]\subset \ngo$, then in the notation of $\S$\ref{sec_betavolM}
we have
\[
    L^\ngo_{X_p} = S(\ad X |_\ngo) := \unm\big( (\ad X |_\ngo) + (\ad X |_\ngo)^T \big),
\]
transpose with respect to $h_p$.
\end{lemma}

\begin{proof}
By the symmetries of the second fundamental form and \eqref{eqn_727}, we have 
\[
  \la L_X U, V\ra_p = - \la \nabla_U V,  X \ra_p = -\unm \,\left(  \la [U,V], X \ra_p + \la [U,X],V  \ra_p +\la U, [V,X] \ra_p\right),
\]
and the lemma follows since $[U,V]_p \perp X_p$.
\end{proof}

Even though the next lemma is well known, 
 we provide a proof for convenience (see e.g. \cite[Proposition 8.1.3]{Pet16} 
 or \cite[Ch.~VI, Prop.~2.6 (2)]{KobNom96}, however notice that the latter contains a sign mistake: see \cite[p.~469]{KobNom96ii}):

\begin{lemma}\label{lem_R(X,Y)}
Let $X,Y$ be Killing fields. Then,  $R(X,Y) = [\nabla X, \nabla Y] + \nabla [X,Y]$.
\end{lemma}

\begin{proof}
Since $X$ is a Killing field, the Lie derivative $\lca_X$ preserves the Levi-Civita connection: 
\[
    [\lca_X, \nabla_Y] = \nabla_{[X,Y]}.
\]
Using $\nabla X = \nabla_X - \lca_X$
and $ [\lca_X, \lca_Y]= \lca_{[X,Y]}$, we compute:
\begin{align*}
  [\nabla X, \nabla Y] + \nabla[ X,Y] =& \,\, [\nabla_X - \lca_X, \nabla_Y - \lca_Y] + \nabla_{[X,Y]} - \lca_{[X,Y]} \\
   =& \,\, [\nabla_X, \nabla_Y] + [\lca_X, \lca_Y] - [\lca_X, \nabla_Y] + [\lca_Y, \nabla_X] +  \nabla_{[X,Y]} - \lca_{[X,Y]} \\
   =& \,\, R(X,Y) + \nabla_{[Y,X]} + \nabla_{[X,Y]} = R(X,Y). 
  \end{align*}
\end{proof}

\section{The endomorphism \texorpdfstring{$\beta$}{b} associated to a Lie algebra}\label{app_beta}

In this appendix we follow the notation from $\S$\ref{sec_scaln}. We refer the reader to \cite{GIT20} for further details. Let $\ggo$ be a non-abelian Lie algebra endowed with a background inner product $\bkg$ and set $\Or(\ggo) := \Or(\ggo,\bkg)$. This induces an inner product $\ipp_{\bkg}$ on $V_\ggo := \Lambda^2(\ggo^*) \otimes \ggo$, see again $\S$\ref{sec_scaln}.  By studying the natural 'change of basis' linear action of $\Gl(\ggo)$ on $V_\ggo$, defined in \eqref{eqn_action_Vn}, from a real geometric invariant theory point of view, one obtains a $\Gl(\ggo)$-invariant stratification 
\[
		V_\ggo \backslash \{ 0\} = \bigcup_{\beta\in \bca} \sca_\beta,
\]
where the union is disjoint. Here, $\bca$ is a finite set of $\bkg$-self-adjoint endomorphisms of $\ggo$ which are uniquely determined up to conjugation by $\Or(\ggo)$. This stratification, first obtained by Kirwan and Ness in the complex setting \cite{Krw1,Ness}, and in the real setting by \cite{HSS,standard}, has a number of remarkable properties. We describe in Proposition \ref{prop_scabetab} below the most important one of them regarding applications in this article.

The Lie bracket $\mu_\ggo$ of $\ggo$, being a non-zero element in $V_\ggo$, belongs to a unique stratum
\[
		\mu_\ggo \in \sca_{\betab}.
\]
With respect to an ordered $\bkg$-orthonormal basis of eigenvectors of $\betab$ with eigenvalues in non-decreasing order, we consider the solvable Lie subgroup $\Bbeb \leq \Gl^+(\ggo)$  represented in said basis  by lower triangular matrices with positive diagonal entries. Notice that 
\[
	\Gl(\ggo) = \Bbeb \Or(\ggo), \qquad \Bbeb \cap \Or(\ggo) = \{ \Id\}.
\] 
Set $\betab^+ := \betab/ \tr (\betab^2) + \Id_\ggo$. The endomorphism $\tau(\betab^+) \in \End(V_\ggo)$ (see \eqref{eqn_action_Vn}) is also self-adjoint, and we denote by $V_{\betab^+}^{\geq 0}$ the sum of eigenspaces of $\tau(\betab^+)$ with non-negative eigenvalues.

\begin{proposition}\cite[Lemma 1.7.13]{GIT20}\label{prop_scabetab}
We have that $\sca_\betab = \Or(\ggo) \cdot U_{\betab^+}^{\geq 0}$, where $ U_{\betab^+}^{\geq 0} \subset V_{\betab^+}^{\geq 0}$ is a certain open, $\Bbeb$-invariant subset. 
\end{proposition}

Since $U_{\alpha^+}^{\geq 0} = k \cdot U_{\betab^+}^{\geq 0}$ for any $\alpha = k \betab k^{-1}$, $k\in \Or(\ggo)$, an upshot of Proposition \ref{prop_scabetab} is that we may choose the stratum label $\betab$ so that $\mu_\ggo \in U_{\betab^+}^{\geq 0}$. We then say that $\mu_\ggo$ is \emph{gauged correctly} with respect to $\betab$. Clearly, this yields
\begin{equation}\label{eqn_appbetaestimate}
	\lla \, \tau\left(\betab^+ \right) \mu, \mu  \, \rra_{\bkg} \geq 0, \qquad \forall  \mu \in \Bbeb \cdot \mu_\ggo.
\end{equation}

\begin{lemma}\label{lem_appbeta}
Given a Lie algebra $\ggo$ with Lie bracket $\mu_\ggo \in V_\ggo \backslash \{0\}$ and a background inner product $\bkg$ on $\ggo$, there exists a  unique $\bkg$-self-adjoint $\betab \in \End(\ggo)$ such that $\mu_\ggo \in \sca_\betab$ is gauged correctly with respect to $\betab$. In particular, \eqref{eqn_appbetaestimate} holds.
\end{lemma}

Let $\G$ be the simply-connected Lie group with Lie algebra $\ggo$. Given $h\in \Sym^2_+(\ggo^*)$, we denote by $\Ric(h) \in \End(\ggo)$ the Ricci endomorphism of the corresponding left-invariant metric on $\G$, at the point $e\in \G$, identifying $T_e \G \simeq \ggo$. The following is the key Ricci curvature estimate:


\begin{proposition} \cite[Lemma 6.2]{BL17}\label{prop_GITestimate}
If $(\ggo,\mu_\ggo)$ is a nilpotent Lie algebra with $\mu_\ggo \in \sca_\betab$ gauged correctly, then for any $h = q\cdot \bkg \in \Sym^2_+(\ggo^*)$, $q\in \Bbeb$, 
\[
	\tr \Ric(h) q \betab^+ q^{-1} \geq 0,
\]
 for all $h = q\cdot \bkg \in \Sym^2_+(\ggo^*)$, $q\in \Bbeb$.
Equality holds if and only if $q\betab^+ q^{-1} \in \Der(\ggo)$.
\end{proposition}

\begin{proof}
By Proposition \ref{prop_ricmm} applied to $E = q \betab^+ q^{-1}$, Lemma \ref{lem_ipp_equiv} and \eqref{eqn_appbetaestimate} we have
\begin{align*}
	\tr \Ric(h)E =& \,\, \unc \, \lla \tau(E) \mu_\ggo, \mu_\ggo \rra_h = \unc \, \lla q^{-1} \cdot (\tau(E) \mu_\ggo), q^{-1} \cdot \mu_\ggo \rra \\
		=& \,\, \unc \, \lla  \tau(\betab^+) (q^{-1} \cdot \mu_\ggo), (q^{-1} \cdot \mu_\ggo) \rra \geq 0.
\end{align*}
In the third eqality we also used the equivariance of $\tau$. Equality holds if and only if we have equality in \eqref{eqn_appbetaestimate}, which happens if and only if $q^{-1} \cdot \mu_\ggo$ is in the kernel of $\tau(\betab^+)$. By equivariance of $\tau$ this is equivalent to $\tau(q \betab^+ q^{-1}) \mu_\ggo = 0$, which by definition means $q\betab^+ q^{-1} \in \Der(\ggo)$.
\end{proof}

\begin{remark}
The estimate in Proposition \ref{prop_GITestimate} holds more generally for arbitrary homogeneous manifolds (replacing $\Ric(h)$ by the so-called \emph{modified Ricci curvature} in the non-unimodular case). This follows essentially from \cite[Lemma 6.2]{BL17}.
\end{remark}

\begin{remark}\label{rmk_betascal}
Let $\G$ be the simply-connected Lie group with Lie algebra $(\ggo,\mu_\ggo)$ and identify $\ggo$ with left-invariant vector fields on $\G$. Each $q\cdot \bkg \in \Sym^2_+(\ggo^*)$ corresponds to a left-invariant metric $g_q$ on $\G$, and we set $\bar g := g_{\Id}$. If $\{U_i\}$ is a left-invariant, $\bar g$-orthonormal frame of eigenvectors of $\betab^+$ with eigenvalues $\beta_i$, then $\{ qU_i \}$ is a $g_q$-orthonormal, left-invariant frame, and we have
\[
	\tr \Ric(h) q \betab^+ q^{-1} = \sum_i \beta_i \ric_{g_q}^{ii}, \qquad \ric_{g_q}^{ii} := \ric_{g_q} (qE_i, qE_i). 
\]
In this sense, the estimate in Proposition \ref{prop_GITestimate} asserts the non-negativity of the \emph{$\beta$-weighted scalar curvature}. The reason one has to be careful with the choice of gauge is that, in general, $\betab^+$ and $\Ric(h)$ do not diagonalise simultaneously.
\end{remark}

Besides Proposition \ref{prop_GITestimate}, the endomorphism $\betab^+$ satisfies a number of algebraic properties which are key for our applications:

\begin{proposition}\label{prop_beta}
Let $(\ggo,\mu_\ggo)$ be a non-abelian  Lie algebra with $\mu_\ggo \in \sca_\betab$ gauged correctly. Then, the following hold: 
\begin{enumerate}[(i)]
  \item $\tr\left(  D \, q \betab q^{-1} \right) = 0$, for all  $q\in \B_\betab$, $D\in \Der(\ggo)$; \label{item_trDbeta}
  \item $\tr \betab^+ = \tr \big( (\betab^+)^2 \big)$; \label{item_trbv+}
  \item $\tr [E, E^T] \betab \geq 0$ for all $E\in \bg_\betab$, with equality if and only if $[E,\betab] = 0$. Here $E^T$ denotes transpose with respect to $\bkg$. In particular, $\betab$ commutes with $\bkg$-self-adjoint derivations. \label{item_trEE^Tbeta}
\end{enumerate}
\end{proposition}

\begin{proof}
Given $D\in \Der(\mu_\ggo)$, we have that $q^{-1} D q \in \Der(q^{-1} \cdot \mu_\ggo)$. Since $q\in \Bbeb$, $q^{-1}\cdot \mu_\ggo \in U_{\betab^+}^{\geq 0}$ is also gauged correctly.  Thus, (i) follows from \cite[Corollary 1.9.2]{GIT20}.  Regarding (ii), by definition of $\betab^+$ and the fact that $\tr \betab = -1$, we have that $\tr \betab \betab^+ = 0$, from which
\[
  \tr (\betab^+)^2 = \frac{1}{\tr \betab^2} \, \tr \betab \betab^+ + \tr \betab^+ = \tr \betab^+ .
\]
Finally, (iii) follows from \cite[Lemma B.3]{BL18}.
\end{proof}

\begin{proposition}\cite{standard}, \cite[Corollary C.2]{BL17}\label{prop_beta+>0}
If $(\ngo, \mu_\ngo)$ denotes the nilradical of $\ggo$, then 
\[
	\betab^+ |_{\ngo^\perp} = 0, \qquad \betab^+|_{\ngo}  =  \betab^+_\ngo > 0.
\] 
In particular, $\B_\betab$ preserves the subspace $\ngo$.
Here, $\betab_\ngo$ denotes the stratum label for the Lie algbera $(\ngo, \mu_\ngo)$ (with background inner product $\bkg|_{\ngo\times \ngo}$), chosen so that $\mu_\ngo$ is gauged correctly. 
\end{proposition}


Finally, the following GIT technical lemma was needed in the proof of Proposition \ref{prop_NS}. For its proof we adopt the notation from \cite{GIT20}. 

\begin{lemma}\label{lem_GITb1b2}
Let $b_1, b_2 \in \Bbeb$ be such that $b_i \betab^+ b_i^{-1} \in \Der(\ggo)$ for $i=1,2$. Then, there exists $a\in \Aut(\ggo)$, $z\in \G_\betab := \{ h \in \Gl(\ggo) : h\betab h^{-1} = \betab \}$, such that
\[
    b_1 = a \, b_2 \, z.
\] 
\end{lemma}




\begin{proof}
Set $\betab^+_i := b_i \betab^+ b_i^{-1}$, $\mu_i := b_i^{-1} \cdot \mu^\ggo$, $i=1,2$. The assumption implies that 
\[
    \tau\left(\betab^+_i \right) \mu^\ggo = 0, \qquad i=1,2.
\]
Acting with $b_i^{-1}$ and using the equivariance of $\tau$ we obtain
\[
  0 = b_i^{-1} \cdot \tau \left(\betab^+_i\right) \mu^\ggo = \tau\left(b_i^{-1} \betab^+_i b_i\right) (b_i^{-1} \cdot \mu^\ggo)  = \tau(\betab^+) \mu_i, \qquad i=1,2.
\]
Thus, $\mu_i \in V^0_{\betab^+}$, $i=1,2$. Notice that $\mu_1 = (b_1^{-1} b_2) \cdot \mu_2$. Write $b_1^{-1} b_2 \in \Bbeb$ as 
\[
  b_1^{-1} b_2 = z u, \qquad z\in \G_\betab, \,\, u\in \U_\betab.
\]
Using that $\G_\betab$ preserves the subspace $V^0_{\betab^+}$, we have that
\[
    u \cdot \mu_2 = z^{-1} \cdot (b_1^{-1} b_2) \cdot \mu_2 = z^{-1} \cdot \mu_1 \in V^0_{\betab^+}.
\]
Applying the $\U_\betab$-invariant projection $p_\betab : V^{\geq 0}_{\betab^+} \to V^0_{\betab^+}$ \cite[Lemma 7.5]{GIT20} we obtain
\[
    u \cdot \mu_2 = p_\betab(u \cdot \mu_2) = p_\betab(\mu_2) = \mu_2,
\]
from which $u\in \Aut(\mu_2)$. Since $\mu_2 = b_2^{-1} \cdot \mu^\ggo$, we can write $u = b_2^{-1} a b_2$ with $a\in \Aut(\ggo)$. Hence, $b_1^{-1} b_2 = z b_2^{-1} a b_2$,
from which $b_1^{-1} = z b_2^{-1} a$. Inverting both sides yields the desired formula.
\end{proof}

\section{Homogeneous quotients of Einstein solvmanifolds are trivial}\label{app_quotients}

Recall that a Riemannian solvmanifold is a Riemannian manifold admitting a transitive solvable group of isometries.  The following key result is due to M.~Jablonski. It holds even more generally for Ricci solitons. However, as already noted in \cite{Jab15}, the proof simplifies significantly in the Einstein case. The argument below follows a sketch of proof indicated in \cite{Jab15}.

\begin{theorem}\cite{Jab15}\label{thm_Einsolvquot}
Let $(M^n,g)$ be an Einstein solvmanifold and $\Gamma < \Iso(M^n,g)$ a subgroup  acting properly discontinuously such that $M/\Gamma$ is homogeneous. Then, $\Gamma$ is trivial.
\end{theorem}


\begin{proof}
Since the universal cover of a solvmanifold is again a solvmanifold, acted transitively by the universal cover of the corresponding solvable Lie group, it is enough to prove the theorem when $M$ is simply-connected. Choosing the solvable group to act simply transitively \cite{GrdWls}, it follows that $M$ is diffeomorphic to $\RR^n$.

Recall that $\Iso(M^n, g)$ is linear by \cite{AC99} (see also the discussion after Theorem B in \cite{Heb}). This implies that any connected transitive group of isometries $\F \leq \Iso(M,g)$ has a solvable subgroup which is still transitive. Indeed, it suffices to show this for an effective presentation $M \simeq \F/\F_p$, where $\F_p$ is the compact isotropy at $p\in M$.  Since $M\simeq \RR^n$, $\F_p$ must be a maximal compact subgroup of $\F$. On the other hand, recall that from the Iwasawa and Levi decompositions, any linear Lie group decomposes as $\F = \K \Ss$ with $\K$ maximal compact and $\Ss$ solvable.  By the conjugacy of maximal compact subgroups of a Lie group, after changing $p$ we may assume that $\F_p = \K$, and it follows that $\Ss$ acts transitively and isometrically on $(M^n,g)$. 

Finally, let $M_2 := M/\Gamma$ be as in the statement, and let $\F_2$ be the identity component of its full isometry group. The universal cover $\tilde \F_2$ acts transitively and isometrically on $M$. By the argument in the previous paragraph, there is a transitive solvable Lie subgroup $\Ss \leq \tilde \F_2$, and it is clear that $\Ss_2 := \Ss / (\Gamma \cap \Ss)$ acts transitively on $M_2$. Hence, $M_2$ is also an Einstein solvmanifold. Since these are simply-connected by \cite{Jbl2015}, it follows that $\Gamma$ is trivial. 
\end{proof}

\end{appendix}

\bibliography{../../bib/ramlaf2}
\bibliographystyle{amsalpha}

\end{document}